\documentclass[11pt,reqno]{amsart}%
\usepackage{graphicx,adjustbox}
\usepackage{amsmath}
\usepackage{amsfonts}
\usepackage{amssymb,mathrsfs}
\usepackage{mathtools}
\usepackage{enumerate}
\usepackage{color}
\usepackage{cite}
\usepackage{url}
\usepackage{enumitem} 
\usepackage{nicematrix}
\usepackage{tikz-cd}
\usepackage{tikz}
\usepackage{xcolor}
\usetikzlibrary{shapes,arrows}
\usetikzlibrary{decorations.pathmorphing, decorations.text}
\usepackage{hyperref}
\usepackage[normalem]{ulem}

\usepackage{stmaryrd}
\usepackage{graphicx}
\usepackage{mleftright}
\usepackage{caption}
\usepackage{stmaryrd}
\usepackage{algorithm}
\usepackage{algpseudocode}
\usepackage{tabu}
\usepackage{enumerate}



\providecommand{\U}[1]{\protect\rule{.1in}{.1in}}
\newtheorem{theorem}{Theorem}[section]
\newtheorem{proposition}[theorem]{Proposition}
\newtheorem{proposition/definition}[theorem]{Proposition/Definition}
\newtheorem{lemma}[theorem]{Lemma}

\theoremstyle{definition}
\newtheorem{definition}[theorem]{Definition}

\newtheorem{example}[theorem]{Example}

\theoremstyle{remark}

\newtheorem{remark}[theorem]{Remark}

\newcommand{\red}[1]{{\color{red}#1}}

\newcommand{\Rmnum}[1]{\expandafter\@slowromancap\Romannumeral #1@} 
\newcommand{\tp}{{\scriptscriptstyle\mathsf{T}}}

\NiceMatrixOptions{xdots/horizontal-labels}
\let\O\undefined

\let\U\undefined

\newcommand{\norm}[1]{\left\lVert#1\right\rVert}
\newcommand{\inner}[1]{\left\langle#1\right\rangle}
\DeclareMathOperator{\O}{O}

\DeclareMathOperator{\U}{U}

\DeclareMathOperator{\cl}{cl}

\DeclareMathOperator{\dom}{dom}
\DeclareMathOperator{\rge}{rge}

\DeclareMathOperator{\Int}{Int}

\DeclareMathOperator{\gph}{gph}

\DeclareMathOperator{\tr}{tr}

\DeclareMathOperator{\im}{im}
\DeclareMathOperator{\rank}{rank}
\DeclareMathOperator{\diag}{diag}
\DeclareMathOperator{\argmax}{argmax}
\DeclareMathOperator{\argmin}{argmin}

\usepackage{geometry}
\tikzset{
	symbol/.style={
		draw=none,
		every to/.append style={
			edge node={node [sloped, allow upside down, auto=false]{$#1$}}}
	}
}
\begin{document}
	
	
	\title[Least Squares with Ball Constraints]{Upper H\"olderian with Explicit Exponent of Solution Mapping with Applications to Ball Constrained Least Squares Problems}
	
	\author[Y.~Wang]{Yu Wang}
	\address{College of Science, National University of Defense Technology, Changsha 410072, Hunan, China.}
	\email{wangyu25@nudt.edu.cn}
	
	\author[S.~Hu]{Shenglong Hu}
	\address{College of Science, National University of Defense Technology, Changsha 410072, Hunan, China.}
	\email{hushenglong@nudt.edu.cn}


		\begin{abstract}
			In this paper, we propose an extension of the well-known Robinson implicit function theorem for generalized equations from the upper Lipschitzian case to the upper H\"olderian case. Explicit exponents dependence between the generalized equation and its linearization is determined. Applications to ball constrained least squares problems, including linear least squares and separable nonlinear least squares, are studied. In particular, we establish that the solution mapping of ball constrained linear least squares under linear perturbation is locally upper H\"older continuous with exponent $\frac{1}{3}$, which is of independent interest. Ultilizing the upper H\"olderian version of the implicit function theorem, we show the local upper H\"olderian of the solution mapping of parametric ball constrained linear least squares.
		\end{abstract}
			
		%
		
		
		
		
		
		\maketitle
		
		
		\section{Introduction}\label{sec:Intro}
		
		A variety of problems can be recast as finding solutions of parameteric generalized equations of the following form 
		\begin{equation}\label{eq:sol_of_ge_eq}
			\mathcal{S}(\mathbf p):=\{\mathbf x\in\mathbb{R}^n\colon \mathbf 0\in F(\mathbf p,\mathbf x)+\mathcal T(\mathbf x)\},
		\end{equation}
		where $\mathbf p$ is a parameter, $F:\mathbb{R}^m\times\mathbb{R}^n\longrightarrow\mathbb{R}^n$ and $\mathcal{T}:\mathbb{R}^n\rightrightarrows\mathbb{R}^n$. 
		
		When $\mathcal{T}$ is single-valued, the generalized equation reduces to a system of classical parametric equations. In this case, monodromy of the solution $\mathcal{S}$ and $\mathcal{C}^1$ property of $\mathcal{S}$ can be obtained from the classical Implicit Function Theorem \cite{apostol1974mathematical,rudin1976principles,JRMunkres1991} when one assumes that both $\mathcal{T}$ and $F$ are $\mathcal{C}^1$. However, $\mathcal{S}$ is set-valued in general and this is common when $\mathcal T$ is set-valued. A particular interest on set-valued $\mathcal T$ is triggered by the demand from the optimization community. 
		The critical point system of the constrained optimization problem 
		\begin{equation}\label{eq:opt-const}		
			\begin{array}{rl}
				\min& f(\mathbf x)\\
				\text{s.t.}&\mathbf x\in C
			\end{array} 
		\end{equation} 
		is a prototypical example of \eqref{eq:sol_of_ge_eq} with $F=\nabla_{\mathbf{x}} f$ the gradient of $f$ and $\mathcal{T}=\mathcal{N}_{C}$ the normal cone of $C$. Usually, modeling parameters present in the objective function $f$, saying $\mathbf p\in\mathbb R^m$. As a result, the set of critical points $\mathcal{S}(\mathbf p)$, depending on $\mathbf p$, plays a pivot role in the senstivity and stability analysis of solution set of \eqref{eq:opt-const} \cite[Chapter~4]{Shapiro2000perturbationanalysis}, algorithmic design for solving \eqref{eq:opt-const} and its convergence rate analysis \cite[Chapter~6]{Rockafellar2014implicitfunctionsandsolutionmapping}, etc. 
		Sometimes, global properties of $\mathcal{S}(\mathbf p)$ may be wild to analysis and unnecessary, then local properties of $\mathcal{S}(\mathbf p)$ in terms of  
		\begin{equation}\label{eq:sigma}
			\Sigma(\mathbf p):=\mathcal{S}(\mathbf p)\cap X_0+\delta B 		
		\end{equation}
		for some (convex) compact $X_0$ and $\delta>0$ are studied.

		Besides the classical Implicit Function Theorem, for classical equations, various relaxations on the condition $\mathcal{C}^1$ were carried out in the literature, such as nonsmooth Lipschitz equations by Clarke \cite{clarke1976inverse} as well as \cite{krantz2013implicit,Rockafellar2014implicitfunctionsandsolutionmapping,kummer1989implicit,jittorntrum1978implicit,Robinson1994_Implicit} and references therein. For generalized equations, Robinson establishes the locally upper Lipschitz property of $\Sigma$ (cf.\ \eqref{eq:sigma}) under the assumption that the linearized generalized equation has an upper Lipschitz inverse at a given $\mathbf{p}_0$ and the upper Lipschitz continuity of $F$ in $\mathbf{p}$ uniformally over $\mathbf x$ \cite{Robinson1979_GenEqI}. Dontchev proves a similar result to Robinson's by assuming the existence of a strong approximation mapping $L_{\mathbf{p}_0}$ of $F(\mathbf{p}_0,\cdot)$ around a given $\mathbf{p}_0$  and upper Lipschitz  continuity of the solution mapping $\mathbf{y}\mapsto \{\mathbf{x}:L_{\mathbf{p}_0}(\mathbf{x})+\mathbf{y}\in \mathcal{T}(\mathbf{x})\}$ \cite{dontchev1993lipschitzian}. Mordukhovich  proposes several equivalent conditions using the coderivative for the Lipschitz continuity of $\mathcal{S}$ \cite{mordukhovich1994lipschitzian}.
		
		When $\mathcal{S}$ arises from the critical point system of an optimization problem, several works put their attention on them. 
		Klatte et al. give a sequence-based description for the strongly Lipschitz stability of $\mathcal{S}$, namely, the existence of a Lipschitz continuous single-valued localization of $\mathcal{S}$ \cite{klatte2005strong}. In a series of works, Izmailov and colleagues explore equivalent and sufficient conditions for the upper Lipschitz continuity, relying on noncriticality assumptions for multipliers, see \cite{izmailov2013note} and references herein. 
		Gfrerer and Klatte establish elaborate sufficient conditions for Lipschitz stablity under the supposition that $f$ possesses the restricted metric regularity on the points far from the Lipschitz distance to the optimal set \cite{gfrerer16lipschitz}.
		C\'anovas and Parra put forward several equivalent conditions for Lipschitz stability around some given $\mathbf{p}_0$ via metric regularization and Aubin property of the set-valued mapping $F(\mathbf{p}_0,\cdot)+\mathcal{T}(\cdot)$ \cite{canovas2025lipschitz}.
		Nghia derives the strongly Lipschitz stability of the solution mapping on the basis of quadratic growth condition of the relative approximation of the subgradient when \eqref{eq:sol_of_ge_eq} represents the KKT system of the regularized nonlinear least square problem \cite{Nghia2025geometrycharacterizationslipschitzstability}.
		Klatte shows the upper Lipschitz continuity of the solution mapping for quadratic programming with polyhedral constraints and fixed quadratic term \cite{klatte2026lipschitz}. 
		Cui et al. prove the Lipschitz stability of \eqref{eq:sol_of_ge_eq} for a special case when \eqref{eq:sol_of_ge_eq} represents the KKT system of the regularized nonlinear least square problem with the premise that the regularizer has $\mathcal{C}^2-$cone reducible conjugates \cite{cui2024lipschitzstabilityleastsquaresproblems}.
		
		A common consequence of the hypothese guaranteeing the above advances is that the mapping $\mathcal{S}$ from \eqref{eq:sol_of_ge_eq} has single-valued localization. However, this fails in several interesting situations \cite{Shapiro1997,JiangL22Holder}, in particular for the target problem considered in this paper (cf.\ \ref{eq:ball-nls}). A direct consequence is that Lipschitzian of $\mathcal{S}$, even local Lipschitzian, is out of our expection. Therefore, instead of Lipschitzian, H\"olderian of $\mathcal{S}$ is the main tone on characterizing the landscape of $\mathcal{S}$. 
		Gfrerer and Klatte also establish elaborate sufficient conditions for H\"older stablity under similar conditions as for the Lipschitz stability case \cite{gfrerer16lipschitz}. Interestingly, 
		Mordukhovich et al. confirm the $\frac{1}{2}$-H\"older stability  under the assumption of the uniform second order growth condition and the basic constraint qualification \cite{mordukhovich2014full}. Several other $\frac{1}{2}$-H\"older stability can be found in the monographs \cite{clarke1990optimization,Shapiro2000perturbationanalysis,Rockafellar2014implicitfunctionsandsolutionmapping,rockafellar2009variational,mordukhovich1996variationalI} and references therein.
		
		We see from the literature that there is short of $\alpha$-H\"olderian stability of $S$ with explicit $\alpha$. In this paper, we fill this gap by present an extension of  Robinson's theorem \cite{Robinson1979_GenEqI} to the H\"olderian setting (see Theorem~\ref{thm:holder-ge}). As an application, upper H\"olderian properties of the optimal solution set of the ball-constrained linear least squares (Ball-LS) problem is studied. It provides an example of $\mathcal{S}$ for which the results in the literature mentioned above fails to hold. The Ball-LS problem has the form
		\begin{align}\label{eq:ball-nls} 
			\begin{array}{rl}
				\min_{\mathbf{x}}& \frac{1}{2}\norm{\mathbf{Ax}-\mathbf{b}}^2\\
				\text{s.t.}& \norm{\mathbf{x}}^2\leq R^2,
			\end{array}\tag{$\mathscr{P}(\mathbf{A},\mathbf{b},R)$}
		\end{align}
		where $\mathbf{A}\in \mathbb{R}^{m\times n}$, $\mathbf{b}\in \mathbb{R}^m$, $R\geq 0$, and $\norm{\cdot}$ denotes the Frobenius norm. Classical LS is unconstrained. The ball constraint in \ref{eq:ball-nls} is a regularization for ill-conditioned $\mathbf A$. It also appears as subproblems in sequential optimization procedures \cite{he-26-tensor}.  
		LS problem is a basic block to numerical linear algebra \cite{Demmel1997appliednumericallnearalgebra}, optimization \cite{NY2018}, and statistics \cite{rao1995linear}. 
		Perturbation of the solutions of LS with respect to $\mathbf A$ and $\mathbf b$ has been a frontier topic \cite{golub1973differentiation,golub1976differentiation,Demmel1997appliednumericallnearalgebra, gulliksson2002perturbation, stewart1990matrix,cui2024lipschitzstabilityleastsquaresproblems,Nghia2025geometrycharacterizationslipschitzstability,Berk2023Lasso}. Nevertheless, most of them are for the case when $\mathbf A$ is nondegenerate or has local constant rank. 
		It is pointed out in \cite{stewart1990matrix} qualitatively that the solution set of LS can be quite ill-posed, see also \cite{HuWang2026variantionalgeometryofoptimizations}. On the other hand, the overall quantitative behavior of the optimal solution set of LS with respect to the parameters $\mathbf A$ and $\mathbf b$ has not been determined completely. In this paper, we present a complete characterization for the case Ball-LS. It is worth to mention that Jiang and Li \cite{JiangL22Holder} obtain H\"olderian with H\"older exponent $\frac{1}{4}$ for ball-constrained quadratic programming via a techique lemma proved in \cite{Wang01011994}. Our result can be regareded as an extension of their result to Ball-LS with a tighter H\"older exponent $\frac{1}{3}$ through a refined analysis (see Proposition~\ref{prop:holder-calm} and Theorem~\ref{thm:upper-holder-nls}).
		
		Specifically, we characterize the H\"olderian properties of the optimal solution set of problem \ref{eq:ball-nls} by treating $\mathbf{A}, \mathbf{b}, R$ as parameters via analyzing the solution mapping
		\begin{equation}\label{eq: solution mapping}
			\begin{array}{rl}
				\mathcal{S}:\mathbb{R}^{m\times n}\times \mathbb{R}^m \times \mathbb{R}_{+}&\rightrightarrows \mathbb{R}^n,\\
				(\mathbf{A},\mathbf{b},R)&\longmapsto\argmin_{x\in\mathcal{B}(R)} \frac{1}{2}\norm{\mathbf{Ax}-\mathbf{b}}^2,
			\end{array}
		\end{equation}
		where $\mathbb{R}_{+}:=\{R\in\mathbb{R}:R\geq 0\}$ and $\mathcal{B}(R):=\{\mathbf{x}\in\mathbb{R}^{n}:\norm{\mathbf{x}}\leq R\}$. Since $\mathcal{B}(R)$ is compact for all $R\geq 0$, the solution mapping $\mathcal{S}$ is nonempty over the entire parameter space $\mathbb{U}:=\mathbb{R}^{m\times n}\times \mathbb{R}^{m}\times \mathbb{R}_{+}$.  
		
		Our main contributions are summarized as follows
		\begin{enumerate}
			\item An extension of Robinson's implicit function theorem \cite[Theorem~1]{Robinson1979_GenEqI} to the H\"olderian case with explicit exponent characterization is established under parallel hypotheses as those in \cite{Robinson1979_GenEqI}. (cf.\ Theorem~\ref{thm:holder-ge}). This extension requires a nontrivial proof whose details are given in Section~\ref{sec:upper_Holder}. 
			\item As an application, the upper H\"olderian of the optimal solution set of Ball-LS with explicit exponent at least $\frac{1}{3}$ is established (cf.\ Theorem~\ref{thm:upper-holder-nls}). It is based on an explicit H\"olderian exponent characterization of Ball-LS under linear perturbations (cf.\ Proposition~\ref{prop:holder-calm}), which is of independent interest. 
		\end{enumerate}
		On top of these, another application of Theorem~\ref{thm:holder-ge} to H\"olderian  continuity for separable  nonlinear least square is given (cf.\  Proposition~\ref{prop:Holder_for_gen_nls}), and explicit formulae for the optimal value function and the optimal solution set of Ball-LS is presented (cf.\ Proposition~\ref{prop:optimal-value-solution-mapping} in Appendix~\ref{sec:basiclsp}).

		The rest of the paper is organized as follows. 	
		Some necessary preliminaries on variational analysis is collected in Section~\ref{sec:preliminary}. 	
		The extension from Robinson's theorem to the upper H\"olderian of \eqref{eq: solution mapping} with explicit H\"older exponent is given in Section~\ref{sec:upper_Holder}. The main result is presented in Theorem~\ref{thm:holder-ge}. Then, a quantitative result of the upper  H\"olderian of the solution mapping of~\eqref{eq: solution mapping} with global H\"older exponent $\frac{1}{3}$ will be presented in Section~\ref{sec:local_upper_Holder} (cf.\ \ Theorem~\ref{thm:upper-holder-nls}). An application for separable nonlinear least square (cf.\  Proposition~\ref{prop:Holder_for_gen_nls}) is also given herein. Some final remarks are put in Section~\ref{sec:con}. To keep the main theme of this paper, properties on Ball-LS~\ref{eq:ball-nls} is put in Appendix~\ref{sec:basiclsp}.

		\section{Preliminary}\label{sec:preliminary}
		In this section, necessary notations and basics on continuity for sets and set-valued mappings are presented. More details can be found from \cite{rockafellar2009variational,Shapiro2000perturbationanalysis,Rockafellar2014implicitfunctionsandsolutionmapping} and references therein.

		\subsection{Notations}\label{sec:notations}
		Notations such as 
		$\mathbf{A}^{\tp}$, $\rank(\mathbf{A})$, $\det(\mathbf{A})$, $\tr(\mathbf{A})$ for a matrix $\mathbf{A}$ have the standard meanings as in linear algebra, and other standards such as $\cl(\Omega)$ for the closure of $\Omega\subseteq \mathbb{V}$ with $\mathbb{V}$ a Euclidean space and $\Int(\Omega)$ for its interior as in topology. The following list of notations will be adopted throughout the paper unless otherwise specified. 
		\begin{itemize}
			\item $\mathbb{R}$ denotes the real line, 
			$\mathbb{R}_{+}:=\{x\in\mathbb{R}:x\geq 0\}$ the nonnegative part, and $\mathbb{R}_{++}:=\{x\in\mathbb{R}:x> 0\}$ the positive part.
			\item $\mathbb{R}^n$ and $\mathbb{R}^{m\times n}$ denote the spaces of $n$-dimensional real {column} vectors and $m\times n$ real matrices respectively. 
			\item  $\mathbf{I}_n$ denotes for the $n\times n$ identity matrix, and $\mathbf{0}$ for zero matrix or zero vector whose dimension is clear from the context. 
			\item $\im(\mathbf{A})$ and $\ker(\mathbf{A})$ for $\mathbf{A}\in\mathbb{R}^{m\times n}$ denote the image space and the kernel space of $\mathbf{A}$, respectively.
			\item $\inner{\cdot,\cdot}$ stands for the canonical Euclidean or Frobenius inner product on $\mathbb{R}^n$ or $\mathbb{R}^{m\times n}$, and $\norm{\cdot}$ for the induced norm. In this paper, all the metric and topological structures on $\mathbb{R}^n$ and $\mathbb{R}^{m\times n}$ are induced by $\norm{\cdot}$.
			\item $\norm{\mathbf{A}}_2:=\min_{\mathbf{x}\in\mathbb{R}^n,\mathbf{x}\neq \mathbf{0}}\frac{\norm{\mathbf{Ax}}}{\norm{\mathbf{x}}}$ denotes the spectral norm for $\mathbf{A}\in\mathbb{R}^{m\times n}$. 
			\item $\mathbf{A}^{\dag}$ represents the Moore-Penrose pesudo-inverse of $\mathbf{A}\in\mathbb{R}^{m\times n}$. 
			\item  $\O(n)$ denotes the orthogonal group on the Euclidean space $(\mathbb{R}^n,\inner{\cdot})$.
			\item For $R\geq 0$, $\mathcal{B}(R):=\{\mathbf{x}\in\mathbb{R}^n:\norm{\mathbf{x}}\leq R\}$ denotes the ball with radius $R$ inside $\mathbb R^n$. When $R=1$, $\mathcal{B}(R)$ simplifies to $\mathcal{B}$.
			\item For a differentiable mapping $g:\mathbb{R}^n\times \mathbb{R}^m\longrightarrow \mathbb{R}^P$ with value $g(\mathbf{x},\mathbf{u})$, $d_\mathbf{x}g$ denotes the partial differential  with respect to  $\mathbf{x}$, and $J_\mathbf{x}g$ the (partial) Jacobian matrix of $d_\mathbf{x}g$ under the canonical bases.
			\item For  $F:\mathbb{V}_1\longrightarrow\mathbb{V}_2$ with $\mathbb{V}_1,\mathbb{V}_2$ some Euclidean spaces, $\mathcal{U}\subseteq \mathbb{V}_1$ (not necessarily open) and $\mathbf{x}\in\mathcal{U}$ an accumulating point of $\mathcal{U}$, $F$ is continuous/differentiable at $\mathbf{x}$ {along} $\mathcal{U}$ if by restricting $F$ on $\mathcal{U}$, $F_{\mathcal{U}}$ is continuous/differentiable at $\mathbf{x}$ with respect to the induced subspace metric structure of $\mathcal{U}$.
		\end{itemize}   
		\subsection{Continuity}\label{sec:continuity}
		This section recalls necessary concepts and results on continuity from variational analysis. For more details, we refer to \cite{Rockafellar2014implicitfunctionsandsolutionmapping,rockafellar2009variational,Shapiro2000perturbationanalysis} and references herein. In this section, we assume that $\mathbb{V},\mathbb{V}_i (i=1,2)$ are Euclidean spaces.

		For a set-valued mapping $\mathcal{F}:\Omega\rightrightarrows \mathbb{V}_2$ where $\Omega\subseteq \mathbb{V}_1$, the {domain} $\dom(\mathcal{F})$ of $\mathcal{F}$ is defined by  $\dom(\mathcal{F}):=\{\mathbf{x}\in \Omega: \mathcal{F}(\mathbf{x})\neq \emptyset\}$ and the {range} $\rge(\mathcal{F})$ of $\mathcal{F}$ by $\rge(\mathcal{F}):=\{\mathbf{y}\in\mathbb{V}_2:\exists \mathbf{x}\in\Omega, \text{ s.t. }\mathbf{y}\in\mathcal{F}(\mathbf{x})\}.$ The {graph} $\gph(\mathcal{F})$ of $\mathcal{F}$ is defined by $\gph(\mathcal{F}):=\{(\mathbf{x},\mathbf{y})\in\Omega\times \mathbb{V}_2:\mathbf{y}\in\mathcal{F}(\mathbf{x})\}.$ The {inverse} $\mathcal{F}^{-1}$ of $\mathcal{F}$ is defined by $\mathcal{F}^{-1}(\mathbf{y}):=\{\mathbf{x}\in\Omega: \mathbf{y}\in \mathcal{F}(\mathbf{x})\}$. 
		
		\begin{definition}[Outer Semi-continuity]\label{defn: continuity of solution sets}
			Let $\mathcal{F}: \Omega\subseteq \mathbb{V}_1\rightrightarrows \mathbb{V}_2$ be a  set-valued mapping, $\Omega_0\subseteq\Omega$ and $\bar{\mathbf{x}}\in \cl(\Omega_0)$. The outer limit of $\mathcal{F}$ at $\bar{\mathbf{x}}$ relative to $\Omega_0$ is defined  as follows:
			\begin{equation}\label{eq: outer and inner limits} 
				\limsup_{\mathbf{x}\to \bar{\mathbf{x}},\mathbf{x}\in\Omega_0}\mathcal{F}(\mathbf{x}):=\{\mathbf{y}\in \mathbb{V}_2:\exists \mathbf{x}^{k}\in \Omega_0, \lim_{k\to\infty}\mathbf{x}^k=\bar{\mathbf{x}}, \exists \mathbf{y}^k\in \mathcal{F}(\mathbf{x}^k),\text{ s.t. } \lim_{k\to \infty} \mathbf{y}^k=\mathbf{y}  \}. 
			\end{equation}
			$\mathcal{F}$ is {outer semi-continuous (osc)} at $\bar{\mathbf{x}}$ relative to $\Omega_0$ if 
			$	\limsup_{\mathbf{x}\to \bar{\mathbf{x}},\mathbf{x}\in\Omega_0}\mathcal{F}(\mathbf{x})\subseteq \mathcal{F}(\bar{\mathbf{x}})$. 
			If $\mathcal{F}$ is osc at any point of $\Omega_1\subseteq \cl(\Omega_0)\bigcap \Omega$ relative to $\Omega_0$, then $\mathcal{F}$ is said to be osc on $\Omega_1$ relative to $\Omega_0$.
		\end{definition}
		The attributive pharse ``relative to $\Omega_0$'' is often omitted whenever $\Omega_0=\Omega$ or it is clear from the context. 
		\begin{definition}[Upper H\"olderian]\label{def:upper-hold}
			Let $\mathbb{X},\mathbb{Y}$ be two normed spaces and $B_{\mathbb{Y}}$ be the closed unit ball in $\mathbb{Y}$. Let $\Xi\subseteq \mathbb{X}$ with $\mathbf{p}_0\in\Xi$.	A set-valued mapping $\mathcal{F}:\mathbb{X}\rightrightarrows \mathbb{Y}$ is \textbf{upper H\"older continuous}  with  exponent $\alpha>0$ and modulus $\lambda>0$ at $\mathbf{p}_0$ with respect to $\Xi$, denoted by $\mathcal{F}\in\mathcal{H}_{\mathbf{p}_0,\Xi}(\alpha;\lambda)$, if for any $\mathbf{p}\in \Xi$, we have 
			\begin{equation}\label{eq: def of upper Holder continuity}
				\mathcal{F}(\mathbf{p})\subseteq \mathcal{F}(\mathbf{p}_0)+\lambda\norm{\mathbf{p}-\mathbf{p}_0}^\alpha B_{\mathbb{Y}}.
			\end{equation}
			We say that $\mathcal{F}$ is {locally upper H\"older continuous with  exponent $\alpha>0$ and modulus $\lambda>0$ at $\mathbf{p}_0$}, denoted by $\mathcal{F}\in\mathcal{H}_{\mathbf{p}_0}(\alpha;\lambda)$ if  $\mathcal{F}\in\mathcal{H}_{\mathbf{p}_0,\Xi}(\alpha;\lambda)$ for some neighborhood $\Xi$ of $\mathbf{p}_0$. 
		\end{definition}
		
		\begin{remark}\label{rmk: Holder continuity}
			\begin{enumerate}[label=(\arabic*)]
				\item In some literature such as \cite{Shapiro2000perturbationanalysis}, the upper H\"older continuous with exponent $\alpha$ for a solution mapping of a parametric optimization problem is also called {H\"older stability of degree $\alpha$}. It is also reasonable to call it \textit{H\"older calmness} with exponent $\alpha$ in view of calmness in the classical sense \cite{Rockafellar2014implicitfunctionsandsolutionmapping}. 
				\item If $\mathcal{F}$ is a single-valued mapping,  then upper H\"olderian $\mathcal{F}\in \mathcal{H}_{\mathbf{p}_0,\Xi}(\alpha,\lambda)$ extends the classical definition that $\mathcal{F}$ is H\"older continuous with exponent $\alpha>0$ and modulus $\lambda>0$ at $\mathbf{p}_0$ with respect to $\Xi$ (cf.\  \cite[Section~1.1]{Fiorenza2016Holder}).
				\item  $\mathcal{F}\in\mathcal{H}_{\mathbf{p}_0}(1,\lambda)$  means that $\mathcal{F}$ is locally upper Lipschitz continuous with modulus $\lambda$ (cf.\ ~\cite{Robinson1979_GenEqI}). 
				\item  $\mathcal{H}_{\mathbf{p}_0}(\alpha_1,\lambda)\subseteq\mathcal{H}_{\mathbf{p}_0}(\alpha_2,\lambda)$ whenever $\alpha_2\leq \alpha_1$. 
			\end{enumerate}
		\end{remark}
		
		Let $\mathscr{P}(\mathbb{Y})$ be the power set of $\mathbb{Y}$, and endow $\mathscr{P}(\mathbb{Y})$ with the {directed Hausdorff metric} 
		\begin{equation}\label{eq:direct-haus-def}
			\rho_{DH}(\Omega_1,\Omega_2):=\sup_{\mathbf{p}\in\Omega_1}\inf_{\mathbf{q}\in\Omega_2}\norm{\mathbf{p}-\mathbf{q}},\quad\ \forall \Omega_1,\Omega_2\in \mathscr{P}(\mathbb{Y}). 
		\end{equation}  
		\begin{proposition}\label{prop:equivalent}
			A set-valued mapping $\mathcal{F}:\mathbb{X}\rightrightarrows \mathbb{Y}$ is {upper H\"older continuous  with  exponent $\alpha>0$ and modulus $\lambda>0$ at $\mathbf{p}_0$ with respect to $\Xi$} (i.e., $\mathcal{F}\in\mathcal{H}_{\mathbf{p}_0,\Xi}(\alpha;\lambda)$) if and only if for any $\mathbf{p}\in\Xi$,
			\begin{equation}\label{eq:direct-haus}
				\rho_{DH}(\mathcal{F}(\mathbf{p}),\mathcal{F}(\mathbf{p}_0))\leq \lambda \norm{\mathbf{p}_0-\mathbf{p}}^\alpha,
			\end{equation}
			which is the upper H\"older continuous (or H\"older stability) for the single-valued map  $\mathbb{X}$ to $ \mathscr{P}(\mathbb{Y})$ induced by $\mathcal{F}:\mathbb{X}\rightrightarrows \mathbb{Y}$ with exponent $\alpha>0$ and modulus $\lambda>0$ at $\mathbf{p}_0$ with respect to $\Xi$. 
		\end{proposition}
		
		\begin{proof}{Proof}
			It follows directly from the definition of directed Hausdorff distance \eqref{eq:direct-haus-def} and upper H\"olderian \eqref{eq: def of upper Holder continuity}. 
			
		\end{proof}
		
		Notice that the directed Hausdorff metric is non-symmetric and thus the order of the sets matters. This partly justifies the nomenclature ``upper H\"older continuous". 		
		It follows from Proposition~\ref{prop:equivalent} that if $\mathcal{F}$ is upper H\"older continuous at $\mathbf{p}_0$, then $\mathcal{F}$ is continuous at $\mathbf p_0$ with respect to the commonly used double infimum quasi-metric defined by 
		\[
		d_{di}(\Omega_1,\Omega_2):=\inf_{\mathbf{p}\in \Omega_1,\mathbf{q}\in \Omega_2}\norm{\mathbf{p}-\mathbf{q}}.
		\]
		Under mild conditions, upper  H\"olderian implies outer semi-continuity, as the following lemma shows.
		\begin{lemma}\label{lemma: Holder continuity implies continuity}
			Let $\mathbb{V}_1,\mathbb{V}_2$ be Euclidean spaces and $\mathcal{F}:\Omega\rightrightarrows \mathbb{V}_2$ with $\Omega\subseteq\mathbb{V}_1$ be a set-valued mapping. Suppose that for $\mathbf{p}_0\in\Omega$, $\mathcal{F}(\mathbf{p}_0)$ is compact  and there exists a neighborhood $\mathcal{U}$ of $\mathbf{p}_0$ relative to $\Omega$  such that $\mathcal{F}\in \mathcal{H}_{\mathbf{x}_0,\mathcal{U}}(\alpha;\lambda)$ for some $\alpha,\lambda>0$, then $\mathcal{F}$ is outer semi-continuous at $\mathbf{p}_0$.
		\end{lemma}
		
		\begin{proof}{Proof}			
			For any $\mathbf{q}_0\in \limsup_{\mathbf{p}\to \mathbf{p}_0,\mathbf{p}\in\Omega}\mathcal{F}(\mathbf{p})$, there exist $\mathbf{p}^k\in\Omega$ and $\mathbf{q}^k\in\mathcal{F}(\mathbf{p}^k)$ such that $\lim_{k\to+\infty}\mathbf{p}^k=\mathbf{p}_0$ and $\lim_{k\to+\infty}\mathbf{q}^{k}=\mathbf{q}_0$. We can choose $K>0$ such that for every $k>K$, $\mathbf{p}^k\in\mathcal{U}$. Then by Definition~\ref{def:upper-hold} (cf.\ \eqref{eq: def of upper Holder continuity}), for every $k>K$,  there exists $\mathbf{y}^k\in\mathcal{F}(\mathbf{p}_0)$, such that 
			\begin{equation}\label{eq: Holder condition}
				\norm{\mathbf{q}^k-\mathbf{y}^k}\leq \lambda \norm{\mathbf{p}^k-\mathbf{p}_0}^{\alpha}.
			\end{equation} 
			Since $\mathcal{F}(\mathbf{p}_0)$ is compact, then by the Bolzano-Weierstrass theorem, there exists a subsequence  $\{\mathbf{y}^{k_l}\}$  such that $\mathbf{y}_0:=\lim_{l\to+\infty}\mathbf{y}^{k_l}$ exists and $\mathbf{y}_0\in \mathcal{F}(\mathbf{p}_0)$. Thus by replacing $\mathbf{q}^{k}, \mathbf{y}^{k}, \mathbf{p}^{k}$  by $\mathbf{q}^{k_l}, \mathbf{y}^{k_l}, \mathbf{p}^{k_l}$, respectively  in~\eqref{eq: Holder condition} and letting $l\to+\infty$, we get that $\norm{\mathbf{q}_0-\mathbf{y}_0}=0$. Thus $\mathbf{q}_0=\mathbf{y}_0\in \mathcal{F}(\mathbf{p}_0)$. Therefore, $\mathcal{F}$ is outer semi-continuous at $\mathbf{p}_0$ by Definition~\ref{defn: continuity of solution sets}.
		\end{proof}

		\section{Upper H\"olderian with Explicit Exponent of Generalized Equations}\label{sec:upper_Holder}
		In this section, we present a quantitative extension of the Robinson implicit function theorem for generalized equations \cite[Theorem~1]{Robinson1979_GenEqI} to the upper H\"olderian case. The theorem (cf.\ Theorem~\ref{thm:holder-ge}) is stated in Section~\ref{sec:implicit-theorem}, followed by discussions. The proof for Theorem~\ref{thm:holder-ge} is nontrivial, which needs several lemmas prepared in Section~\ref{sec:lemmas}, and is given in Section~\ref{sec:proof-ge}. Sections~\ref{sec:lemmas} and \ref{sec:proof-ge} can be safely skipped without compromising the understanding the results for Ball-LS in Section~\ref{sec:local_upper_Holder}. 
		
		\subsection{Upper H\"olderian Implicit Function Theorem}\label{sec:implicit-theorem} 
		
		Given subsets $\mathcal{K}\subset\Omega\subseteq \mathbb{R}^n$ and  $\Xi\subseteq\mathbb{R}^m$, we consider the following parameteric generalized equation 
		\begin{equation}\label{eq:parametric_generalized_equation}
			\mathbf{0}\in F(\mathbf{p},\mathbf{x})+\mathcal{T}(\mathbf{x}),\ \mathbf x\in\mathcal K,
		\end{equation}
		where $\mathcal{T}:\mathcal{K}\rightrightarrows \mathbb{R}^n$ is a set-valued mapping and $F: \Xi\times \Omega\longrightarrow \mathbb{R}^n$.  Denote the solution mapping of~\eqref{eq:parametric_generalized_equation} as 
		\begin{equation}\label{eq:solution_mapping_for_parametric_generalized_equations}
			\begin{array}{rl}
				\mathcal{S}_{GE}: \Xi&\rightrightarrows \mathcal{K},\\
				\mathbf{p}&\mapsto \{\mathbf{x}\in\mathcal{K}: \mathbf{0}\in F(\mathbf{p},\mathbf{x})+\mathcal{T}(\mathbf{x})\}.
			\end{array}
		\end{equation}
		
		If $F$ is
		differentiable with respect to the variables in $\Omega$, we denote 
		$F_2(\mathbf{p},\mathbf{x}):=J_{\mathbf{x}}F(\mathbf{p},\mathbf{x})$ and for $(\mathbf{p}_0,\mathbf{x}_0)\in \Xi\times\mathcal{K}$, the ``partial linearization" of the set-valued mapping $\mathbf{x}\mapsto F(\mathbf{p}_0,\mathbf{x})+\mathcal{T}(\mathbf{x})$:
		\begin{equation}\label{eq:linearized_generalized_equation}
			\begin{array}{rl}
				\mathcal{L}_{\mathbf{x}_0}: \mathcal{K}&\rightrightarrows \mathbb{R}^n,\\
				\mathbf{x}&\mapsto F(\mathbf{p}_0,\mathbf{x}_0)+F_2(\mathbf{p}_0,\mathbf{x}_0)(\mathbf{x}-\mathbf{x}_0)+\mathcal{T}(\mathbf{x}).
			\end{array}
		\end{equation} 
		Inside \eqref{eq:linearized_generalized_equation},  $LF_{\mathbf{x}_0}(\mathbf{x}):=F(\mathbf{p}_0,\mathbf{x}_0)+F_2(\mathbf{p}_0,\mathbf{x}_0)(\mathbf{x}-\mathbf{x}_0)$ is the “linearization" of $F$ at $(\mathbf{p}_0,\mathbf{x}_0)$ with respect to $\mathbf{x}$.
		\begin{theorem}[Implicit Function Theorem with Explicit H\"older Exponent]\label{thm:holder-ge}
			Assuming notations as above, suppose that $\mathcal{K}$ is closed and $\Omega$ is open, $\mathcal{T}$ is outer semi-continuous,  $F$ is continuous on $\Xi\times \Omega$ and differentiable with respect to $\mathbf{x}\in\Omega$, and $F_2(\mathbf{p},\mathbf{x})$ is continuous on $\Xi\times \Omega$. Let $\mathbf{p}_0\in \Xi$. 
			Suppose that there exist a compact convex set $\mathcal{X}_0\neq \emptyset$, and constants $\alpha,\gamma,\eta>0$ and $\lambda>1$ such that $\mathcal{X}_{\theta}:=\mathcal{X}_0+\theta \mathcal{B}\subseteq \mathcal{K}$ for $\theta\in [0,\gamma]$, and for each $\mathbf{x}_0\in \mathcal{X}_0$ with $\mathcal{IL}_{\mathbf{x}_0,\theta}(\mathbf{y}):=\mathcal{L}_{\mathbf{x}_0}^{-1}(\mathbf{y})\bigcap \mathcal{X}_\theta$ for $\theta\in[0,\gamma]$, there hold   
			\begin{enumerate}[label=(A\arabic*),ref=(A\arabic*)]
				\item $\mathcal{IL}_{\mathbf{x}_0,\gamma}(\mathbf{0})=\mathcal{X}_0;$\label{item:ass_1}
				\item $\mathcal{IL}_{\mathbf{x}_0,\gamma}\in \mathcal{H}_{\mathbf{0},\eta \mathcal{B}}(\alpha,\lambda);$\label{item:ass_2}
				\item for each $\mathbf{y}\in \eta\mathcal{B}$, $\mathcal{IL}_{\mathbf{x}_0,\gamma}(\mathbf{y})$ is convex and nonempty;\label{item:ass_3}
				\item if $\alpha<1$, then there exists $\tilde{\lambda}\leq\frac{1}{2}\lambda^{-\frac{1}{\alpha}}$ such that for any $\mathbf{x}_0\in\mathcal{X}_0$, $F_2(\mathbf{p}_0,\cdot)\in\mathcal{H}_{\mathbf{x}_0,\mathcal{X}_\gamma}(\frac{1}{\alpha}-1;\tilde{\lambda})$.\label{item:ass_4}
			\end{enumerate}     
			Then there exists $\bar \delta\in (0,\gamma]$ such that for any $\delta\in(0,\bar\delta]$, there exists a neighborhood $\mathcal{W}$ of $\mathbf{p}_0$ such that for
			\begin{equation}\label{eq: def of the inverse mapping}
				\mathcal{A}_\delta(\mathbf{p}):=\mathcal{S}_{GE}(\mathbf{p})\cap\mathcal{X}_\delta,
			\end{equation}
			we have
			\begin{enumerate}[label=(R\arabic*),ref=(R\arabic*)]
				\item For any $\mathbf{p}\in\mathcal{W}$, $\mathcal{A}_\delta(\mathbf{p})\neq \emptyset$ and $\mathcal{A}_\delta$ is outer semi-continuous on $\mathcal{W}$ (relative to $\mathcal{W}$).\label{item:res_1}
				\item $\mathcal{A}_\delta(\mathbf{p}_0)=\mathcal{X}_0$.\label{item:res_2}
				\item Let $\mu_\theta(\mathbf{p}):=\max_{\mathbf{x}\in\mathcal{X}_\theta}\{\norm{F(\mathbf{p},\mathbf{x})-F(\mathbf{p}_0,\mathbf{x})}\}$ for any $\theta\in [0,\gamma]$.
				\begin{enumerate}
					\item If $0<\alpha<1$, then 
					for any $\mathbf{p}\in \mathcal{W}$, it holds
					\begin{equation}\label{eq:Holder continuity,1}
						\mathcal{A}_\delta(\mathbf{p})\subseteq \mathcal{A}_\delta(\mathbf{p}_0)+(1-\frac{1}{2^\alpha})^{-1}\lambda\mu_\delta(\mathbf{p})^\alpha\mathcal{B}.
					\end{equation} 
					\item If $\alpha\geq 1$, then for any $\mathbf{p}\in \mathcal{W}$, it holds
					\begin{equation}\label{eq:Holder continuity,2-0}
						\mathcal{A}_{\delta}(\mathbf{p})\subseteq \mathcal{A}_{\delta}(\mathbf{p}_0)+2\lambda\mu_0(\mathbf{p})^\alpha\mathcal{B};
					\end{equation}
				\end{enumerate}  \label{item:res_3}
			\end{enumerate}
			Moreover, if in addition, $\mu_{\delta}$ is locally upper H\"older continuous at $\mathbf{p}_0$ with exponent $\beta$, then $\mathcal{A}_{\delta}\in \mathcal{H}_{\mathbf{p}_0}(\alpha\beta,M)$ for some $M>0$.
		\end{theorem}
		\begin{remark}\label{rmk:inverse_mapping_theorem_Holder_case}
			\begin{enumerate}[label=(\arabic*)] 
				\item The requirement of assumption \ref{item:ass_4} for $\alpha<1$ in Theorem~\ref{thm:holder-ge} is as follows: all kinds of inverse mapping theorem and fixed point theorem require some type of contraction property, explicitly or implicitly. That is, the mapping in consideration must map a  neighborhood of the given point into a smaller neighborhood contained in it. 
				Note that when $\alpha<1$,  $a<a^\alpha$ for $a<1$. Thus the standard implication as in \cite{Robinson1979_GenEqI} from $\rho_{DH}(\mathcal{IL}_{\mathbf{x}_0,\gamma}(\mathbf{y}_1),\mathcal{IL}_{\mathbf{x}_0,\gamma}(\mathbf{y}_2))\leq\lambda\norm{\mathbf{y}_1-\mathbf{y}_2}^\alpha$ to that $\rho_{DH}(\mathcal{IL}_{\mathbf{x}_0,\gamma}(\mathbf{y}_1),\mathcal{IL}_{\mathbf{x}_0,\gamma}(\mathbf{y}_2))$ is bounded by a constant times $\norm{\mathbf{y}_1-\mathbf{y}_2}$ when $\norm{\mathbf{y}_1-\mathbf{y}_2}$ is small cannot be adopted.  The assumption \ref{item:ass_4} is for the purpose to offset the speed of the local growth rate when we compose $F_2$ with $\mathcal{IL}_{\mathbf{x}_0,\gamma}$. 
				\item Theorem~\ref{thm:holder-ge} with $\alpha=1$ reduces to the classical Lipschitz case proved by Robinson in \cite[Theorem~1]{Robinson1979_GenEqI}.
				\item A sufficient condition for locally upper H\"older continuous of $\mu_\delta$ at $\mathbf{p}_0$ with exponent $\beta$ and modulus $\nu$ is that $F$ is locally upper H\"older continuous with respect to $\mathbf{p}$ at $\mathbf{p}_0$ with exponent $\beta$ and modulus $\nu$ uniformly on $\mathcal{X}_\delta$.
				\item If $F$ is differentiable with respect to $\mathbf{p}$ and the partial differential $d_{\mathbf{p}}F$ is continuous, then we can choose a convex compact neighborhood $\mathcal{W}_0$ of $\mathbf{p}_0$ so that for any $\mathbf{x}\in\mathcal{X}_0$ and for any $\mathbf{p}\in \mathcal{W}_0$, we have  $\mathbf{p}(\tau):=\tau\mathbf{p}+(1-\tau)\mathbf{p}_0\in \mathcal{W}_0$  and  
				\begin{equation*}
					\begin{aligned}
						\norm{F(\mathbf{p},\mathbf{x})-F(\mathbf{p}_0,\mathbf{x})}&=\norm{\int_0^1 d_{\mathbf{p}}F(\mathbf{p}(\tau),\mathbf{x})[\mathbf{p}-\mathbf{p}_0]d\tau}\\&\leq \norm{\mathbf{p}-\mathbf{p}_0}\cdot \max_{\mathbf{x}\in\mathcal{X}_0}\norm{\int_0^1d_{\mathbf{p}}F(\mathbf{p}(\tau),\mathbf{x})d\tau}_2. 
					\end{aligned}
				\end{equation*}
				Since $d_{\mathbf{p}}F$ is continuous and $\mathcal{X}_0$ is compact, $\max_{\mathbf{x}\in\mathcal{X}}\norm{\int_0^1d_{\mathbf{p}}F(\mathbf{p}(\tau),\mathbf{x})d\tau}_2<+\infty$. Thus $F$ is locally upper Lipschitz continuous and hence upper H\"older continuous with respect to $\mathbf{p}$ at $\mathbf{p}_0$ uniformly on $\mathcal{X}_0$.
				\item The modulus $2\lambda$ in \eqref{eq:Holder continuity,2-0} can be improved to as close as to $\lambda$. From \textbf{Part II.3(b)} of the proof for Theorem~\ref{thm:holder-ge}, it follows that for each $\epsilon>0$, there exist $\delta_\epsilon\leq \gamma$ and a neighborhood $\mathcal{W}_{\epsilon}$ of $\mathbf{p}_0$ such that for any $\mathbf{p}\in \mathcal{W}_{\epsilon}$, it holds
				\begin{equation}\label{eq:Holder continuity,2}
					\mathcal{A}_{\delta_\epsilon}(\mathbf{p})\subseteq \mathcal{A}_{\delta_\epsilon}(\mathbf{p}_0)+(\lambda+\epsilon)\mu_0(\mathbf{p})^\alpha\mathcal{B}.
				\end{equation}
			\end{enumerate}
		\end{remark}
		
		\subsection{Technical Lemmas}\label{sec:lemmas}
		Technical lemmas for the proof of Theorem~\ref{thm:holder-ge} are collected in this subsection. 
		\begin{lemma}[Estimation for Linearization]\label{lem:estimation} 
			Assume the same notations and assumptions in Theorem~\ref{thm:holder-ge}, and let $\mathcal{P}:\mathbb{R}^n\longrightarrow\mathcal{X}_0$ be the projection onto $\mathcal{X}_0$ with respect to the Euclidean norm. Let $\mathcal{Z}\subseteq \Omega$ be a convex set with $\mathcal{X}_0\subseteq\mathcal{Z}$  and $M$ be a constant such that for any $\mathbf{x}\in \mathcal{Z}$, $\norm{F_{2}(\mathbf{p}_0,\mathbf{x})-F_2(\mathbf{p}_0,\mathcal{P}(\mathbf{x}))}_2\leq M$,   then we have that for any $\mathbf{x}\in\mathcal{Z}$,
			\begin{equation}\label{eq: an estimation for linearization}
				\norm{LF_{\mathcal{P}(\mathbf{x})}(\mathbf{x})-F(\mathbf{p}_0,\mathbf{x})}\leq M\norm{\mathbf{x}-\mathcal{P}(\mathbf{x})}.
			\end{equation}
		\end{lemma}
		\begin{proof}{Proof}
			Since $\mathcal{X}_0$ is compact and convex, $\mathcal{P}$ is well-defined and it is well known that $\mathcal{P}$ is Lipschitz continuous with modulus $1$ (cf.\  \cite[Proposition~4.8]{bauschke2011convex}). For any $\mathbf{x}\in\mathcal{Z}$, by the convexity of $\mathcal{Z}$, $\mathbf{x}_{\tau}:=\tau\mathbf{x}+(1-\tau)\mathcal{P}(\mathbf{x})\in\mathcal{Z}$ with $\tau\in[0,1]$, and by \cite[Proposition~3.21]{bauschke2011convex} that $\mathcal{P}(\mathbf{x}_\tau)=\mathcal{P}(\mathbf{x})$ for any $\tau\in [0,1]$.
			
			Let $g(\tau):=F(\mathbf{p}_0,\mathbf{x}_\tau)-LF_{\mathcal{P}(\mathbf{x})}(\mathbf{x}_\tau)$ for $\tau\in [0,1]$. Clearly, $g(0)=\mathbf{0}$ and $g$ is continuously differentiable with $g'(\tau)=(F_2(\mathbf{p}_0,\mathbf{x}_\tau)-F_2(\mathbf{p}_0,\mathcal{P}(\mathbf{x})))(\mathbf{x}-\mathcal{P}(\mathbf{x}))$. By the Newton-Leibniz formula and the assumptions, we get that
			$$\begin{aligned}
				\norm{LF_{\mathcal{P}(\mathbf{x})}(\mathbf{x})-F(\mathbf{p}_0,\mathbf{x})}&=\norm{g(1)}=\norm{g(1)-g(0)}=\norm{\int_{0}^1g'(\tau)d\tau}\leq \int_{0}^1\norm{g'(\tau)}d\tau\\&\leq \int_{0}^1\norm{F_2(\mathbf{p}_0,\mathbf{x}_\tau)-F_2(\mathbf{p}_0,\mathcal{P}(\mathbf{x}))}_2\norm{\mathbf{x}-\mathcal{P}(\mathbf{x})}d\tau\leq M \norm{\mathbf{x}-\mathcal{P}(\mathbf{x})}.
			\end{aligned}$$
			Consequently, \eqref{eq: an estimation for linearization} holds.
		\end{proof}
		
		\begin{lemma}\label{lemma: closeness of a set defined by a set-valued map}
			Let $\mathcal{K}_1\subseteq\mathbb{R}^n$, $\mathcal{K}_2\subseteq\mathbb{R}^m$, $\mathcal{F}:\mathcal{K}_1\rightrightarrows \mathbb{R}^p$ be outer semi-continuous on $\mathcal{K}_1$ and $F:\mathcal{K}_1\times \mathcal{K}_2\longrightarrow \mathbb{R}^p$ be continuous. Define $\mathcal{Y}:=\{(\mathbf{x},\mathbf{y})\in \mathcal{K}_1\times \mathcal{K}_2: F(\mathbf{x},\mathbf{y})\in \mathcal{F}(\mathbf{x})\}$. Then $\mathcal{Y}$ is a closed set relative to $\mathcal{K}_1\times\mathcal{K}_2$.
		\end{lemma}
		\begin{proof}{Proof}
			For any convergent sequence $\{(\mathbf{x}^k,\mathbf{y}^k)\}_{k\geq 1}\subseteq \mathcal{Y}$ with $\lim_{k\to+\infty}(\mathbf{x}^k,\mathbf{y}^k)=(\mathbf{x}^0,\mathbf{y}^0)\in\mathcal{K}_1\times \mathcal{K}_2$, we have $F(\mathbf{x}^0,\mathbf{y}^0)$ and $\mathcal{F}(\mathbf{x}^0)$ are well-defined. Denote $\mathbf{z}^k=F(\mathbf{x}^k,\mathbf{y}^k)$ for $k\geq 0$.  By definition, $\mathbf{z}^k\in \mathcal{F}(\mathbf{x}^k)$ for $k\geq 1$, that is, $(\mathbf{x}^k,\mathbf{z}^k)\in\gph(\mathcal{F})$. Since $F$ is continuous, then 
			\[
			\mathbf{z}^0:=F(\mathbf{x}^0,\mathbf{y}^0)=\lim_{k\to+\infty}F(\mathbf{x}^k,\mathbf{y}^k)=\lim_{k\to+\infty}\mathbf{z}^k.
			\]
			Since $\mathcal{F}$ is outer semi-continuous on $\mathcal{K}_1$, then by \cite[Theorem~3B.2]{Rockafellar2014implicitfunctionsandsolutionmapping}, we get that $\gph(\mathcal{F})$ is closed. Therefore $(\mathbf{x}^0,\mathbf{z}^0)\in\gph(\mathcal{F})$. This implies that $\mathbf{z}^0=F(\mathbf{x}^0,\mathbf{y}^0)\in\mathcal{F}(\mathbf{x}^0)$. Thus $(\mathbf{x}^0,\mathbf{y}^0)\in\mathcal{Y}$ and hence $\mathcal{Y}$ is closed relative to $\mathcal{K}_1\times\mathcal{K}_2$.
		\end{proof}
		\begin{lemma}\label{lem:an_estimation}
			Let $\alpha>0$ and $x,y\geq0$.  
			\begin{enumerate}
				\item If $\alpha\in (0,1)$, then \begin{equation}\label{eq:ineq_for_alpha_less_1}
					(x+y)^\alpha \leq x^\alpha +y^\alpha.
				\end{equation}
				\item If $\alpha\geq 1$, then 
				\begin{equation}\label{eq:ineq_for_alpha_less_2}
					(x+y)^\alpha \geq x^\alpha +y^\alpha.
				\end{equation}
			\end{enumerate}
		\end{lemma}
		\begin{proof}{Proof}
			If $\alpha=1$, then the assertion is clear. In the following, we give a proof for the case $0<\alpha<1$. A similar proof works for the case $\alpha> 1$. 
			
			Let $y\geq 0$ and $H_y(x):=(x+y)^\alpha - x^\alpha -y^\alpha$ for $x\geq 0$. Then $H_y$ is continuous on $[0,+\infty)$ and smooth on $(0,\infty)$ with  $H_y'(x)=\alpha (x+y)^{\alpha-1}-\alpha x^{\alpha-1}$. Note that the function $x\mapsto x^{\alpha-1}$ is decreasing. Since $y\geq 0$, then $H_y'(x)\leq 0$ for any $x>0$. Thus $H_y$ is decreasing on $(0,+\infty)$. Therefore $H_y(x)\leq H_y(0)=0$ for any $x\geq 0$. Thus~\eqref{eq:ineq_for_alpha_less_1} holds. 
			
		\end{proof}
		\begin{lemma}\label{lem:inequality}
			Let $\alpha\geq 1,\lambda,a>0$ with $ \lambda a(a+1)^{\alpha-1}<1$. Let $y\geq 0$ be a parameter, define
			\begin{equation}\label{eq: key inequality}
				\mathcal{W}_y:=\{x\geq 0:x\leq \lambda(ax+y)^\alpha\},
			\end{equation} 
			and $C=C(a,\lambda,\alpha):=\frac{\lambda (1+a)^{\alpha-1}}{1-a\lambda(a+1)^{\alpha-1}}>0$. Then for any $y\in \mathbb R_+$, $W_y\bigcap [0,1]\subseteq [0,C y^{\alpha}]$.  
		\end{lemma}
		\begin{proof}{Proof}
			If $\alpha=1$, then the assertion is clear. In the following, we assume that $\alpha> 1$.
			
			Note that $x\mapsto x^\alpha$ is a convex function on $\mathbb R_+$ when $\alpha>1$. Then for any $x,y\geq 0$, $(\frac{ax+y}{a+1})^\alpha\leq \frac{a}{a+1} x^\alpha +\frac{1}{a+1}y^\alpha$, which implies that $(ax+y)^\alpha\leq a(a+1)^{\alpha-1} x^\alpha+(a+1)^{\alpha-1}y^\alpha$. Also note that when $0\leq x\leq 1$, $x^\alpha\leq x$.  Therefore for any $x\in W_y\cap [0,1]$, 
			\[
			x\leq \lambda(ax+y)^\alpha\leq \lambda a(a+1)^{\alpha-1} x^\alpha+\lambda(a+1)^{\alpha-1}y^\alpha\leq \lambda a(a+1)^{\alpha-1} x+\lambda(a+1)^{\alpha-1}y^\alpha.
			\]
			This indicates that  
			$x\leq C y^\alpha$ with $C=C(a,\lambda,\alpha)$. The conclusion then follows. 
		\end{proof}
		
		\begin{lemma}\label{lem:inequality-para}
			Let $\alpha\geq 1,\lambda>0$ and $C_{\lambda,\alpha}(a):=C(a,\lambda,\alpha)=\frac{\lambda (1+a)^{\alpha-1}}{1-a\lambda(a+1)^{\alpha-1}}$ for $a\in S_{\lambda,\alpha}:=\{a> 0: a(a+1)^{\alpha-1}\lambda<1 \}$ . For any $\epsilon>0$, there exists $a\in S_{\lambda,\alpha}$ such that $C_{\lambda,\alpha}(a)\leq\lambda+\epsilon$. 
		\end{lemma}
		\begin{proof}{Proof}
			The map $h:a\mapsto a\lambda(a+1)^{\alpha-1}$ is continuous on $[0,+\infty)$ and $h(0)=0$. Then $C_{\lambda,\alpha}$ is well-defined in a neighborhood of $0$ with respect to $[0,+\infty)$ and continuous at $0$ with $C_{\lambda,\alpha}(0)=\lambda$, by continuity, for $\epsilon>0$, there exists $a\in S_{\lambda,\alpha}$  such that $C_{\lambda,\alpha}(a)\leq \lambda+\epsilon$. 
		\end{proof}
		
		\subsection{Proof of Theorem~\ref{thm:holder-ge}}\label{sec:proof-ge}
		
		\begin{proof}{Proof}
			Let $\bar\eta\in(0,\eta]$. 
			Let $\mathcal{P}:\mathbb{R}^n\longrightarrow \mathcal{X}_0$ be the projection onto $\mathcal{X}_0$ with respect to the Euclidean norm. Since $\mathcal{X}_0$ is compact and convex, $\mathcal{P}$ is well-defined and single-valued (cf.\ \ \cite[Proposition~4.8]{bauschke2011convex}). In the following, we will use $\mathcal{L}_{\mathcal{P}(\mathbf{\cdot})}^{-1}$ to approximate $(F(\mathbf{p},\cdot)+\mathcal{T})^{-1}$, in a similar spirit of the proof for \cite[Theorem~1]{Robinson1979_GenEqI}. 
			
			In the sequel, we split the proof into three parts and each part is subdivided into several pieces whenever necessary.
			
			\textbf{Part I. Nonemptiness of $\mathcal A_{\delta}$.}
			
			Let $\delta\geq 0$, and define a set-valued map  $H_{\mathbf{p},\delta}:\mathcal{X}_\delta\rightrightarrows \mathbb{R}^n$ by
			\begin{equation*}
				H_{\mathbf{p},\delta}(\mathbf{x}):=\mathcal{IL}_{\mathcal{P}(\mathbf{x}),\gamma}(LF_{\mathcal{P}(\mathbf{x})}(\mathbf{x})-F(\mathbf{p},\mathbf{x})).
			\end{equation*}  
			In the following, we show that there exists a $\bar\delta\in(0,\gamma]$ such that for any $\delta\in(0,\bar\delta]$, there exists a companion neighborhood $\mathcal{U}_{\delta}$ of $\mathbf{p}_0$ such that for any $\mathbf{p}\in\mathcal{U}_\delta$, there hold
			\begin{enumerate}[label=(C\arabic*),ref=(C\arabic*)]
				\item $H_{\mathbf{p},\delta}(\mathbf{x})$ is nonempty, closed and convex for any $\mathbf{x}\in \mathcal{X}_\delta$.\label{item:con_a}
				\item $H_{\mathbf{p},\delta}(\mathbf{x})\subseteq\mathcal{X}_\delta$ for any $\mathbf{x}\in \mathcal{X}_\delta$. \label{item:con_b}
				\item $H_{\mathbf{p},\delta}$ is outer semi-continuous. \label{item:con_c}\footnote{The definition of ``upper semi-continuous" in \cite{kakutani1941generalization} is just the same as ``outer semi-continuous" defined in our paper (cf.\ Definition~\ref{defn: continuity of solution sets}).}
			\end{enumerate}
			Consequently, by the Kakutani fixed point theorem \cite[Theorem~1]{kakutani1941generalization}, for any $\mathbf{p}\in\mathcal{U}_\delta$, there exists $\mathbf{x}_\mathbf{p}\in \mathcal{X}_\delta$ such that $\mathbf{x}_\mathbf{p}\in H_{\mathbf{p},\delta}(\mathbf{x}_\mathbf{p})$. Therefore, $LF_{\mathcal{P}(\mathbf{x}_\mathbf{p})}(\mathbf{x}_\mathbf{p})-F(\mathbf{p},\mathbf{x}_\mathbf{p})\in LF_{\mathcal{P}(\mathbf{x}_\mathbf{p})}(\mathbf{x}_\mathbf{p})+\mathcal{T}(\mathbf{x}_\mathbf{p})$. Hence 
			\begin{equation}\label{eq:con_of_fixed_point}
				\mathbf{0}\in F(\mathbf{p},\mathbf{x}_\mathbf{p})+\mathcal{T}(\mathbf{x}_\mathbf{p})\ \ \text{i.e., } \mathbf{x}_\mathbf{p}\in\mathcal{A}_{\delta}(\mathbf{p}).
			\end{equation}

			\textbf{Part I.1: Verification of \ref{item:con_a}.}

			\textbf{Step 1. Estimation on $\norm{F(\mathbf{p}_0,\mathbf{x})-LF_{\mathcal{P}(\mathbf{x})}(\mathbf{x})}$.}
			
			Define for $\delta\in[0,\gamma]$ that 
			\begin{equation}\label{eq:cdelta}
				c(\delta):=\max_{\mathbf{x}\in \mathcal{X}_\delta}\norm{F_2(\mathbf{p}_0,\mathbf{x})-F_2({\mathbf{p}_0,\mathcal{P}(\mathbf{x})})}_2.  	
			\end{equation}
			Since $F_2$ is continuous,  $\mathcal{P}$ is Lipschitz continuous and  $\mathcal{X}_{\delta}$ is compact, then $c(\delta)$ is well-defined and $c(0)=0$. In the following, we show that $c$ is continuous at $0$. 
			
			Let $\epsilon>0$ be given and arbitrary. Note that $\norm{F_2(\mathbf{p}_0,\cdot)-F_2(\mathbf{p}_0,\mathcal{P}(\cdot))}_2$ is continuous and for any $\mathbf{x}\in\mathcal{X}_0$, $\norm{F_2(\mathbf{p}_0,\mathbf{x})-F_2(\mathbf{p}_0,\mathcal{P}(\mathbf{x}))}_2=0$. Then for every $\mathbf{x}\in\mathcal{X}_0$, there exists  $\delta_{\mathbf{x}}>0$ such that for any $\bar{\mathbf{x}}\in\mathbb{R}^n$ with $\norm{\mathbf{x}-\bar{\mathbf{x}}}<\delta_{\mathbf{x}}$,  $\norm{F_2(\mathbf{p}_0,\bar{\mathbf{x}})-F_2(\mathbf{p}_0,\mathcal{P}(\bar{\mathbf{x}}))}_2<\epsilon$. Since $\mathcal{X}_0$ is compact and $\bigcup_{\mathbf{x}\in\mathcal{X}_0}\mathcal{N}_{\mathbf{x}}\supseteq \mathcal{X}_0$ with $\mathcal{N}_{\mathbf{x}}:=\{\bar{\mathbf{x}}\in\mathbb{R}^n:\norm{\mathbf{x}-\bar{\mathbf{x}}}<\frac{1}{3}\delta_{\mathbf{x}}\}$, then there exist $\mathbf{x}_1,\dots,\mathbf{x}_K\in \mathcal{X}_0$ such that $\bigcup_{k=1}^K\mathcal{N}_{\mathbf{x}_k}\supseteq \mathcal{X}_0$. Let $\tilde\delta_\epsilon:=\frac{1}{3}\min\{\delta_{\mathbf{x}_1},\dots,\delta_{\mathbf{x}_K}\}>0$. Then by definition of $\mathcal{X}_{\tilde\delta_{\epsilon}}$, we get that for any $\mathbf{x}\in \mathcal{X}_{\tilde\delta_{\epsilon}}$, $\norm{\mathbf{x}-\mathcal{P}(\mathbf{x})}\leq\tilde\delta_{\epsilon}$ and $\mathcal{P}(\mathbf{x})\in \mathcal{N}_{\mathbf{x}_k}$ for some $k\in\{1,\dots,K\}$. Therefore 
			\[
			\norm{\mathbf{x}-\mathbf{x}_k}\leq \norm{\mathcal{P}(\mathbf{x})-\mathbf{x}_k}+\norm{\mathbf{x}-\mathcal{P}(\mathbf{x})} <\frac{1}{3}\delta_{\mathbf{x}_k}+\tilde\delta_\epsilon<\delta_{\mathbf{x}_k}.
			\]
			Thus, $\norm{F_2(\mathbf{p}_0,\mathbf{x})-F_2(\mathbf{p}_0,\mathcal{P}(\mathbf{x}))}_2<\epsilon$, and hence $c(\tilde\delta_\epsilon)\leq \epsilon$. By the monotonicity of $c$, the continuity of $c$ at $0$ follows.

			Hence, we can choose $\delta_0\in(0,\gamma]$ such that for any $\tau\in [0,\delta_0]$, 
			\begin{equation}\label{eq: estimation of c}
				\tau<\frac{1}{2},\quad \tau c(\tau)< \frac{\bar\eta}{2},\quad \lambda^{\frac{1}{\alpha}}c(\tau)<\frac{1}{2}.
			\end{equation} 
			
			Moreover, when $0<\alpha<1$, by assumption \ref{item:ass_4}, we get that for any $\delta\in(0,\gamma]$
			\begin{equation}\label{eq:est_if_alp_leq_1}
				c(\delta)\leq \max_{\mathbf{x}\in\mathcal{X}_\delta}\{\tilde{\lambda}\norm{\mathbf{x}-\mathcal{P}(\mathbf{x})}^{\frac{1}{\alpha}-1}\}\leq \tilde{\lambda}\delta^{\frac{1}{\alpha}-1}\leq \frac{1}{2}\lambda^{-\frac{1}{\alpha}}\delta^{\frac{1}{\alpha}-1}.
			\end{equation} 
			
			\textbf{Step 2: Estimation on $\norm{F(\mathbf{p},\mathbf{x})-F(\mathbf{p}_0,\mathbf{x})}$.}	
			
			Given $\delta\in [0,\gamma]$, define  
			\begin{equation}\label{eq:mudelta}
				\mu_{\delta}(\mathbf{p}):=\max_{\mathbf{x}\in\mathcal{X}_\delta}\norm{F(\mathbf{p},\mathbf{x})-F(\mathbf{p}_0,\mathbf{x})}.
			\end{equation} 
			Since $\mathcal{X}_\delta$ is fixed and compact, and $F$ is continuous, $\mu_{\delta}$ is well-defined for all $\mathbf{p}\in\Xi$ and $\mu_{\delta}(\mathbf{p}_0)=0$. Moreover, by ~\cite[Theorem~3B.5]{Rockafellar2014implicitfunctionsandsolutionmapping}, $\mu_\delta$ is continuous at $\mathbf{p}_0$. Thus there exists a neighborhood $\mathcal{U}_{\delta}$ of $\mathbf{p}_0$ such that for each $\mathbf{p}\in\mathcal{U}_\delta$,
			\begin{equation}\label{eq: estimation of mu}
				\mu_{\delta}(\mathbf{p})<\frac{\bar\eta}{2},\ \text{and } \lambda^{\frac{1}{\alpha}}\mu_{\delta}(\mathbf{p})\leq \frac{1}{2}\delta^{\frac{1}{\alpha}}.
			\end{equation} 
			Thus, for any $\mathbf{p}\in \mathcal{U}_{\delta}$ and $\mathbf{x}\in\mathcal{X}_\delta$, 
			\begin{equation}\label{eq: estimation 2}
				\norm{F(\mathbf{p},\mathbf{x})-F(\mathbf{p}_0,\mathbf{x})}\leq \mu_{\delta}(\mathbf{p})<\frac{\bar\eta}{2}.
			\end{equation}
			Note that when $\delta\leq\delta_0$, we have that
			\begin{align}\label{eq: estimation 1}
				\norm{LF_{\mathcal{P}(\mathbf{x})}(\mathbf{x})-F(\mathbf{p},\mathbf{x})}&\leq \norm{F(\mathbf{p},\mathbf{x})-F(\mathbf{p}_0,\mathbf{x})}+\norm{F(\mathbf{p}_0,\mathbf{x})-LF_{\mathcal{P}(\mathbf{x})}(\mathbf{x})}\nonumber\\
				&\leq\mu_\delta(\mathbf p)+c(\delta)\norm{\mathbf{x}-\mathcal{P}(\mathbf{x})}\nonumber\\
				&\leq\mu_\delta(\mathbf p)+ c(\delta)\delta\\
				&<\bar\eta,\nonumber
			\end{align}
			where the second inequality follows from \eqref{eq: estimation 2} and Lemma~\ref{lem:estimation}, the third from the definition of $\mathcal{X}_\delta$, and the last from \eqref{eq: estimation of c} and \eqref{eq: estimation 2}. 
			
			In summary, let $\delta\leq\bar\delta:=\frac{\delta_0}{2}$ and $\mathbf{p}\in\mathcal{U}_\delta$, $LF_{\mathcal{P}(\mathbf{x})}(\mathbf{x})-F(\mathbf{p},\mathbf{x})\in \bar\eta \mathcal{B}\subseteq\eta \mathcal{B}$ for any $\mathbf x\in \mathcal X_\delta$. This, together with $\mathcal P(\mathbf x)\in \mathcal X_0$, implies from assumptions~\ref{item:ass_3} that $H_{\mathbf{p},\delta}(\mathbf{x})=\mathcal{IL}_{\mathcal{P}(\mathbf{x}),\gamma}(LF_{\mathcal{P}(\mathbf{x})}(\mathbf{x})-F(\mathbf{p},\mathbf{x}))$ is nonempty and convex.

			By the fact that $\mathcal{T}$ is outer semi-continuous and $LF_{\mathbf{x}_0}$ is a linear map, then clearly by the definition of outer semi-continuity (cf.\ \ Definition~\ref{defn: continuity of solution sets}) that $\mathcal{L}_{\mathbf{x}_0}=LF_{\mathbf{x}_0}+\mathcal{T}$ is outer semi-continuous. By~\cite[Theore 3B.2]{Rockafellar2014implicitfunctionsandsolutionmapping}, we have that $\mathcal{L}_{\mathbf{x}_0}^{-1}(\mathbf{y})$ is closed in $\mathcal{K}$, and hence $\mathcal{L}_{\mathbf{x}_0}^{-1}(\mathbf{y})$ is closed in $\mathbb{R}^n$ as $\mathcal{K}$ itself is closed.  This, together with the compactness of $\mathcal X_\gamma$ and the definition of $\mathcal{IL}_{\mathbf{x}_0,\gamma}$, it follows that $\mathcal{IL}_{\mathbf{x}_0,\gamma}(\mathbf y)$ is closed for any $\mathbf x_0\in\mathcal X_0$ and $\mathbf y\in\eta\mathcal B$.   
			
			Thus, $H_{\mathbf{p},\delta}(\mathbf{x})$ is nonempty, convex and closed for any $\mathbf{x}\in\mathcal{X}_{\delta}$. 
			
			\textbf{Part I.2: Verification of \ref{item:con_b}.}
			
			Recall the directed Hausdorff metric $\rho_{DH}$ given by \eqref{eq:direct-haus-def} for two sets in $\mathbb{R}^n$. By assumptions \ref{item:ass_1} and \ref{item:ass_2}, and Proposition~\ref{prop:equivalent}, we get that for  $\delta\leq\bar\delta=\frac{\delta_0}{2}$, $\mathbf{p}\in\mathcal{U}_\delta$, and $\mathbf{x}\in\mathcal{X}_\delta$, 
			\begin{equation}\label{eq: estimation 4}
				\begin{aligned}
					\rho_{DH}(H_{\mathbf{p},\delta}(\mathbf{x}),\mathcal{X}_0)&=\rho_{DH}(\mathcal{IL}_{\mathcal{P}(\mathbf{x}),\gamma}(LF_{\mathcal{P}(\mathbf{x})}(\mathbf{x})-F(\mathbf{p},\mathbf{x})),\mathcal{IL}_{\mathcal{P}(\mathbf{x}),\gamma}(\mathbf{0}))\\
					&\leq \lambda \norm{LF_{\mathcal{P}(\mathbf{x})}(\mathbf{x})-F(\mathbf{p},\mathbf{x})}^\alpha
					\leq \lambda (\mu_{\delta}(\mathbf{p})+c(\delta)\delta)^\alpha\leq \delta,
				\end{aligned}
			\end{equation}
			where the second inequality follows from \eqref{eq: estimation 1}, and the last from \eqref{eq: estimation of c}, \eqref{eq:est_if_alp_leq_1} and~\eqref{eq: estimation of mu}. Thus $H_{\mathbf{p},\delta}(\mathbf{x})\subseteq \mathcal{X}_\delta$ for any $\mathbf{x}\in \mathcal{X}_\delta$ since $\mathcal{X}_\delta$ is closed. Consequently, condition \ref{item:con_b} holds.
			
			\textbf{Part I.3: Verification of \ref{item:con_c}.}
			
			We have that 
			\begin{equation*}
				\begin{aligned}
					\gph(H_{\mathbf{p},\delta})&=\{(\mathbf{x},\mathbf{y})\in \mathcal{X}_\delta\times\mathcal{X}_\gamma:LF_{\mathcal{P}(\mathbf{x})}(\mathbf{x})-F(\mathbf{p},\mathbf{x})\in LF_{\mathcal{P}(\mathbf{x})}(\mathbf{y})+\mathcal{T}(\mathbf{y})\}\\
					&=\{(\mathbf{x},\mathbf{y})\in \mathcal{X}_\delta\times\mathcal{X}_\gamma: F_2(\mathbf{p}_0,\mathcal{P}(\mathbf{x}))(\mathbf{x}-\mathbf{y})-F(\mathbf{p},\mathbf{x})\in \mathcal{T}(\mathbf{y})\}.
				\end{aligned}
			\end{equation*}
			By continuity of $F,F_2,\mathcal{P}$, outer semi-continuity of $\mathcal{T}$, closedness of $\mathcal{X}_\delta$ and $\mathcal{X}_\gamma$, and Lemma~\ref{lemma: closeness of a set defined by a set-valued map}, we get that $\gph(H_{\mathbf{p},\delta})$ is closed. Thus by \cite[Theore~5.7]{rockafellar2009variational} and \cite[Theorem~3B.2]{Rockafellar2014implicitfunctionsandsolutionmapping}, condition \ref{item:con_c} holds.
			
			\textbf{Part II. Confirmation of Results \ref{item:res_1}, \ref{item:res_2} and \ref{item:res_3}.} 
			In the following, we verify Results \ref{item:res_1}, \ref{item:res_2} and \ref{item:res_3} item by item.
			
			\textbf{Part II.1.  Result \ref{item:res_1}.}
			
			Clearly, from \eqref{eq:con_of_fixed_point}, $\mathcal{A}_\delta(\mathbf{p})\neq \emptyset$ for any $\mathbf{p}\in\mathcal{U}_\delta$. 
			Note that $\gph(\mathcal{A}_\delta)=\{(\mathbf{p},\mathbf{x})\in\mathcal{U}_\delta\times \mathcal{X}_\delta: -F(\mathbf{p},\mathbf{x})\in\mathcal{T}(\mathbf{x})\}$, then by continuity of $F$, the outer semi-continuity of $\mathcal{T}$ and Lemma~\ref{lemma: closeness of a set defined by a set-valued map}, $\gph(\mathcal{A}_\delta)$ is closed in $\mathcal{U}_\delta\times\mathcal{X}_\delta$. Thus $\mathcal{A}_\delta$ is outer semi-continuous on $\mathcal{U}_\delta$ (relative to $\mathcal{U}_\delta$) by  \cite[Theorem~3B.2]{Rockafellar2014implicitfunctionsandsolutionmapping}. So we can choose $\mathcal{W}=\mathcal{U}_\delta$  and Result \ref{item:res_1} in the assertion of Theorem~\ref{thm:holder-ge} holds.
			
			\textbf{Part II.2.  Result  \ref{item:res_2}.}
			
			In the following, we show that $\mathcal{A}_\delta(\mathbf{p}_0)=\mathcal{X}_0$. On the one hand, for any $\mathbf{x}_0\in\mathcal{X}_0$, by assumption \ref{item:ass_1}, we get that $\mathbf{0}\in LF_{\mathbf{x}_0}(\mathbf{x}_0)+\mathcal{T}(\mathbf{x}_0)=F(\mathbf{p}_0,\mathbf{x}_0)+\mathcal{T}(\mathbf{x}_0)$, so $\mathbf{x}_0\in \mathcal{A}_{\delta}(\mathbf{p}_0)$ and thus $\mathcal{X}_0\subseteq\mathcal{A}_\delta(\mathbf{p}_0)$. On the other hand, if $\mathbf{x}\in \mathcal{A}_{\delta}(\mathbf{p}_0)$, then $\mathbf{x}\in\mathcal{X}_\delta$ and $\mathbf{0}\in F(\mathbf{p}_0,\mathbf{x})+\mathcal{T}(\mathbf{x})$. Therefore, $LF_{\mathcal{P}(\mathbf{x})}(\mathbf x)-F(\mathbf{p}_0,\mathbf{x})\in LF_{\mathcal{P}(\mathbf{x})}(\mathbf{x})+\mathcal{T}(\mathbf{x})$, i.e., $\mathbf{x}\in H_{\mathbf{p}_0,\delta}(\mathbf{x})$.  Because of $\mathbf{x}\in\mathcal{X}_\delta$, by assumption~\ref{item:ass_2}, Proposition~\ref{prop:equivalent} and the proof of Lemma~\ref{lem:estimation}, we get that with $\mathbf{x}_\tau:=\tau\mathbf{x}+(1-\tau)\mathcal{P}(\mathbf{x})\in\mathcal{X}_\delta$ for $\tau\in[0,1]$,
			\begin{equation}\label{eq: estimation 5}
				\begin{aligned}		\norm{\mathbf{x}-\mathcal{P}(\mathbf{x})}&=\inf_{\mathbf{z}\in\mathcal{X}_0}\norm{\mathbf{x}-\mathbf{z}}\leq \rho_{DH}(H_{\mathbf{p}_0,\delta}(\mathbf{x}),\mathcal{X}_0)\leq \lambda\norm{LF_{\mathcal{P}(\mathbf{x})}(\mathbf{x})-F(\mathbf{p}_0,\mathbf{x})}^\alpha\\
					&\leq \lambda (\int_{0}^1\norm{F_2(\mathbf{p}_0,\mathbf{x}_\tau)-F_2(\mathbf{p}_0,\mathcal{\mathbf{x}})}_2\norm{\mathbf{x}-\mathcal{P}(\mathbf{x})}d\tau)^\alpha.
				\end{aligned}
			\end{equation}
			\begin{enumerate}[label=(\arabic*)] 
				\item 	 If $\alpha\geq 1$,  we get from \eqref{eq: estimation 5} that 
				\begin{equation*}
					\norm{\mathbf{x}-\mathcal{P}(\mathbf{x})}\leq \lambda(c(\delta)	\norm{\mathbf{x}-\mathcal{P}(\mathbf{x})})^\alpha\leq (\frac{1}{2})^\alpha	\norm{\mathbf{x}-\mathcal{P}(\mathbf{x})}^\alpha \leq \frac{1}{2}	\norm{\mathbf{x}-\mathcal{P}(\mathbf{x})},
				\end{equation*} 
				where the first inequality follows from the definition of $c(\delta)$, the second from \eqref{eq: estimation of c}, and the last from $	\norm{\mathbf{x}-\mathcal{P}(\mathbf{x})}\leq \delta<\frac{1}{2}$. Therefore, we get that $	\norm{\mathbf{x}-\mathcal{P}(\mathbf{x})}=0$, and hence $\mathbf{x}=\mathcal{P}(\mathbf{x})\in \mathcal{X}_0$.
				\item If $\alpha\in (0,1)$,  we get from \eqref{eq: estimation 5} that 
				\begin{equation*}
					\norm{\mathbf{x}-\mathcal{P}(\mathbf{x})}\leq \lambda \norm{\mathbf{x}-\mathcal{P}(\mathbf{x})}^\alpha(\int_{0}^1\tilde{\lambda}\norm{\mathbf{x}_\tau-\mathbf{x}}^{\frac{1}{\alpha}-1}d\tau)^\alpha\leq \lambda\tilde{\lambda}^\alpha \norm{\mathbf{x}-\mathcal{P}(\mathbf{x})}\leq  
					\frac{1}{2^\alpha}	\norm{\mathbf{x}-\mathcal{P}(\mathbf{x})},
				\end{equation*}
				where the first inequality follows from assumption~\ref{item:ass_4},  the second from $\norm{\mathbf{x}_\tau-\mathbf{x}}\leq \norm{\mathcal{P}(\mathbf{x})-\mathbf{x}}$, and the last also from assumption~\ref{item:ass_4}. Thus $\norm{\mathbf{x}-\mathcal{P}(\mathbf{x})}=0$, and hence $\mathbf{x}=\mathcal{P}(\mathbf{x})\in \mathcal{X}_0$.
			\end{enumerate}
			Therefore we get that $\mathbf{x}\in\mathcal{X}_0$ and hence $\mathcal{A}_{\delta}(\mathbf{p}_0)\subseteq\mathcal{X}_0$ for any $\alpha> 0$.  So $\mathcal{A}_{\delta}(\mathbf{p}_0)=\mathcal{X}_0$, and hence Result \ref{item:res_2} holds.
			
			\textbf{Part II.3. Result \ref{item:res_3}.}
			
			It remains to prove the conclusion \ref{item:res_3}.  This needs to estimate $\rho_{DH}(\mathcal{A}_\delta(\mathbf{p}),\mathcal{A}_\delta(\mathbf{p}_0))$ by Proposition~\ref{prop:equivalent}. By \ref{item:res_2}, we have that $\mathcal{A}_\delta(\mathbf{p}_0)=\mathcal{X}_0$. Moreover, for $\mathbf{x}\in\mathcal{X}_\delta\subseteq\mathcal{X}_\gamma$, we have 
			\[\begin{aligned}
				\mathbf{x}\in \mathcal{A}_{\delta}(\mathbf{p})&\Longleftrightarrow \mathbf{0}\in F(\mathbf{p},\mathbf{x})+\mathcal{T}(\mathbf{x})\\&\Longleftrightarrow LF_{\mathcal{P}(\mathbf{x})}(\mathbf{x})-F(\mathbf{x},\mathbf{p})\in LF_{\mathcal{P}(\mathbf{x})}(\mathbf{x})+\mathcal{T}(\mathbf{x})\\&\Longleftrightarrow \mathbf{x}\in H_{\mathbf{p},\delta}(\mathbf{x}).
			\end{aligned}
			\]
			Thus, the estimation~\eqref{eq: estimation 4} is a start point for the estimation $\rho_{DH}(\mathcal{A}_\delta(\mathbf{p}),\mathcal{A}_{\delta}(\mathbf{p}_0))$. 
			In the following, we split the proof into two parts, corresponding to $\alpha\in(0,1)$ and $\alpha\geq 1$ respectively. The ingredient is the reference on either $\mu_0$ or $\mu_\delta$.

			\textbf{Part II.3(a). The case $\alpha\in(0,1)$.}	
			
			In this case, for any $\mathbf{x}\in\mathcal{A}_\delta(\mathbf{p})$, we get that  
			\begin{align}\label{eq:estimation_alpha_less_1}
				\rho_{DH}(\{\mathbf{x}\},\mathcal{A}_\delta(\mathbf{p}_0)) 
				&\leq \rho_{DH}(H_{\mathbf{p},\delta}(\mathbf{x}),\mathcal{X}_0)\nonumber\\
				&\leq \lambda\norm{LF_{\mathcal{P}(\mathbf{x})}(\mathbf{x})-F(\mathbf{p},\mathbf{x})}^\alpha\nonumber\\
				&\leq \lambda (\norm{F_2(\mathbf{p}_0,\xi)(\mathbf{x}-\mathcal{P}(\mathbf{x}))-F_2(\mathbf{p}_0,\mathcal{P}(\mathbf{x}))(\mathbf{x}-\mathcal{P}(\mathbf{x}))}+\mu_{\delta}(\mathbf{p}))^\alpha\nonumber\\
				&\leq \lambda (\norm{F_2(\mathbf{p}_0,\xi)-F_2(\mathbf{p}_0,\mathcal{P}(\mathbf{x}))}_2\norm{\mathbf x-\mathcal{P}(\mathbf{x})}+\mu_{\delta}(\mathbf{p}))^\alpha\nonumber\\
				&\leq \lambda (\tilde{\lambda}\norm{\mathbf{x}-\mathcal{P}(\mathbf{x})}^{\frac{1}{\alpha}}+\mu_{\delta}(\mathbf{p}))^\alpha\nonumber\\
				&\leq \lambda \tilde{\lambda}^\alpha \norm{\mathbf{x}-\mathcal{P}(\mathbf{x})}+\lambda(\mu_{\delta}(\mathbf{p}))^\alpha\nonumber\\
				&\leq \frac{1}{2^\alpha}\norm{\mathbf{x}-\mathcal{P}(\mathbf{x})}+\lambda(\mu_{\delta}(\mathbf{p}))^\alpha,
			\end{align} 
			where $\xi$ is some point lying on the line segment between $\mathbf{x}$ and $\mathcal{P}(\mathbf{x})$,  the first inequality follows from the definition of $\rho_{DH}$ (cf.\ \eqref{eq:direct-haus-def}), the second from assumption \ref{item:ass_2} and Proposition~\ref{prop:equivalent}, the third from \eqref{eq: estimation 1},  the mean value theorem and the definition of $\mu_\delta$ (cf.\ \eqref{eq:mudelta}), the fifth from assumption \ref{item:ass_4} and the fact that $\norm{\xi-\mathcal{P}(\mathbf{x})}\leq \norm{\mathbf{x}-\mathcal{P}(\mathbf{x})}$,  the penultimate from Lemma~\ref{lem:an_estimation}, and the last from also assumption \ref{item:ass_4}. Note that by definition of $\mathcal{P}$ and conclusion \ref{item:res_2}, we get that $\rho_{DH}(\{\mathbf{x}\},\mathcal{A}_\delta(\mathbf{p}_0))=\rho_{DH}(\{\mathbf{x}\},\mathcal{X}_0)=\norm{\mathbf{x}-\mathcal{P}(\mathbf{x})}$. Thus by \eqref{eq:estimation_alpha_less_1}, we get that~\eqref{eq:Holder continuity,1} holds.

			\textbf{Part II.3(b). The case $\alpha\geq 1$.}	
			
			In this case, for any $\mathbf{x}\in\mathcal{A}_\delta(\mathbf{p})$, we have that
			\begin{equation}\label{eq: estimation 6}
				\begin{aligned}
					\rho_{DH}(\{\mathbf{x}\},\mathcal{A}_\delta(\mathbf{p}_0))&\leq \rho_{DH}(H_{\mathbf{p},\delta}(\mathbf{x}),\mathcal{X}_0)\leq \lambda\norm{LF_{\mathcal{P}(\mathbf{x})}(\mathbf{x})-F(\mathbf{p},\mathbf{x})}^\alpha\\&\leq \lambda (\norm{g(\mathbf{x})-g(\mathcal{P}(\mathbf{x}))}+\norm{g(\mathcal{P}(\mathbf{x}))}+\norm{F(\mathbf{p}_0,\mathbf{x})-LF_{\mathcal{P}(\mathbf{x})}(\mathbf{x})})^\alpha,
				\end{aligned}
			\end{equation}
			where $g(\mathbf{x}):=F(\mathbf{p},\mathbf{x})-F(\mathbf{p}_0,\mathbf{x})$. As $\mathcal{P}(\mathbf{x})\in \mathcal{X}_0$, then by definition, 
			\begin{equation}\label{eq: estimation 7}
				\norm{g(\mathcal{P}(\mathbf{x}))}\leq \mu_0(\mathbf{p}).
			\end{equation}
			Denote $\mathbf{x}_\tau:=\tau \mathbf{x}+(1-\tau)\mathcal{P}(\mathbf{x})$ and $G(\tau):=g(\mathbf{x}_\tau)$, then for $\delta\in[0,\delta_0]$ and $\mathbf{x}\in \mathcal{X}_\delta$, we have 
			\begin{equation}\label{eq: estimation 8}
				\begin{aligned}
					\norm{g(\mathbf{x})-g(\mathcal{P}(\mathbf{x}))}&=\norm{G(1)-G(0)}=\norm{\int_0^1G'(\tau)d\tau}\\
					&=\norm{\int_0^1F_2(\mathbf{p},\mathbf{x}_\tau)(\mathbf{x}-\mathcal{P}(\mathbf{x}))-F_2(\mathbf{p}_0,\mathbf{x}_\tau)(\mathbf{x}-\mathcal{P}(\mathbf{x}))d\tau}\\
					&\leq \int_0^1 \norm{F_2(\mathbf{p},\mathbf{x}_\tau)-F_2(\mathbf{p}_0,\mathbf{x}_\tau)}_2\norm{\mathbf{x}-\mathcal{P}(\mathbf{x})}d\tau\leq \phi_{\delta}(\mathbf{p})\norm{\mathbf{x}-\mathcal{P}(\mathbf{x})},
				\end{aligned}
			\end{equation}
			where
			\begin{equation*}
				\phi_{\delta}(\mathbf{p}):=\max_{\mathbf{x}\in\mathcal{X}_\delta}\norm{F_2(\mathbf{p},\mathbf{x})-F_2(\mathbf{p}_0,\mathbf{x})}_2.
			\end{equation*}
			Similar to the discussion on $\mu_{\delta}(\mathbf{p})$ (cf.\ \eqref{eq:mudelta}), $\phi_{\delta}(\mathbf{p})$ is well-defined and continuous at $\mathbf{p}_0$ with $\phi_{\delta}(\mathbf{p}_0)=0$. Combining ~\eqref{eq:cdelta} with Lemma~\ref{lem:estimation},~\eqref{eq: estimation 6},~\eqref{eq: estimation 7} and~\eqref{eq: estimation 8} with the fact that $\rho_{DH}(\{\mathbf{x}\},\mathcal{A}_\delta(\mathbf{p}_0))=\norm{\mathbf{x}-\mathcal{P}(\mathbf{x})}$, we get that 
			\begin{equation}\label{eq: estimation 9}
				\norm{\mathbf{x}-\mathcal{P}(\mathbf{x})}\leq \lambda(\phi_{\delta}(\mathbf{p})\norm{\mathbf{x}-\mathcal{P}(\mathbf{x})}+\mu_0(\mathbf{p})+c(\delta)\norm{\mathbf{x}-\mathcal{P}(\mathbf{x})})^\alpha.
			\end{equation}
			
			For any $\epsilon>0$,  by Lemma~\ref{lem:inequality-para}, we can choose $a>0$, such that $C_{\lambda,\alpha}(a)\leq\lambda+\epsilon$ with $C_{\lambda,\alpha}$ defined in Lemma~\ref{lem:inequality-para}.
			
			Since $c(\delta)$ is continuous at $0$ and $c(0)=0$ (cf.\ the disucssion following \eqref{eq:cdelta}), then there exists $\delta_\epsilon\leq \gamma$ such that $c(\delta_\epsilon')<\frac{1}{2}a$ for all $\delta_\epsilon'\leq\delta_\epsilon$.   Since $\phi_{\delta_\epsilon}(\mathbf{p})$ is continuous at $\mathbf p_0$ with $\phi_{\delta_\epsilon}(\mathbf{p}_0)=0$, then there exists a neighborhood $\mathcal{W}_{\epsilon}$ such that for any $\mathbf{p}\in\mathcal{W}_{\epsilon}$ it holds $\phi_{\delta_\epsilon}(\mathbf{p})<\frac{1}{2}a$.
			So by~\eqref{eq: estimation 9}, we get that for any $\mathbf{p}\in \mathcal{W}_{\epsilon}$ and $\mathbf{x}\in\mathcal{A}_{\delta_\epsilon}(\mathbf{p})$, we have  
			\begin{equation}\label{eq: estimation 10}
				\norm{\mathbf{x}-\mathcal{P}(\mathbf{x})}\leq \lambda (a\norm{\mathbf{x}-\mathcal{P}(\mathbf{x})}+\mu_0(\mathbf{p}))^\alpha.
			\end{equation}
			Thus by Lemma~\ref{lem:inequality}, we get that 
			\begin{equation}\label{eq:norm_estimate}
				\norm{\mathbf{x}-\mathcal{P}(\mathbf{x})}\leq C_{\lambda,\alpha}(a)(\mu_0(\mathbf{p}))^\alpha\leq(\lambda+\epsilon)(\mu_0(\mathbf{p}))^\alpha.
			\end{equation}
			This implies $\mathcal{A}_{\delta_\epsilon}(\mathbf{p})\subseteq\mathcal{A}_{\delta_\epsilon}(\mathbf{p}_0)+(\lambda+\epsilon)(\mu_0(\mathbf{p}))^\alpha\mathcal{B}$, which is exactly \eqref{eq:Holder continuity,2}. 
			
			If we fix $\epsilon=\lambda$, then $\delta_\lambda$ corresponding to $\epsilon=\lambda$ is determined by the preceding analaysis. Shrinking $\bar\delta$ if necessary, we can assume without loss of generality that $\bar\delta\leq\delta_\lambda$. Further shrinking $\mathcal W$ if necessary, then $\mathcal W\subseteq\mathcal W_\lambda$ which corresponds to $\epsilon=\lambda$, and hence \eqref{eq:Holder continuity,2-0} follows.

			\textbf{Part III. Composite exponent.}	
			
			Finally, for any $\delta\leq \bar\delta$, when $\mu_{\delta}$ is locally upper H\"older continuous at $\mathbf{p}_0$ with exponent $\beta$ and modulus $\nu$,  we can shrink $\mathcal{W}$ if necessary so that in the underlying neighborhood $\mu_{\delta}(\mathbf{p})\leq \nu\norm{\mathbf{p}-\mathbf{p}_0}^\beta$. By definition of $\mu_{\theta}$ for $\theta\in [0,\gamma]$ and recall that $\delta\leq \bar\delta=\frac{\delta_0}{2}<\gamma$, we get that 
			\begin{equation*} 
				\mu_0(\mathbf{p})\leq\mu_{\delta}(\mathbf{p})\leq \nu\norm{\mathbf{p}-\mathbf{p}_0}^\beta. 
			\end{equation*}
			Thus by \eqref{eq:Holder continuity,1} and \eqref{eq:Holder continuity,2-0}, the last assertion of Theorem~\ref{thm:holder-ge} holds for $M=(1-\frac{1}{2^\alpha})^{-1}\lambda\nu^\alpha$ when $0<a<1$ and $M=2\lambda\nu^\alpha$ for $\alpha\geq 1$.
		\end{proof}
		
		\section{Upper H\"olderian with Explicit Exponent of Least Squares}\label{sec:local_upper_Holder}

		\subsection{Ball-LS under Linear Perturbation}\label{sec:linear-pert-ls}
		The following proposition, which characterizes the upper H\"olderian and its explicit exponent of the solution mapping of Ball-LS under linear perturbation, plays a crucial role in the proof of Lemma~\ref{lem:holder-nls-exponent} for the case of parameteric Ball-LS. Proposition~\ref{prop:holder-calm} is of independent interest, and has the potential connections to general quadratically constrained quadratic programming (QCQP).
		\begin{proposition}\label{prop:holder-calm}
			Let $\mathbf{b}\in\mathbb{R}^m$ and $\mathbf{A}\in\mathbb{R}^{m\times n}$, 
			consider Ball-LS under linear perturbation $\mathbf{z}\in\mathbb{R}^n$:
			\begin{equation}\label{eq: a special parametric optimization}
				\min_{\mathbf{y}\in\mathcal{B}}\frac{1}{2}\norm{\mathbf{Ay}-\mathbf{b}}^2-\inner{\mathbf{y},\mathbf{z}}.
			\end{equation}
			Then the solution map $\mathcal{R}_{\mathbf{A},\mathbf{b}}:\mathbb{R}^n\rightrightarrows \mathbb{R}^n$ defined by 
			\begin{equation*}
				\mathcal{R}_{\mathbf{A},\mathbf{b}}(\mathbf{z})=\argmin_{\mathbf{y}\in\mathcal{B}}\frac{1}{2}\norm{\mathbf{Ay}-\mathbf{b}}^2-\inner{\mathbf{y},\mathbf{z}}
			\end{equation*}
			is a nonempty set-valued mapping and $\mathcal{R}_{\mathbf{A},\mathbf{b}}\in\mathcal{H}_{\mathbf{0}}(\alpha,\mu)$ for some $\mu>0$ and $\alpha$ chosen as follows:
			\begin{equation}\label{eq:Holder_exponent_1}
				\alpha=\begin{cases}
					\frac{1}{2}& \tilde{\Delta}_{\mathbf{A},\mathbf{b}}(\mathbf{0})<0,\\
					1&\tilde{\Delta}_{\mathbf{A},\mathbf{b}}(\mathbf{0})>0,\\
					\frac{1}{3}&  \tilde{\Delta}_{\mathbf{A},\mathbf{b}}(\mathbf{0})=0,
				\end{cases}
			\end{equation}
			where  \begin{equation}\label{eq: crieterion for special equation 1}
				\tilde{\Delta}_{\mathbf{A},\mathbf{b}}(\mathbf{z}):=(\mathbf{z}^{\tp}+\mathbf{b}^{\tp}\mathbf{A})((\mathbf{AA}^{\tp})^{\dag})^2(\mathbf{A}^{\tp}\mathbf{b}+\mathbf{z})-1.
			\end{equation} 
		\end{proposition}
		
		\begin{proof}{Proof}
			For $\mathbf{A}=\mathbf{0}$, $\mathcal{R}_{\mathbf{0},\mathbf{b}}(\mathbf{z})=\begin{cases}
				\{\frac{\mathbf{z}}{\norm{\mathbf{z}}}\}	&\mathbf{z}\neq \mathbf{0};\\
				\mathcal{B}	&\mathbf{z}=\mathbf{0}.
			\end{cases}$  and thus $\mathcal{R}_{\mathbf{0},\mathbf{b}}\in\mathcal{H}_{\mathbf{0}}(\alpha,\mu)$ for any $\alpha, \mu>0$.  
			In the following, we assume that $\mathbf{A}\neq \mathbf{0}$.
			
			Since $\mathcal{B}$ is compact, then $\mathcal{R}_{\mathbf{A},\mathbf{b}}$ is surely a nonempty set-valued mapping. 	Suppose that $\rank(\mathbf{A})=r\in\{1,\dots, n\}$ and  $\mathbf{A}=\mathbf{U}\begin{bmatrix}
				\mathbf{\Sigma}_0&0\\0&0
			\end{bmatrix}\mathbf{V}^{{\tp}}$ is the singular value decomposition of $\mathbf{A}$ with  $$\mathbf{\Sigma}_0:=\diag(\sigma_1,\dots,\sigma_r)\in\mathbb{R}^{r\times r}, \sigma_1\geq \sigma_2\geq \cdots\geq \sigma_r>0.$$ Denote $\mathbf{c}:=\begin{bmatrix}
				\mathbf{\Sigma}_0&0\\0&0
			\end{bmatrix}\mathbf{U}^{\tp} \mathbf{b}=(c_1,\dots,c_r,0,\dots,0)^{\tp}$, $\mathbf{c}_1:=(c_1,\dots,c_r)^{\tp}$. Denote $\mathbf{w}=(w_1,\dots,w_n)^{\tp}=\mathbf{V}^{\tp}\mathbf{z}=(\mathbf{w}_1^{\tp},\mathbf{w}_2^{\tp})^{\tp}$ with $\mathbf{w}_1\in \mathbb{R}^r$. Then by Proposition~\ref{prop:optimal-value-solution-mapping},  we have that 
			\begin{equation}\label{eq: solution mapping for the specail problem 2}
				\mathcal{R}_{\mathbf{A},\mathbf{b}}(\mathbf{z})=\begin{cases}
					\bigg\{\mathbf{V}\begin{bmatrix}
						\mathbf{\Sigma}_0^{-2}(\mathbf{w}_1+\mathbf{c}_1)\\\mathbf{x}
					\end{bmatrix}:\norm{\mathbf{x}}^2+\norm{	\mathbf{\Sigma}_0^{-2}(\mathbf{w}_1+\mathbf{c}_1)}^2\leq 1\bigg\},&\tilde{\Delta}_{\mathbf{A},\mathbf{b}}(\mathbf{z})\leq 0,\\
					\bigg\{\mathbf{V}\begin{bmatrix}
						(\mathbf{\Sigma}_0^2+2\lambda_{\mathbf{A},\mathbf{b}}(\mathbf{z})\mathbf{I}_{r})^{-1}(\mathbf{w}_1+\mathbf{c}_1)\\\frac{1}{2\lambda_{\mathbf{A},\mathbf{b}}(\mathbf{z})}\mathbf{w}_2
					\end{bmatrix}\bigg\},&\tilde{\Delta}_{\mathbf{A},\mathbf{b}}(\mathbf{z})> 0.
				\end{cases}
			\end{equation}
			where $\tilde{\Delta}_{\mathbf{A},\mathbf{b}}(\mathbf{z})$ is defined as~\eqref{eq: crieterion for special equation 1}, which also has the form (cf.\ \eqref{eq:criteria-delta})
			\begin{equation}\label{eq: crieterion for special equation 2}
				\tilde{\Delta}_{\mathbf{A},\mathbf{b}}(\mathbf{z})=\norm{	\mathbf{\Sigma}_0^{-2}(\mathbf{w}_1+\mathbf{c}_1)}^2-1=\sum_{i=1}^r\frac{(w_i+c_i)^2}{\sigma_i^4}-1,
			\end{equation}
			and the optimal multiplier $\lambda_{\mathbf{A},\mathbf{b}}(\mathbf{z})>0$ is the unique element satisfying  	\begin{equation}\label{eq: first order equation for special problem 2}
				\begin{aligned}
					1&=\frac{1}{4(\lambda_{\mathbf{A},\mathbf{b}}(\mathbf{z}))^2}\norm{\mathbf{w}_2}^2+\norm{(\mathbf{\Sigma}_0^2+2\lambda_{\mathbf{A},\mathbf{b}}(\mathbf{z})\mathbf{I}_{r})^{-1}(\mathbf{w}_1+\mathbf{c}_1)}^2\\&=\frac{\norm{\mathbf{w}_2}^2}{4(\lambda_{\mathbf{A},\mathbf{b}}(\mathbf{z}))^2}+\sum_{i=1}^r\frac{(w_i+c_i)^2}{(\sigma_i^2+2\lambda_{\mathbf{A},\mathbf{b}}(\mathbf{z}))^2}.
				\end{aligned}
			\end{equation}
			
			The rest proof of $\mathcal{R}_{\mathbf{A},\mathbf{b}}\in \mathcal{H}_{\mathbf{0}}(\alpha,\mu)$ for some $\mu>0$ and $\alpha$ defined as~\eqref{eq:Holder_exponent_1} is divided into three cases aligning with the categories in~\eqref{eq:Holder_exponent_1}. 
			
			\textbf{Part I. The case  $\tilde{\Delta}_{\mathbf{A},\mathbf{b}}(\mathbf{0})=\norm{\mathbf{\Sigma}_0^{-2}\mathbf{c}_1}-1<0$.} Since $\tilde{\Delta}_{\mathbf{A},\mathbf{b}}(\mathbf{z})$ is continuous with respect to $\mathbf{z}$, then there exists a bounded neighborhood $\mathcal{U}_1$ of $\mathbf{0}$, such that for any $\mathbf{z}\in\mathcal{U}_1$,  
			$\tilde{\Delta}_{\mathbf{A},\mathbf{b}}(\mathbf{z})=\norm{	\mathbf{\Sigma}_0^{-2}(\mathbf{w}_1+\mathbf{c}_1)}^2-1< 0$. Thus, by~\eqref{eq: solution mapping for the specail problem 2}, for any $\mathbf{y}_0\in \mathcal{R}_{\mathbf{A},\mathbf{b}}(\mathbf{z})$, there exists $\mathbf{x}_0\in\mathbb{R}^{n-r}$ with $\norm{\mathbf{x}_0}^2+\norm{	\mathbf{\Sigma}_0^{-2}(\mathbf{w}_1+\mathbf{c}_1)}^2\leq 1$, such that $\mathbf{y}_0=\mathbf{V}\begin{bmatrix}
				\mathbf{\Sigma}_0^{-2}(\mathbf{w}_1+\mathbf{c}_1)\\\mathbf{x}_0
			\end{bmatrix}.$
			
			Let $\theta:=\min\Big\{(\frac{1-\norm{\mathbf{\Sigma}_0^{-2}\mathbf{c}_1}^2}{1-\norm{\mathbf{\Sigma}_0^{-2}(\mathbf{w}_1+\mathbf{c}_1)}^2})^{\frac{1}{2}},1\Big\}> 0$ and $\tilde{\mathbf{y}}_0:=\mathbf{V}\begin{bmatrix}
				\mathbf{\Sigma}_0^{-2}\mathbf{c}_1\\\theta\mathbf{x}_0
			\end{bmatrix}$. Then 
			\begin{equation}\label{eq: error bound ,1}
				(1-\theta)^2\leq 1-\theta^2\leq \frac{\vert\norm{\mathbf{\Sigma}_0^{-2}(\mathbf{w}_1+\mathbf{c}_1)}^2-\norm{\mathbf{\Sigma}_0^{-2}\mathbf{c}_1}^2\vert}{1-\norm{\mathbf{\Sigma}_0^{-2}(\mathbf{w}_1+\mathbf{c}_1)}^2}\leq \frac{L\norm{\mathbf{\Sigma}_0^{-2}\mathbf{w}_1}}{1-\norm{\mathbf{\Sigma}_0^{-2}(\mathbf{w}_1+\mathbf{c}_1)}^2},\end{equation} 
			where $L:=\max_{\mathbf{z}\in \cl(\mathcal{U}_1)}\{\norm{\mathbf{\Sigma}_0^{-2}\mathbf{c}_1}+2\norm{\mathbf{\Sigma}_0^{-2}\mathbf{w}_1}\}$ which is well-defined since $\mathcal{U}_1$ is assumed bounded. Since $$\norm{\mathbf{\Sigma}_0^{-2}\mathbf{c}_1}^2+\norm{\theta\mathbf{x}_0}^2\leq \norm{\mathbf{\Sigma}_0^{-2}\mathbf{c}_1}^2+\frac{1-\norm{\mathbf{\Sigma}_0^{-2}\mathbf{c}_1}^2}{1-\norm{\mathbf{\Sigma}_0^{-2}(\mathbf{w}_1+\mathbf{c}_1)}^2}(1-\norm{	\mathbf{\Sigma}_0^{-2}(\mathbf{w}_1+\mathbf{c}_1)}^2)=1,$$
			then $\tilde{\mathbf{y}}_0\in \mathcal{R}_{\mathbf{A},\mathbf{b}}(\mathbf{0})$.
			We have that 
			\[
			\begin{aligned}\rho_{DH}(\{\mathbf{y}_0\},\mathcal{R}_{\mathbf{A},\mathbf{b}}(\mathbf{0}))&\leq \norm{\mathbf{y}_0-\tilde{\mathbf{y}}_0}= (\norm{\mathbf{\Sigma}_0^{-2}\mathbf{w}_1}^2+(1-\theta)^2\norm{\mathbf{x}_0}^2)^{\frac{1}{2}}\\&\leq (\norm{\mathbf{\Sigma}_0^{-2}\mathbf{w}_1}^2+(1-\theta^2)(1-\norm{	\mathbf{\Sigma}_0^{-2}(\mathbf{w}_1+\mathbf{\Sigma}_0\mathbf{c}_1)}^2))^{\frac{1}{2}}\\
				&\leq (\norm{\mathbf{\Sigma}_0^{-2}\mathbf{w}_1}^2+L\norm{\mathbf{\Sigma}_0^{-2}\mathbf{w}_1})^{\frac{1}{2}}\leq L'\norm{\mathbf{\Sigma}_0^{-2}\mathbf{w}_1}^{\frac{1}{2}}\leq \frac{L'}{\sigma_r}\norm{\mathbf{w}_1}^{\frac{1}{2}}\leq \frac{L'}{\sigma_r}\norm{\mathbf{z}}^{\frac{1}{2}},
			\end{aligned}
			\]
			where the second inequality follows from the fact that $\norm{\mathbf{x}_0}^2+\norm{\mathbf{\Sigma}_0^{-2}(\mathbf{w}_1+\mathbf{c}_1)}^2\leq 1$, the third from \eqref{eq: error bound ,1}, the fourth from $L'=\max_{\mathbf{z}\in\cl(\mathcal{U}_1)}(\norm{\mathbf{\Sigma}_0^{-2}\mathbf{w}_1}+L)$, and the last from $\norm{\mathbf{z}}^2=\norm{\mathbf{w}}^2=\norm{\mathbf{w}_1}^2+\norm{\mathbf{w}_2}^2$.
			Therefore, by Proposition~\ref{prop:equivalent}, $\mathcal{R}_{\mathbf{A},\mathbf{b}}\in\mathcal{H}_{\mathbf{0}}(\frac{1}{2},\frac{L'}{\sigma_r})$. 
			
			\textbf{Part II. The case  $\tilde{\Delta}_{\mathbf{A},\mathbf{b}}(\mathbf{0})=\norm{\mathbf{\Sigma}_0^{-2}\mathbf{c}_1}-1>0$.} 
			Likewise, there exists a neighborhood $\mathcal{U}_2$ of $\mathbf{0}$, such that for any $\mathbf{z}\in\mathcal{U}_2$, we have 
			$\tilde{\Delta}_{\mathbf{A},\mathbf{b}}(\mathbf{z})=\norm{	\mathbf{\Sigma}_0^{-2}(\mathbf{w}_1+\mathbf{\Sigma}_0\mathbf{c}_1)}^2-1=\sum_{i=1}^r\frac{(w_i+c_i)^2}{\sigma_i^4}-1> 0$ and thus $ \mathcal{R}_{\mathbf{A},\mathbf{b}}(\mathbf{z})=\bigg\{\mathbf{V}\begin{bmatrix}
				(\mathbf{\Sigma}_0^2+2\lambda_{\mathbf{A},\mathbf{b}}(\mathbf{z})\mathbf{I}_{r})^{-1}(\mathbf{w}_1+\mathbf{c}_1)\\\frac{1}{2\lambda_{\mathbf{A},\mathbf{b}}(\mathbf{z})}\mathbf{w}_2
			\end{bmatrix}\bigg\}$ with $\lambda_{\mathbf{A},\mathbf{b}}(\mathbf{z})>0$ satisfies~\eqref{eq: first order equation for special problem 2}. So   $\mathcal{R}_{\mathbf{A},\mathbf{b}}$ is a single-valued mapping on $\mathcal{U}_2$.
			
			Note that $\lambda_{\mathbf{A},\mathbf{b}}(\mathbf{z})$ satisfies~\eqref{eq: first order equation for special problem 2}, which is equivalent to the condition that $\lambda_{\mathbf{A},\mathbf{b}}(\mathbf{z})$ is the unique positive element satisfying the implicit equation for $\lambda$:
			\begin{equation}\label{eq: implicit equation}
				m(\lambda,\mathbf{z}):=\frac{\norm{\mathbf{w}_2}^2}{4\lambda^2}+\sum_{i=1}^r\frac{(w_i+c_i)^2}{(\sigma_i^2+2\lambda)^2}-1=0.
			\end{equation} 
			Since for $\mathbf{z}\in\mathcal{U}_2$,  $\tilde{\Delta}_{\mathbf{A},\mathbf{b}}(\mathbf{z})=\sum_{i=1}^r\frac{(w_i+c_i)^2}{\sigma_i^4}-1>0$, then there exists at least one $|c_{i_0}+w_{i_0}|>0$. Note that $\frac{\partial m}{\partial \lambda}(\lambda,\mathbf{z})=-\frac{\norm{\mathbf{w}_2}^2}{2\lambda^3}-4\sum_{i=1}^r\frac{(w_i+c_i)^2}{(\sigma_i^2+2\lambda)^3}$ and $\lambda_{\mathbf{A},\mathbf{b}}(\mathbf{z})>0$, then $\frac{\partial m}{\partial \lambda}(\lambda_{\mathbf{A},\mathbf{b}}(\mathbf{z}),\mathbf{z})<0$ for any $\mathbf{z}\in\mathcal{U}_2$. Thus by the classical implicit function theorem~\cite[Theorem~9.2]{JRMunkres1991}, we get that $\lambda_{\mathbf{A},\mathbf{b}}(\mathbf{z})$ is a $\mathcal{C}^1$ function of $\mathbf{z}$ in a compact convex neighborhood $\mathcal{U}_3\subseteq \mathcal{U}_2$ of $\mathbf{0}$.  Therefore $\mathcal{R}_{\mathbf{A},\mathbf{b}}$ is $\mathcal{C}^1$ on $\mathcal{U}_3$.
			Since $\mathcal{U}_3$ is assumed to be  compact, then there exists $L>0$ such that $\norm{d\mathcal{R}_{\mathbf{A},\mathbf{b}}(\mathbf{z})}_2\leq L$ for all $\mathbf{z}\in \mathcal{U}_3$. Thus in this case, for all $\mathbf{z}\in\mathcal{U}_3$, we have
			\begin{equation*}
				\begin{aligned}
					\rho_{DH}(\mathcal{R}_{\mathbf{A},\mathbf{b}}(\mathbf{z}),\mathcal{R}_{\mathbf{A},\mathbf{b}}(\mathbf{0}))&=\norm{\mathcal{R}_{\mathbf{A},\mathbf{b}}(\mathbf{z})-\mathcal{R}_{\mathbf{A},\mathbf{b}}(\mathbf{0})}=\norm{\int_{0}^{1}d\mathcal{R}_{\mathbf{A},\mathbf{b}}(\tau\mathbf{z})[\mathbf{z}]d\tau}
					\\&\leq \int_0^1\norm{d\mathcal{R}_{\mathbf{A},\mathbf{b}}(\tau\mathbf{z})}_2\norm{\mathbf{z}}d\tau \leq L\norm{\mathbf{z}}.
				\end{aligned}
			\end{equation*}
			Therefore, by Proposition~\ref{prop:equivalent},  $\mathcal{R}_{\mathbf{A},\mathbf{b}}\in\mathcal{H}_{\mathbf{0}}(1,L)$.  
			
			\textbf{Part III. The case  $\tilde{\Delta}_{\mathbf{A},\mathbf{b}}(\mathbf{0})=\norm{\mathbf{\Sigma}_0^{-2}\mathbf{c}_1}-1=0$.}  In this case, $\mathcal{R}_{\mathbf{A},\mathbf{b}}(\mathbf{0})=\bigg\{\mathbf{V}\begin{bmatrix}
				\mathbf{\Sigma}_0^{-2}\mathbf{c}_1\\\mathbf{0}
			\end{bmatrix}\bigg\}$. 
			In the following, we have two cases depending on $\mathbf{z}$. 
			
			\textbf{Part III(a). The case $\tilde{\Delta}_{\mathbf{A},\mathbf{b}}(\mathbf{z})>0$.}
			Note that $\lambda_{\mathbf{A},\mathbf{b}}(\mathbf{z})>0$ satisfies~\eqref{eq: implicit equation} which is equivalent to $r_{\mathbf{z}}(\lambda_{\mathbf{A},\mathbf{b}}(\mathbf{z}))=0$ with
			\begin{equation}\label{eq: implicit function,2}
				r_{\mathbf{z}}(\lambda):=\norm{\mathbf{w}_2}^2-4\lambda^2(1-\sum_{i=1}^r(\frac{w_i+c_i}{\sigma_i^2+2\lambda})^2).
			\end{equation}
			By this reformulation, the case when $\mathbf{z}=\mathbf{0}$ and $\lambda=0$ is covered by  ~\eqref{eq: implicit function,2} as well.

			For the equation $r_{\mathbf{z}}(\lambda_{\mathbf{A},\mathbf{b}}(\mathbf{z}))=0$, by eliminating the denominators of \eqref{eq: implicit function,2}, we get that $\lambda_{\mathbf{A},\mathbf{b}}(\mathbf{z})$ is a root of a (real) polynomial in $\mathbb{R}[w_1,\dots,w_n][\lambda]$ with $\mathbb{R}[w_1,\dots,w_n]$ the real polynomial ring of the variables $w_1,\dots,w_n$. Note that $\mathbb{R}[w_1,\dots,w_n]=\mathbb{R}[z_1,\dots,z_n]$. Then by the multivariate version of the Newton-Puiseux theorem (cf.\ \cite[Theorem~3.6]{McDonald1995}), we know that $\lambda_{\mathbf{A},\mathbf{b}}(\mathbf{z})$ can be expressed as a fractional power series of $z_1,\dots,z_n$, that is,
			\begin{equation}\label{eq:NP_series} 
				\lambda_{\mathbf{A},\mathbf{b}}(\mathbf z)=\sum_{\mathbf{k} \in C\cap\frac{1}{M}\mathbb{Z}^{n}}p_{\mathbf{k}} \mathbf{z}^{\mathbf k}
			\end{equation} 
			for some $0<M\in \mathbb{Z}$ and $C:=\cup_{i=1}^K\{\lambda \mathbf{h}_i:\lambda\in\mathbb{R}_+\}$ a finitely generated cone with certain finitely many generators $\mathbf{h}_i\in\mathbb{R}^n$. Here $\mathbf{z}^\mathbf{k}:=z_1^{k_1}z_2^{k_2}\cdots z_n^{k_n}$ with $\mathbf{k}=(k_1,\dots,k_n)$. 
			
			Since $\lambda_{\mathbf{A},\mathbf{b}}(\mathbf{0})=0$, then by the continuity of the root of a polynomial with respect to the coefficients (cf.\ \ \cite[Section VI.I, Chapter VI]{bhatia2013matrix}), we know that the exponent $\mathbf{k}$ in the summation on the righthand side of \eqref{eq:NP_series} with $p_{\mathbf{k}}\neq 0$ must satisfy $\mathbf{k}\in \mathbf{R}_+^n$ and $\mathbf{k}\neq \mathbf{0}$. Consequently, as $\norm{\mathbf{z}}^2=\norm{\mathbf{w}}^2\geq |w_i|^2$ for any $1\leq i\leq n$, there exists $\alpha>0$ such that $\lambda_{\mathbf{A},\mathbf{b}}(\mathbf{z})=O(\norm{\mathbf{z}}^\alpha)=O(\norm{\mathbf{w}}^\alpha)$ when $\mathbf{z}\to\mathbf{0}$.  In the following, we give an estimation of  $\alpha$.
			
			Take Taylor's formula for~\eqref{eq: implicit function,2} at $\lambda=0$, we get that 
			\begin{equation*}
				\begin{aligned}
					r_{\mathbf{z}}(\lambda)&=r_{\mathbf{z}}(0)+r_{\mathbf{z}}'(0)\lambda+\frac{1}{2}r_{\mathbf{z}}''(0)\lambda^2+\frac{1}{6}
					r_{\mathbf{z}}^{(3)}(0)\lambda^3+\frac{1}{24}r_{\mathbf{z}}^{(4)}(\zeta)\lambda^4\\&=\norm{\mathbf{w}_2}^2-4(1-\sum_{i=1}^r\frac{(w_i+c_i)^2}{\sigma_i^4})\lambda^2-16\sum_{i=1}^r\frac{(w_i+c_i)^2}{\sigma_i^6}\lambda^3+\frac{1}{24}
					r_{\mathbf{z}}^{(4)}(\xi)\lambda^4\\
					&=\norm{\mathbf{w}_2}^2+4\sum_{i=1}^r\frac{2c_iw_i+w_i^2}{\sigma_i^4}\lambda^2
					-16\sum_{i=1}^r\frac{(w_i+c_i)^2}{\sigma_i^6}\lambda^3+\frac{1}{24}r_{\mathbf{z}}^{(4)}(\xi)\lambda^4,\end{aligned}
			\end{equation*}
			where $\xi\in(0,\lambda)$ and the last equality follows from the relation $\sum_{i=1}^r\frac{c_i^2}{\sigma_i^4}=1$.
			So combining this with $r_{\mathbf{z}}(\lambda_{\mathbf{A},\mathbf{b}}(\mathbf{z}))=0$, we get that 
			\begin{equation}\label{eq: implicit function, 3}
				\norm{\mathbf{w}_2}^2=-4\sum_{i=1}^r\frac{2c_iw_i+w_i^2}{\sigma_i^4}\lambda_{\mathbf{A},\mathbf{b}}(\mathbf{z})^2
				+16\sum_{i=1}^r\frac{(w_i+c_i)^2}{\sigma_i^6}\lambda_{\mathbf{A},\mathbf{b}}(\mathbf{z})^3-\frac{1}{24}r_{\mathbf{z}}^{(4)}(\xi(\mathbf{z}))\lambda_{\mathbf{A},\mathbf{b}}(\mathbf{z})^4,
			\end{equation}
			with $\xi(\mathbf{z})\in (0,\lambda_{\mathbf{A},\mathbf{b}}(\mathbf{z}))$. By a direct calculation, we get that 
			\begin{equation}\label{eq: 4th order derivative}
				\frac{1}{24}r_{\mathbf{z}}^{(4)}(\lambda)=48\sum_{i=1}^r\frac{(w_i+c_i)^2}{(\sigma_i^2+2\lambda)^4}-256\lambda\sum_{i=1}^r\frac{(w_i+c_i)^2}{(\sigma_i^2+2\lambda)^5}+320\lambda^2\sum_{i=1}^r\frac{(w_i+c_i)^2}{(\sigma_i^2+2\lambda)^6}.
			\end{equation} 
			
			\begin{itemize}
				\item[$\bullet$]	If $\norm{\mathbf{w}_2}=0$, then by~\eqref{eq: implicit function, 3},~\eqref{eq: 4th order derivative} and  $\lambda_{\mathbf{A},\mathbf{b}}(\mathbf{z})>0$ when $\mathbf{z}\neq \mathbf{0}$, we get that for $\mathbf{z}\neq \mathbf{0}$ sufficiently small, 
				\begin{equation*}
					\begin{aligned}
						4\sum_{i=1}^r\frac{2c_iw_i+w_i^2}{\sigma_i^4}&=16\sum_{i=1}^r\frac{(w_i+c_i)^2}{\sigma_i^6}\lambda_{\mathbf{A},\mathbf{b}}(\mathbf{z})-\frac{1}{24}r_{\mathbf{z}}^{(4)}(\xi(\mathbf{z}))\lambda_{\mathbf{A},\mathbf{b}}(\mathbf{z})^2
						=O(\norm{\mathbf{w}}^\alpha),
					\end{aligned}
				\end{equation*}
				since $r_{\mathbf{z}}^{(4)}(\xi(\mathbf{z}))$ and $16\sum_{i=1}^r\frac{(w_i+c_i)^2}{\sigma_i^6}$ are bounded when $\norm{\mathbf{z}}$ is sufficiently small, and $c_i(1\leq i\leq r)$ are not all $0$ as $0=\tilde{\Delta}_{\mathbf{A},\mathbf{b}}(\mathbf{0})=\norm{\Sigma_0^{-2}\mathbf{c}_1-1}-1=\sum_{i=1}^r\frac{c_i^2}{\sigma_i^4}-1$ in this case. Thus,  we must have $\alpha=1$.
				
				\item[$\bullet$]	If $\norm{\mathbf{w}_2}\neq 0$, then the term with the lowerest exponet  on $\|\mathbf w\|$ in both sides of \eqref{eq: implicit function, 3} should have combined zero coefficient. Thus, by  $\lambda_{\mathbf{A},\mathbf{b}}(\mathbf{z})>0$ when $\mathbf{z}\neq \mathbf{0}$, we get that for $\mathbf{z}\neq \mathbf{0}$ sufficiently small, it holds
				\begin{equation*}
					\begin{aligned}
						\norm{\mathbf{w}_2}^2=O(\max\{\norm{\mathbf{w}}^{1+2\alpha},\norm{\mathbf{w}}^{3\alpha}\}).
					\end{aligned}
				\end{equation*}
				Hence, we must have $\alpha=\frac{2}{3}$.
			\end{itemize}	
			Therefore $\lambda_{\mathbf{A},\mathbf{b}}=O(\norm{\mathbf{z}}^{\frac{2}{3}})$ in all cases.

			In this case, denote 
			\[ 
			\mathbf{d}(\mathbf{z}):=\begin{bmatrix}(\mathbf{\Sigma}_0^2+2\lambda_{\mathbf{A},\mathbf{b}}(\mathbf{z})\mathbf{I}_{r})^{-1}(\mathbf{w}_1+\mathbf{c}_1)\\\frac{1}{2\lambda_{\mathbf{A},\mathbf{b}}(\mathbf{z})}\mathbf{w}_2\end{bmatrix}=(d_1(\mathbf{z}),\dots,d_n(\mathbf{z}))^{\tp}
			\] 
			and $\mathbf{d}({\mathbf{0}}):=\begin{bmatrix}
				\mathbf{\Sigma}_0^{-2}\mathbf{c}_1\\\mathbf{0}
			\end{bmatrix}=(d_1(\mathbf{0}),\dots,d_r(\mathbf{0}),0,\dots,0)^{\tp}$, then $\mathcal{R}_{\mathbf{A},\mathbf{b}}(\mathbf{0})=\{\mathbf{V}\mathbf{d}(\mathbf{0})\}$ and $\mathcal{R}_{\mathbf{A},\mathbf{b}}(\mathbf{z})=\{\mathbf{V}\mathbf{d}(\mathbf{z})\}$.
			By a direct calculation, we get that $\{d_i(\mathbf{z})\}$ and $\lambda_{\mathbf{A},\mathbf{b}}(\mathbf{z})>0$ is the solution of the following nonlinear equations with respect to $\{d_i\}$ and $\lambda>0$:
			\begin{equation}\label{eq: nonlinear equation}
				\begin{cases}
					(\sigma_i^2+2\lambda)d_i=w_i+c_i ,&1\leq i\leq r,\\
					2\lambda d_i=w_i,& r+1\leq i\leq n,\\
					\sum_{i=1}^n d_i^2=1.&
				\end{cases}
			\end{equation}
			Since $\lambda_{\mathbf{A},\mathbf{b}}(\mathbf{z})>0$, then we can solve \eqref{eq: nonlinear equation} to get that $d_i(\mathbf{z})=\frac{w_i+c_i}{\sigma_i^2+2\lambda_{\mathbf{A},\mathbf{b}}(\mathbf{z})}$ for $1\leq i\leq r$ and $\sum_{i=r+1}^nd_i(\mathbf{z})^2=1-\sum_{i=1}^rd_i(\mathbf{z})^2=1-\sum_{i=1}^r(\frac{w_i+c_i}{\sigma_i^2+2\lambda_{\mathbf{A},\mathbf{b}}(\mathbf{z})})^2$.
			Note that the squared directed Hausdorff distance $(\rho_{DH}(\mathcal{R}_{\mathbf{A},\mathbf{b}}(\mathbf{z}),\mathcal{R}_{\mathbf{A},\mathbf{b}}(\mathbf{0})))^2=\norm{\mathbf{d}_i(\mathbf{z})-\mathbf{d}_i(\mathbf{0})}^2=\sum_{i=r+1}^nd^2_i(\mathbf{z})+\sum_{i=1}^r(\frac{c_i}{\sigma_i^2}-d_i(\mathbf{z}))^2$ can then be expressed by $M_{\mathbf{z}}(\lambda_{\mathbf{A},\mathbf{b}}(\mathbf{z}))$, where
			\begin{equation}\label{eq: directed Hausdorff distance}
				M_{\mathbf{z}}(\lambda):=1-\sum_{i=1}^r(\frac{w_i+c_i}{\sigma_i^2+2\lambda})^2+\sum_{i=1}^r(\frac{c_i}{\sigma_i^2}-\frac{w_i+c_i}{\sigma_i^2+2\lambda})^2
				=2-2\sum_{i=1}^r\frac{c_i(w_i+c_i)}{\sigma_i^2(\sigma_i^2+2\lambda)},
			\end{equation}
			where the second equality follows from the fact that $\sum_{i=1}^r\frac{c_i^2}{\sigma_i^4}=1$. 
			With~\eqref{eq: directed Hausdorff distance}, by Taylor's formula and the relation $\sum_{i=1}^r\frac{c_i^2}{\sigma_i^4}=1$, we get that when $\norm{\mathbf{z}}$ sufficiently small,
			\begin{equation}\label{eq:ball-ls-exp-1}
				\begin{aligned}
					(\rho_{DH}(\mathcal{R}_{\mathbf{A},\mathbf{b}}(\mathbf{z}),\mathcal{R}_{\mathbf{A},\mathbf{b}}(\mathbf{0})))^2&=M_{\mathbf{z}}(\lambda_{\mathbf{A},\mathbf{b}}(\mathbf{z}))\\&=-2\sum_{i=1}^r\frac{w_ic_i}{\sigma_i^4}+4\sum_{i=1}^r\frac{c_i(w_i+c_i)}{\sigma_i^6}\lambda_{\mathbf{A},\mathbf{b}}(\mathbf{z})+o(\lambda_{\mathbf{A},\mathbf{b}}(\mathbf{z}))\\
					&=4\sum_{i=1}^r\frac{c_i^2}{\sigma_i^6}\lambda_{\mathbf{A},\mathbf{b}}(\mathbf{z})-2\sum_{i=1}^r\frac{w_ic_i}{\sigma_i^4}+4\sum_{i=1}^r\frac{c_iw_i}{\sigma_i^6}\lambda_{\mathbf{A},\mathbf{b}}(\mathbf{z})+o(\lambda_{\mathbf{A},\mathbf{b}}(\mathbf{z}))\\
					&=O(\lambda_{\mathbf{A},\mathbf{b}}(\mathbf{z}))=O(\norm{\mathbf{z}}^{\frac{2}{3}}).
				\end{aligned}
			\end{equation}
			
			\textbf{Part III(b). The case $\tilde{\Delta}_{\mathbf{A},\mathbf{b}}(\mathbf{z})\leq 0$.} Let $\epsilon_0>0$ and assume that $\mathbf{z}\in\epsilon_0\mathcal{B}$. In this case, we have $\mathcal{R}_{\mathbf{A},\mathbf{b}}(\mathbf{z})=\bigg\{\mathbf{V}\begin{bmatrix}
				\mathbf{\Sigma}_0^{-2}(\mathbf{w}_1+\mathbf{c}_1)\\\mathbf{x}
			\end{bmatrix}:\norm{\mathbf{x}}^2+\norm{\mathbf{\Sigma}_0^{-2}(\mathbf{w}_1+\mathbf{c}_1)}^2\leq 1$\bigg\}.
			So for any $\mathbf{y}\in \mathcal{R}_{\mathbf{A},\mathbf{b}}(\mathbf{z})$, there exists $\mathbf{x}\in\mathbb{R}^{n-r}$ such that $\mathbf{y}=\mathbf{V}\begin{bmatrix}
				\mathbf{\Sigma}_0^{-2}(\mathbf{w}_1+\mathbf{c}_1)\\\mathbf{x}
			\end{bmatrix}$ and $\norm{\mathbf{x}}^2+\norm{	\mathbf{\Sigma}_0^{-2}(\mathbf{w}_1+\mathbf{c}_1)}^2\leq 1$. Recall that $\tilde{\Delta}_{\mathbf{A},\mathbf{b}}(\mathbf{0})=\norm{\mathbf{\Sigma}_0^{-2}\mathbf{c}_1}^2-1=0$. Thus 
			$$
			\begin{aligned}
				\norm{\mathbf{x}}^2&\leq 1-\norm{\mathbf{\Sigma}_0^{-2}(\mathbf{w}_1+\mathbf{c}_1)}^2=\norm{\mathbf{\Sigma}_0^{-2}\mathbf{c}_1}^2-\norm{\mathbf{\Sigma}_0^{-2}(\mathbf{w}_1+\mathbf{c}_1)}^2=\langle \mathbf{\Sigma}_0^{-2}(\mathbf{w}_1+2\mathbf{c}_1),-\mathbf{\Sigma}_0^{-2}\mathbf{w}_1\rangle
				\\&\leq \norm{\mathbf{\Sigma}_0^{-2}(\mathbf{w}_1+2\mathbf{c}_1)}\norm{\mathbf{\Sigma}_0^{-2}\mathbf{w}_1}\leq L\norm{\mathbf{\Sigma}_0^{-2}\mathbf{w}_1}\leq \frac{L}{\sigma_r^2}\norm{\mathbf{w}_1},
			\end{aligned}$$
			where  $L:=\max_{\mathbf{z}\in\epsilon_0\mathcal{B}}\norm{\mathbf{\Sigma}_0^{-2}(\mathbf{w}_1+2\mathbf{c}_1)}$, which is well-defined since $\epsilon_0\mathcal{B}$ is compact.	
			Therefore for any $\mathbf{z}\in\epsilon_0\mathcal{B}$ with $\tilde{\Delta}_{\mathbf{A},\mathbf{b}}(\mathbf{z})\leq 0$, there exists $\mathbf{y}\in\mathcal{R}_{\mathbf{A},\mathbf{b}}(\mathbf z)$ such that 
			\begin{equation}\label{eq: singular case for upper holder continuity, 2}
				\begin{aligned}
					\rho_{DH}(\mathcal{R}_{\mathbf{A},\mathbf{b}}(\mathbf{z}),\mathcal{R}_{\mathbf{A},\mathbf{b}}(\mathbf{0}))&=	\rho_{DH}(\mathbf{y},\mathcal{R}_{\mathbf{A},\mathbf{b}}(\mathbf{0}))=(\norm{\mathbf{x}}^2+\norm{\mathbf{\Sigma}_0^{-2}\mathbf{w}_1}^2)^{\frac{1}{2}}\\
					&\leq \frac{1}{\sigma_r}(L+\norm{\mathbf{\Sigma}_0^{-2}\mathbf{w}_1})^{\frac{1}{2}}\norm{\mathbf{w}_1}^{\frac{1}{2}}\leq \frac{1}{\sigma_r}L'\norm{\mathbf{w}}^{\frac{1}{2}}=\frac{1}{\sigma_r}L'\norm{\mathbf z}^{\frac{1}{2}},
				\end{aligned}
			\end{equation}
			where $L':=\max_{\mathbf{z}\in\epsilon_0\mathcal{B}}(L+\norm{\mathbf{\Sigma}_0^{-2}\mathbf{w}_1})^{\frac{1}{2}}$, which is bounded since $\epsilon_0\mathcal{B}$ is compact.

			Combining~\eqref{eq:ball-ls-exp-1} and~\eqref{eq: singular case for upper holder continuity, 2}, we get that $\mathcal{R}_{\mathbf{A},\mathbf{b}}\in \mathcal{H}_{\mathbf{0}}(\frac{1}{3},M)$ for some constant $M>0$. 	 
		\end{proof}

		\begin{remark}\label{rmk: the Holder exponent}
			It follows from Proposition~\ref{prop:holder-calm} that when $\tilde{\Delta}_{\mathbf{A},\mathbf{b}}(\mathbf{0})> 0$, the H\"older exponent can be $1$, which means that $\mathcal{R}_{\mathbf{A},\mathbf{b}}$ is actually upper Lipschitz continuous at $\mathbf{0}$. 
			
			However, from the proof Proposition~\ref{prop:holder-calm}, when $\tilde{\Delta}_{\mathbf{A},\mathbf{b}}(\mathbf{0})\leq 0$, $\mathcal{R}_{\mathbf{A},\mathbf{b}}$ cannot be upper Lipschitz continuous. This is illustrated by the next example. 
		\end{remark}
		
		\begin{example}\label{exm:fractional}
			We present an axample to showcase the fact that the exponent $\alpha=\frac{1}{3}$ in Proposition~\ref{prop:holder-calm} is tight. Notations in the proof of Proposition~\ref{prop:holder-calm} are inherited.
			
			Let $\mathbf{A}=\begin{bmatrix}
				1&0\\
				0&0
			\end{bmatrix}$, and $\mathbf{b}=(1,0)^{\tp}$. Then $\mathbf{A}$ itself is diagonal, so $\mathbf{\Sigma}_0=[1], \mathbf{c}=\mathbf{\Sigma}_0\mathbf{b}=(1,0)^{\tp}, \mathbf{c}_1=(1)$ and $\mathbf{w}=\mathbf{z}=(z_1,z_2)^{\tp}$. We see that $\tilde{\Delta}_{\mathbf{A},\mathbf{b}}(\mathbf{0})=\norm{\mathbf{\Sigma}_0^{-2}\mathbf{c}_1}-1=0$ and the equation \eqref{eq: implicit equation} for $\lambda$ becomes 
			\begin{equation*}
				\frac{z_2^2}{4\lambda^2}+\frac{(z_1+1)^2}{(1+2\lambda)^2}-1=0,
			\end{equation*}
			which clearly can be reduced to the following quartic equation for $\lambda$:
			\begin{equation}\label{eq:poly_equa}
				\lambda^4+\lambda^3-\frac{1}{4}(z_1^2+z_2^2+2z_1)\lambda^2-\frac{1}{4}z_2^2\lambda-\frac{1}{16}z_2^2=0.
			\end{equation}
			By the proof of Proposition~\ref{prop:holder-calm},  when $z_2\neq 0$, the solution $\lambda_{\mathbf{A},\mathbf{b}}=O(\norm{\mathbf{z}}^{\frac{2}{3}})$ and  $\rho_{DH}(\mathcal{R}_{\mathbf{A},\mathbf{b}}(\mathbf{z}),\mathcal{R}_{\mathbf{A},\mathbf{b}}(\mathbf{0}))=O(\norm{\mathbf{z}}^{\frac{1}{3}})$.  We claim that when the point $\mathbf{z}$ tends to $\mathbf{0}$ along the line $z_1=0$, there exist $0<C_1<C_2$ such that $C_1\norm{\mathbf{z}}^{\frac{2}{3}}<\lambda_{\mathbf{A},\mathbf{b}}<C_2\norm{\mathbf{z}}^{\frac{2}{3}}$, and thus the exponent $\frac{2}{3}$ is tight.
			
			The claim can be shown  by computing the asymptotic series of $\lambda$ with respect to $z_2$ around $z_2=0$ when analyzing the asymptotic behavior of the solution $\lambda_{\mathbf{A},\mathbf{b}}$ along the line $z_1=0$. Via the function ``AsymptoticSolve[ ]" in the symbolic computational software Mathematica, we can get that when $|z_2|$ is small, the unique positive real root has the asymptotic expansion as  \begin{equation}
				\lambda_{\mathbf{A},\mathbf{b}}=2^{-\frac{4}{3}}z_2^{\frac{2}{3}}+2^{-\frac{8}{3}}z_2^{\frac{4}{3}}+\text{higher order terms}=2^{-\frac{4}{3}}\norm{\mathbf{z}}^{\frac{2}{3}}+o(\norm{\mathbf{z}}^{\frac{2}{3}}).
			\end{equation}
			Here the second equality holds since $\norm{\mathbf{z}}=|z_2|$ when $\mathbf{z}$ varies along the line $z_1=0$. The directed Hausdorff distance $O(\norm{\mathbf z}^{\frac{1}{3}})$ between solution sets follows from a further analysis as \eqref{eq:ball-ls-exp-1}. Therefore the claim holds.
		\end{example}
		
		\subsection{Parameteric Ball-LS}
		In this subsection, we  apply Theorem~\ref{thm:holder-ge} to the problem~\ref{eq:ball-nls}. To that end, we recall the generalized equation representation for a parametric convex programming:
		\begin{equation}\label{eq: parametric convex programming}
			\min_{\mathbf{y}\in \mathcal{D}(\mathbf{p})}f(\mathbf{p},\mathbf{y}),
		\end{equation}
		where  $f:\Omega\times \mathbb{R}^n\longrightarrow\mathbb{R}$ with $\Omega\subseteq \mathbb{R}^m$, $f(\mathbf p,\cdot)$ is a proper closed convex function for any $\mathbf{p}\in\Omega$, and $\mathcal{D}:\Omega\rightrightarrows \mathbb{R}^n$ be a set-valued mapping such that for any $\mathbf{p}\in\dom(\mathcal{D})$ the set $\mathcal{D}(\mathbf{p})$ is convex. Then by \cite[Theorem~3.1]{dhara2012optimality}, for any $\mathbf p\in\Omega$, we get that $\mathbf{y}_0\in \mathcal{D}(\mathbf{p})$ is an optimal solution of~\eqref{eq: parametric convex programming} if and only if
		\begin{equation}\label{eq:gen-eqn-gen}
			\mathbf{0}\in \partial_{\mathbf{y}}f(\mathbf{p},\mathbf{y}_0)+ N_{\mathcal{D}(\mathbf{p})}(\mathbf{y}_0), 
		\end{equation}
		where $\partial_{\mathbf{y}}f(\mathbf{p},\mathbf{y}_0)$ is the subdifferential of $f(\mathbf p,\cdot)$ at $\mathbf{y}_0$ and $N_{\mathcal{D}(\mathbf{p})}(\mathbf{y}_0)$ is the normal cone of $\mathcal{D}(\mathbf{p})$ at $\mathbf{y}_0$. In particular,
		the optimality condition of problem~\ref{eq:ball-nls} reads
		\begin{equation}\label{eq: KKT system of the generalized equation form}
			\mathbf{0}\in \mathbf{A}^{\tp}(\mathbf{Ay}-\mathbf{b})+N_{\mathcal{B}(R)}(\mathbf{y}),
		\end{equation}
		where $N_{\mathcal{B}(R)}(\mathbf{y})$ is calculated as
		\begin{equation*}\label{eq: normal cone of closed ball}
			N_{\mathcal{B}(R)}(\mathbf{y})=\begin{cases}
				\{\mathbf{0}\}, &\norm{\mathbf{y}}<R,\\
				\{\lambda \mathbf{y}:\lambda\geq 0\},&\norm{\mathbf{y}}=R,\\
				\emptyset,&\norm{\mathbf{y}}>R.
			\end{cases}
		\end{equation*}
		Note that the convention $\mathbf{x}+\emptyset=\emptyset$ is adopted in this paper.
		
		Note that the set-valued mapping $\mathcal{T}$ in Theorem~\ref{thm:holder-ge} is parameter free, while we treat $R$ as a parameter for the problem \ref{eq:ball-nls}. Thus, a direct application of Theorem~\ref{thm:holder-ge} is not immediate. Nevertheless,  we can consider the case when $R=1$ and employ a scaling $\mathbf{x}:=\frac{\mathbf{y}}{R}$ when $R\neq 0$, i.e., 
		\begin{equation}\label{eq:ball-nls substituted}
			\begin{array}{rl}
				\min_{\mathbf{x}} & \frac{1}{2}\norm{R\mathbf{Ax}-\mathbf{b}}^2\\ 
				\text{s.t.} &\norm{\mathbf{x}}\leq 1.
			\end{array}
		\end{equation}
		The optimal value mapping $\tilde{\mathcal{V}}$ and the solution mapping $\tilde{\mathcal{S}}$ of~\eqref{eq:ball-nls substituted}  are related to those of~\ref{eq:ball-nls} as follows
		\begin{equation}\label{eq:opt-val-solu-scal} 
			\tilde{\mathcal{V}}(\mathbf{A},\mathbf{b},R)=\mathcal{V}(\mathbf{A},\mathbf{b},R),\ \text{and } \tilde{\mathcal{S}}(\mathbf{A},\mathbf{b},R)=\frac{1}{R}\mathcal{S}(\mathbf{A},\mathbf{b},R). 
		\end{equation}   
		Therefore, we can study the case when $R=1$ and employ the scaling $\mathbf{x}:=\frac{\mathbf{y}}{R}$ when $R\neq 0$ to convert the general case to the problem~\eqref{eq:ball-nls substituted}. 
		The optimality condition of~\eqref{eq:ball-nls substituted} is
		\begin{equation}\label{eq:opt-cond-scal}
			\mathbf{0}\in R\mathbf{A}^{\tp}(R\mathbf{Ax}-\mathbf{b})+N_{\mathcal{B}}(\mathbf{x}).
		\end{equation}
		The following lemma shows that the parametric generalized equation~\eqref{eq:opt-cond-scal} fulfills the conditions of Theorem~\ref{thm:holder-ge}.  
		
		Recall that  $\mathbb{U}:=\mathbb{R}^{m\times n}\times \mathbb{R}^{m}\times \mathbb{R}_{+}$ and from \eqref{eq:criteria-delta without svd-0} that 
		\begin{equation}\label{eq:criteria-delta without svd} 
			\Delta(\mathbf{A},\mathbf{b},R)=\mathbf{b}^{\tp} \mathbf{A}((\mathbf{A}^{\tp} \mathbf{A})^{\dag})^2\mathbf{A}^{\tp} \mathbf{b}-R^2=\norm{(\mathbf{A}^{\tp} \mathbf{A})^{\dag}\mathbf{A}^{\tp} \mathbf{b}}^2-R^2.
		\end{equation} 
		\begin{lemma}\label{lem:holder-nls-exponent}
			Let $\mathbb{U}_0:=\{(\mathbf{A},\mathbf{b},R)\in\mathbb{U}:R>0\}$. Define 
			\begin{equation*}
				\begin{array}{rl}
					F:\mathbb{U}_0\times \mathbb{R}^n&\longrightarrow \mathbb{R}^n\\
					((\mathbf{A},\mathbf{b},R),\mathbf{x})&\mapsto R\mathbf{A}^{\tp}(R\mathbf{Ax}-\mathbf{b}).
				\end{array}
			\end{equation*} 
			Given $\mathbf p_0=(\mathbf{A}_0,\mathbf{b}_0,R_0)\in\mathbb{U}_0$,  the set 
			\begin{equation}\label{eq:set-x0}
				\mathcal{X}_0:=\{\mathbf{x}\in \mathbb{R}^n: \mathbf{0}\in R_0\mathbf{A}_0^{\tp}(R_0\mathbf{A}_0\mathbf{x}-\mathbf{b}_0)+N_{\mathcal{B}}(\mathbf{x})\}=\{\mathbf{x}\in \mathbb{R}^n:\mathbf{0}\in F(\mathbf p_0,\mathbf{x})+N_{\mathcal{B}}(\mathbf{x})\}
			\end{equation}
			is nonempty, compact and convex. For any $\mathbf{x}_0\in \mathcal{X}_0,\theta>0$, define $\mathcal{X}_{\theta}:=\mathcal{X}_0+\theta\mathcal{B}$, 
			\begin{equation}\label{eq:linear}
				\mathcal{L}_{\mathbf{x}_0}(\mathbf{x}):=F(\mathbf p_0,\mathbf{x}_0)+\dfrac{\partial F}{\partial \mathbf{x}}(\mathbf p_0,\mathbf{x}_0)(\mathbf{x}-\mathbf{x}_0)+N_{\mathcal{B}}(\mathbf{x})=F(\mathbf p_0,\mathbf{x})+N_{\mathcal{B}}(\mathbf{x}), 
			\end{equation}
			and $\mathcal{IL}_{\mathbf{x}_0,\theta}(\mathbf{y}):=\mathcal{L}_{\mathbf{x}_0}^{-1}(\mathbf{y})\bigcap \mathcal{X}_{\theta}$. Then there exists $\lambda,\gamma,\eta>0$ such that for every $\mathbf{x}_0\in \mathcal{X}_0$, the conditions \ref{item:ass_1}-\ref{item:ass_4} in Theorem~\ref{thm:holder-ge} hold with $\alpha$ chosen as follows:
			\begin{equation}\label{eq: choice of exponent}
				\alpha=\begin{cases}
					1&\Delta(\mathbf{A}_0,\mathbf{b}_0,R_0)>0,\\
					\frac{1}{2}& \Delta(\mathbf{A}_0,\mathbf{b}_0,R_0)<0,\\
					\frac{1}{3}&  \Delta(\mathbf{A}_0,\mathbf{b}_0,R_0)=0.
				\end{cases}
			\end{equation}  
		\end{lemma} 
		
		\begin{proof}{Proof}
			By \eqref{eq:opt-val-solu-scal}, $\mathcal{X}_0=\tilde{\mathcal{S}}(\mathbf{A}_0,\mathbf{b}_0,R_0)=\frac{1}{R_0}\mathcal{S}(\mathbf{A}_0,\mathbf{b}_0,R_0)$, which is nonempty, convex and compact by Proposition~\ref{prop:optimal-value-solution-mapping}. 			
			In the following, take any $\mathbf x_0\in \mathcal X_0$, we examine assumptions \ref{item:ass_1}-\ref{item:ass_4} in Theorem~\ref{thm:holder-ge}.  
			\begin{enumerate}[label=(\arabic*)]
				\item For \ref{item:ass_1}, by \eqref{eq:linear}, the definition of $\mathcal{IL}_{\mathbf{x}_0,\gamma}(\mathbf{0})$ and $\mathcal{X}_0$, for any $\gamma>0$ it holds
				\[
				\mathcal{IL}_{\mathbf{x}_0,\gamma}(\mathbf{0})=\{\mathbf{x}\in\mathcal{X}_\gamma: \mathbf{x}\in \mathcal{L}_{\mathbf{x}_0}^{-1}(\mathbf{0})\}=\{\mathbf{x}\in \mathcal{X}_\gamma: \mathbf{0}\in F(\mathbf p_0,\mathbf{x})+N_{\mathcal{B}}(\mathbf{x})\}=\mathcal{X}_0.
				\]
				\item Note that by \eqref{eq:linear} and \eqref{eq:gen-eqn-gen}, for any $\mathbf{y}\in\mathbb{R}^n$, 
				\begin{multline*}\label{eq:def of L,1}
					\mathcal{L}_{\mathbf{x}_0}^{-1}(\mathbf{y})=\{\mathbf{x}\in \mathbb{R}^n:\mathbf{y}\in \mathcal{L}_{\mathbf{x}_0}(\mathbf x)\}	=\{\mathbf{x}\in\mathbb{R}^n:\mathbf{0}\in-\mathbf{y}+R_0\mathbf{A}_0^{\tp}(R_0\mathbf{A}_0\mathbf{x}-\mathbf{b}_0)+N_{\mathcal{B}}(\mathbf{x})\}\\
					=\argmin_{\mathbf{x}\in\mathcal{B}}\frac{1}{2}\norm{R_0\mathbf{A}_0\mathbf{x}-\mathbf{b}_0}^2-\inner{\mathbf{x},\mathbf{y}}.		
				\end{multline*}
				Therefore by Proposition~\ref{prop:holder-calm} (with  $\mathbf{A}_0$ replaced by $R_0\mathbf{A}_0$), there exists some $\lambda>0$ such that $\mathcal{L}^{-1}_{\mathbf{x}_0}\in \mathcal{H}_{\mathbf{0}}(\alpha,\lambda)$ with $\alpha$ chosen as~\eqref{eq: choice of exponent} and $\mathcal{L}_{\mathbf{x}_0}^{-1}(\mathbf{y})\neq \emptyset$.  Thus there exists $ \eta>0$ such that for any $\gamma>0$, $\mathcal{IL}_{\mathbf{x}_0,\gamma}=\mathcal{L}^{-1}_{\mathbf{x}_0}(\cdot)\cap\mathcal X_\gamma\in\mathcal{H}_{\mathbf{0},\eta\mathcal{B}}(\alpha,\lambda)$. So \ref{item:ass_2} holds for $\alpha$ chosen as~\eqref{eq: choice of exponent}.
				\item 
				Since $\mathcal{L}^{-1}_{\mathbf{x}_0}(\mathbf 0)=\mathcal X_0$ is compact and $\mathcal{L}^{-1}_{\mathbf{x}_0}\in \mathcal{H}_{\mathbf{0}}(\alpha,\lambda)$, then by Definition~\ref{def:upper-hold}, there exist $\gamma>0$ such that for any $\mathbf{y}\in\eta\mathcal B$, $\mathcal{IL}_{\mathbf{x}_0,\gamma}(\mathbf{y})=\mathcal{L}^{-1}_{\mathbf{x}_0}(\mathbf y)\cap\mathcal X_\gamma\neq \emptyset$. The convexity follows from the fact that both sets of the intersection are convex. Therefore \ref{item:ass_3} holds.
				
				\item Since $\frac{\partial F}{\partial \mathbf{x}}(\mathbf p_0,\cdot)=R_0^2\mathbf{A}_0^{\tp}\mathbf{A}_0$ is constant, then \ref{item:ass_4} obviously holds. \red{}	 
			\end{enumerate}
		\end{proof}

		Combining Theorem~\ref{thm:holder-ge} and Lemma~\ref{lem:holder-nls-exponent}, we get one of our main result on the upper H\"olderian of the solution mapping $\mathcal{S}$ defined by~\eqref{eq: solution mapping} of  the problem~\ref{eq:ball-nls}. 			Given $\mathbf p_0=(\mathbf{A}_0,\mathbf{b}_0,R_0)\in \mathbb{U}$ and $\delta>0$, define 
		\begin{equation}\label{eq: local solution mapping}
			\mathcal{S}_{\delta}(\mathbf p):=(\mathcal{S}(\mathbf p_0)+\delta\mathcal{B})\cap \mathcal{S}(\mathbf p)\ \quad \forall \mathbf p\in \mathbb{U},
		\end{equation}
		where $\mathcal{S}$ is defined as~\eqref{eq: solution mapping}.

		\begin{theorem}[Upper H\"olderian of Ball-LS]\label{thm:upper-holder-nls}
			For any $\mathbf p_0:=(\mathbf{A}_0,\mathbf{b}_0,R_0)\in\mathbb{U}$,
			there exist $\delta>0$, $\bar\lambda>0$ and a neighborhood $\mathcal{W}$ of $\mathbf p_0$ such that $\mathcal{S}_\delta(\mathbf p)\neq \emptyset$ for all $\mathbf p\in\mathcal W$, and  
			\begin{equation*}\label{eq: locally upper semicontinuity of solution mapping of the core problem}
				\mathcal{S}_{\delta}\in\mathcal{H}_{\mathbf p_0,\mathcal{W}}(\alpha, \bar\lambda),
			\end{equation*}
			where the exponent $\alpha$ can be chosen as follows
			\begin{equation}\label{eq: Holder exponent}
				\alpha=\begin{cases}
					1&R_0=0 \text{ or }\Delta(\mathbf p_0)>0,\\
					\frac{1}{2}& R_0> 0 \text{ and } \Delta(\mathbf p_0)<0,\\
					\frac{1}{3}& R_0> 0 \text{ and } \Delta(\mathbf p_0)=0.
				\end{cases}
			\end{equation}
		\end{theorem}
		
		\begin{proof}{Proof}	
			When $R_0=0$, it is clear that $\mathcal{S}(\mathbf p_0)=\{\mathbf{0}\}$. Choose an arbitrary $\delta>0$ and $\mathcal{W}:=\{\mathbf p=(\mathbf{A},\mathbf{b},R)\in \mathbb{U}:|R-R_0|<\frac{\delta}{2}\}$, then by the definition of $\mathcal{S}$ (cf.\ \eqref{eq: solution mapping}) and Proposition~\ref{prop:optimal-value-solution-mapping}, for any $\mathbf p\in \mathcal{W}$, $\mathcal{S}_\delta(\mathbf p)=\delta\mathcal{B}\bigcap\mathcal{S}(\mathbf p)=\mathcal{S}(\mathbf p)\neq \emptyset$. On the other hand, for any $\mathbf p=(\mathbf{A},\mathbf{b},R)\in\mathcal{W}$, $\norm{\mathbf p-\mathbf p_0}=(\norm{\mathbf{A}-\mathbf{A}_0}^2+\norm{\mathbf{b}-\mathbf{b}_0}^2+R^2)^{\frac{1}{2}}\geq R$. And thus by Proposition~\ref{prop:optimal-value-solution-mapping}, 
			\begin{equation*}
				\mathcal{S}_\delta(\mathbf p)=\mathcal{S}(\mathbf p)\subseteq R\mathcal{B}\subseteq \norm{\mathbf p-\mathbf p_0}\mathcal{B}=\mathcal{S}(\mathbf p_0)+ \norm{\mathbf p-\mathbf p_0}\mathcal{B}.
			\end{equation*} Therefore, the assertion holds.

			In the following, we assume that $R_0>0$. 
			The rest conclusions follow from an application of Theorem~\ref{thm:holder-ge} to \eqref{eq:opt-cond-scal}. Therefore hereinafter, we will examine the corresponding hypotheses in Theorem~\ref{thm:holder-ge}. 
			
			The assertion that $F$ is continuously differentiable is obvious. Since $\mathcal{B}$ is closed and convex, then the normal cone mapping $N_{\mathcal{B}}$ is outer semi-continuous on $\mathcal{B}$ \cite[Proposition~6.6]{rockafellar2009variational}.

			By Theorem~\ref{thm:holder-ge} and Lemma~\ref{lem:holder-nls-exponent}, there exists $\tilde{\delta},\lambda>0$,  such that for any $\delta<\tilde{\delta}$, there exists  a neighborhood $\mathcal{W}_{\delta}$ of $\mathbf p_0$ such that for any $\mathbf p=(\mathbf{A},\mathbf{b},R)\in\mathcal{W}_\delta$, we have $R\neq 0$, 
			\begin{equation}\label{eq:set-nonempty-scal}
				\tilde{\mathcal{S}}_\delta(\mathbf p):=\tilde{\mathcal S}(\mathbf p)\cap (\tilde{S}(\mathbf{p}_0)+\mathcal X_\delta)\neq \emptyset,\ \text{and }
				\tilde{\mathcal{S}}_\delta(\mathbf p)\subseteq \tilde{\mathcal{S}}_\delta(\mathbf p_0)+\lambda\mu_{\delta}(\mathbf p)^{\alpha}\mathcal{B},
			\end{equation}
			since $\mu_0(\mathbf p)\leq\mu_\delta(\mathbf p)$. Here $\alpha$ is given by \eqref{eq: Holder exponent}, and 
			\[
			\mu_{\delta}(\mathbf p):= 
			\max_{\mathbf x\in \mathcal{X}_\delta}\|\Psi(\mathbf p;\mathbf{x})-\Psi(\mathbf p_0;\mathbf{x})\|,
			\]
			where $\mathcal X_{\delta}=\mathcal X_0+\delta\mathcal B$ and $\mathcal X_0$ is given as \eqref{eq:set-x0} and $\Psi: (\mathbf p;\mathbf{x})=(\mathbf{A},\mathbf{b},R;\mathbf{x})\mapsto R\mathbf{A}^{\tp}(R\mathbf{A}\mathbf{x}-\mathbf{b})$. 
			
			Let $L>\max\{R_0,\frac{1}{3}\}$ be a constant. For any $\delta<\tilde{\delta}$, let $\delta'=\frac{1}{3L}\delta<\tilde{\delta}$ and $\mathcal{W}\subseteq \mathcal{W}_{\delta'}$ a neighborhood of $\mathbf p_0$ such that $\norm{\mathbf p-\mathbf p_0}<\min\{\delta,1,\frac{\delta}{3L}\}$ and $R<L$ for any $\mathbf p=(\mathbf{A},\mathbf{b},R)\in\mathcal{W}$. 
			Since the mapping $\Psi$ is smooth with respect to $\mathbf p$, then by $(3)$ in Remark~\ref{rmk:inverse_mapping_theorem_Holder_case}, the map $\mathbf p\mapsto \mu_{\delta}(\mathbf p)$ is locally upper Lipschitz continuous at $\mathbf p_0$.
			Thus, as $\mu_\delta(\mathbf p_0)=0$, we have that for any $\mathbf p\in\mathcal{W}$,
			\[
			\tilde{\mathcal{S}}_\delta(\mathbf p)\subseteq \tilde{\mathcal{S}}_\delta(\mathbf p_0)+\tilde\lambda\norm{\mathbf p-\mathbf p_0}^{\alpha}\mathcal{B}
			\]
			for some $\tilde\lambda>0$.

			Given $\mathbf{p}=(\mathbf{A},\mathbf{b},R)\in\mathcal{W}$, by \eqref{eq:set-nonempty-scal}, we know that $\tilde{S}_{\delta'}(\mathbf{p})\neq \emptyset$. By the definition of $\tilde{S}_{\delta'}$, there exists $\tilde{\mathbf{x}}\in \tilde{S}(\mathbf{p})$ and $\tilde{\mathbf{x}}_0\in\tilde{S}(\mathbf{p}_0)$, such that $\norm{\tilde{\mathbf{x}}-\tilde{\mathbf{x}}_0}\leq \delta'=\frac{\delta}{3L}$. By~\eqref{eq:opt-val-solu-scal}, we get that $\mathbf{x}:=R\tilde{\mathbf{x}}\in\mathcal{S}(\mathbf{p})$ and $\mathbf{x}_0:=R_0\tilde{\mathbf{x}}_0\in\mathcal{S}(\mathbf{p}_0)$. In addition, by the choice of $L,\delta$ and $\mathcal{W}$ and note that $\tilde{\mathcal{S}}(\mathbf{p})\subseteq\mathcal{B}$, we get that
			\begin{equation*} 
				\norm{\mathbf{x}-\mathbf{x}_0}=\norm{R\mathbf{x}-R_0\mathbf{x}_0}\leq R\norm{\tilde{\mathbf{x}}-\tilde{\mathbf{x}}_0}+|R-R_0|\norm{\mathbf{x}_0} 
				\leq L\cdot \frac{\delta}{3L}+ \frac{\delta}{3L}\cdot L<\delta. 
			\end{equation*}
			Therefore we get that $\mathbf{x}\in \mathcal{S}(\mathbf{p}_0)+\delta\mathcal{B}$. So $\mathbf{x}\in\mathcal{S}(\mathbf{p})\cap (\mathcal{S}(\mathbf{p}_0)+\delta\mathcal{B})= \mathcal{S}_\delta(\mathbf{p})$. Thus $\mathcal{S}_{\delta}(\mathbf{p})\neq \emptyset$.
			
			Moreover, we have 
			\begin{equation*}
				\begin{aligned}
					\mathcal{S}_{\delta}(\mathbf p)&=R\tilde{\mathcal{S}}_{\delta}(\mathbf p)\subseteq R\tilde{\mathcal{S}}_\delta(\mathbf p_0)+ R\tilde\lambda\norm{\mathbf p-\mathbf p_0}^{\alpha}\mathcal{B}\\
					&\subseteq R_0\tilde{\mathcal{S}}_\delta(\mathbf p_0)+(R-R_0)\tilde{\mathcal{S}}_\delta(\mathbf p_0)+ L\tilde\lambda\norm{\mathbf p-\mathbf p_0}^{\alpha}\mathcal{B}\\
					&\subseteq R_0\tilde{\mathcal{S}}_\delta(\mathbf p_0)+\norm{\mathbf p-\mathbf p_0}\mathcal{B}+L\tilde\lambda\norm{\mathbf p-\mathbf p_0}^{\alpha}\mathcal{B}\\
					&\subseteq \mathcal{S}_\delta (\mathbf p_0)+\bar\lambda\norm{\mathbf p-\mathbf p_0}^{\alpha}\mathcal{B},
				\end{aligned}
			\end{equation*}
			where $\bar\lambda:=L\tilde{\lambda}+1$, and the last inclusion follows from the fact that $\alpha\leq 1$ in~\eqref{eq: Holder exponent}. Consequently, the assertion holds.
		\end{proof}
		\subsection{Separable NLS}\label{sec:snls}
		In this subsection, we present an application of Theorem~\ref{thm:upper-holder-nls} for separable nonlinear least squares (NLS) problems with ball constraint:
		\begin{equation}\label{eq:gen_ball_nls}
			\begin{array}{rl}
				\min_{\mathbf{x}}&\frac{1}{2}\norm{G(\mathbf{y})\mathbf{x}-\mathbf{b}}^2\\
				\text{s.t. }& \norm{\mathbf x}^2\leq R^2,
			\end{array}
		\end{equation}
		where $G:\Omega\subseteq\mathbb{R}^l\longrightarrow \mathbb{R}^{m\times n}$ is a matrix-valued mapping. Define $\tilde{G}:\Omega\times \mathbb{R}^m\times \mathbb{R}_{++} \longrightarrow \mathbb{U}$ as $\tilde{G}(\mathbf{y},\mathbf{b},R)=(G(\mathbf{y}),\mathbf{b},R)$.  
		Separable NLS problems are extensively studied in \cite{golub1973differentiation} and references herein. 
		
		Define the solution mapping of \eqref{eq:gen_ball_nls} as 
		\begin{equation}\label{eq:sol_map_gen_nls}
			\begin{array}{rl}
				\mathcal{S}^G:\Omega\times \mathbb{R}^m\times \mathbb{R}_{++}&\longrightarrow \mathbb{R}^n,\\
				(\mathbf{y},\mathbf{b},R)&\mapsto \argmin_{\norm{\mathbf x}^2\leq R^2}\frac{1}{2}\norm{G(\mathbf{y})\mathbf{x}-\mathbf{b}}^2.
			\end{array}
		\end{equation}
		Numerical methods for \eqref{eq:gen_ball_nls} usually try to eliminate $\mathbf x$ first ($\mathbf x$-pharse) and then solve the result problem for $\mathbf y$ ($\mathbf y$-pharse). In this scheme, properties on the solution mapping $\mathcal{S}^G$ play crucial roles in algorithmic designs for the $\mathbf y$-pharse. 
		
		Given $\xi_0=(\mathbf{y}_0,\mathbf{b}_0,R_0)\in \Omega\times \mathbb{R}^m\times \mathbb{R}_{++}$ and $\delta>0$, define 
		\begin{equation*}
			\mathcal{S}_{\delta}^G(\xi):=(\mathcal{S}^G(\xi_0)+\delta\mathcal{B})\cap \mathcal{S}^G(\xi),\quad \forall \xi\in \Omega\times \mathbb{R}^m\times \mathbb{R}_{++}.
		\end{equation*}
		Clearly, 
		\begin{equation}\label{eq:relation_of_nls_and_gen_nls}
			\mathcal{S}^G=\mathcal{S}\circ \tilde{G},\ \text{and } \mathcal{S}^G_\delta=\mathcal{S}_\delta\circ \tilde{G},
		\end{equation}
		where $\mathcal{S}$ and $\mathcal{S}_\delta$ are defined as \eqref{eq: solution mapping} and \eqref{eq: local solution mapping}, respectively.
		\begin{proposition}\label{prop:Holder_for_gen_nls}
			Let separable NLS~\eqref{eq:gen_ball_nls} be defined as above. For any $\xi_0:=(\mathbf{y}_0,\mathbf{b}_0,R_0)\in \Omega\times \mathbb{R}^m\times \mathbb{R}_{++}$, suppose $G$ is locally upper H\"older continuous at $\mathbf{y}_0$ with exponent $\beta\in(0,1]$, then there exists $\delta>0$, $\lambda>0$ and a neighborhood $\mathcal{N}$ of $\xi_0$ relative to  $\Omega\times \mathbb{R}^m\times \mathbb{R}_{++}$, such that $\mathcal{S}^G_\delta(\xi)\neq \emptyset$ for any $\xi\in \mathcal{N}$, and 
			\begin{equation*}
				\mathcal{S}^G_\delta\in \mathcal{H}_{\xi_0,\mathcal{N}}(\alpha\beta,\lambda),
			\end{equation*}
			where the exponent $\alpha$ is given as \eqref{eq: Holder exponent} with $\mathbf p_0=\tilde{G}(\xi_0)$.
		\end{proposition}
		\begin{proof}{Proof}
			Denote $\mathbf p_0=\tilde{G}(\xi_0)\in\mathbb{U}$.
			By Theorem~\ref{thm:upper-holder-nls}, there exists $\delta>0$, $\bar\lambda>0$ and a neighborhood $\mathcal{W}$ of $\mathbf p_0$ such that $\mathcal{S}_\delta(\mathbf p)\neq \emptyset$ for any $\mathbf p\in\mathcal{W}$ and  
			\begin{equation*}
				\mathcal{S}_{\delta}\in\mathcal{H}_{\mathbf p_0,\mathcal{W}}(\alpha, \bar\lambda),
			\end{equation*}
			where the exponent $\alpha$ is given as \eqref{eq: Holder exponent}. Thus, for any $\mathbf p\in \mathcal{W}$, it holds
			\begin{equation}\label{eq:def_of_Hold_for_S}
				\mathcal{S}_{\delta}(\mathbf p)\subseteq \mathcal{S}_\delta(\mathbf p_0)+\bar\lambda\norm{\mathbf p-\mathbf p_0}^\alpha\mathcal{B}.
			\end{equation}
			
			Since $G$ is locally upper H\"older continuous at $\mathbf{y}_0$ with exponent $\beta>0$, then by the definition of product topology, there exist a neighborhood $\mathcal{N}$ of $\xi_0$ relative to $\Omega\times \mathbb{R}^m\times \mathbb{R}_{++}$ and \red{$L>1$} such that $\tilde{G}(\mathcal{N})\subset \mathcal{W}$ and for any $\xi=(\mathbf{y},\mathbf{b},R)\in\mathcal{N}$ it holds
			\begin{equation}\label{eq:Holder_for_par_map}
				\begin{aligned}
					\norm{\tilde{G}(\xi)-\tilde{G}(\xi_0)}&=\norm{(G(\mathbf{y})-G(\mathbf{y}_0),\mathbf b-\mathbf b_0,R-R_0)}\\&\leq \norm{\mathbf b-\mathbf b_0}+|R-R_0|+L\norm{\mathbf{y}-\mathbf{y}_0}^\beta\\&\leq 3L\norm{\xi-\xi_0}^\beta,
				\end{aligned}
			\end{equation}
			where the last inequality follows from $\beta \in (0,1]$.
			These, together with \eqref{eq:relation_of_nls_and_gen_nls}, imply that $\mathcal{S}^G_\delta(\xi)\neq \emptyset$ for any $\xi\in \mathcal{N}$. 
			Moreover, for any $\xi=(\mathbf{y},\mathbf{b},R)\in\mathcal{N}$, \eqref{eq:Holder_for_par_map}, together with \eqref{eq:def_of_Hold_for_S} and \eqref{eq:relation_of_nls_and_gen_nls}, implies that 
			\begin{equation*}
				\mathcal{S}^G_{\delta}(\xi)=\mathcal{S}_\delta\circ\tilde{G}(\xi)\subseteq \mathcal{S}_\delta(\tilde{G}(\xi_0))+\bar\lambda\norm{\tilde{G}(\xi)-\tilde{G}(\xi_0)}^\alpha\mathcal{B}\subseteq \mathcal{S}^G_{\delta}(\xi_0)+3L\bar\lambda\norm{\xi-\xi_0}^{\alpha\beta}\mathcal{B}.
			\end{equation*}
			Let $\lambda=3L\bar\lambda$, the assertion follows.
			
		\end{proof} 
		
		Proposition~\ref{prop:Holder_for_gen_nls} gives an upper H\"olderian for the separable NLS. To solve the separable NLS \eqref{eq:gen_ball_nls}, a selection of the set-valued solution mapping with favorable properties is desireable. It sheds some light on solving separable NLS when $G$ fails to have local constant rank. Those discussions will be carried out in our upcoming  work \cite{HuWang2026variantionalgeometryofoptimizations}. 
		
		\section{Conclusion}\label{sec:con}  
		In this paper, we analyze the upper H\"olderian of the  solution mapping of parameteric generalized equations and its applications to that for $\mathcal{S}$~\eqref{eq: solution mapping} of the Ball-LS problem~\ref{eq:ball-nls}. 
		
		We first present an extension of Robinson's generalized implicit function theorem from the upper Lipschitzian case  (cf.\  \cite{Robinson1979_GenEqI}) to the upper H\"olderian case (cf.\  Theorem~\ref{thm:holder-ge}). 
		It is worth mentioning that given the widespread impact of  Robinson's generalized implicit function theorem for generalized equations under upper Lipschitz continuity, the extension to the case under upper H\"older continuity assumption may be applied for complicated polynomial systems for which only H\"older continuous can be expected \cite{mordukhovich2014full,gfrerer16lipschitz}. Then, based on this generalization, we further provide explicit exponents for the upper H\"olderian of the solution mapping (cf.\  Theorem~\ref{thm:upper-holder-nls}) for linear least squares problem with ball constraint. It has potential applications in convergence rate analysis for related numerical schemes.

		A next study is on extensions to general quadratic constrained quadratic programs (QCQP) and optimization problems on classical manifolds. These problems appear very often as subproblems or sequential approaches for more complicated models, for which variational geometry on the optimal value mapping and the optimal solution mapping are critical for both algorithmic design and convergence analysis. In particular, when standard regularization conditons fail, the corresponding landscape is less known \cite{rockafellar2009variational,Rockafellar2014implicitfunctionsandsolutionmapping,Shapiro2000perturbationanalysis}.

		%


		\subsection*{Acknowledgement}
		
		This work is partially supported by the Natural Science Foundation of Hunan Province of China (Grant No. 2025JJ20006), the National Science Foundation of China (Grant No. 12571334), and the Innovation Research Foundation of National University of Defense Technology.

		
		\bibliographystyle{plain}
		\bibliography{ref}

@book{stewart1990matrix,
  title={Matrix Perturbation Theory},
  author={Stewart, Gilbert W and Sun, Ji-Guang},
  publisher = {Academic Press}, 
  address={New York},
  year={1990}
}

@book{apostol1974mathematical,
  author    = {Apostol, Tom M.},
  title     = {Mathematical Analysis},
  publisher = {Addison-Wesley},
  year      = {1974},
  address   = {Reading, Massachusetts}
}

@book{bauschke2011convex,
  title={Convex Analysis and Monotone Operator Theory in Hilbert Spaces},
  author={Bauschke, Heinz H. and Combettes, Patrick L.},
  series={CMS Books in Mathematics},
  year={2011},
  publisher={Springer},
  address={New York}
}

@book{bhatia2013matrix,
  title     = {Matrix Analysis},
  author    = {Bhatia, Rajendra},
  series    = {Graduate Texts in Mathematics},
  volume    = {169},
  publisher = {Springer},
  year      = {1997},
  address   = {New York},
  pages     = {xi, 347},
  mrclass   = {15-01 (47-01)},
  mrnumber  = {1478117},
  zbmath    = {0863.15001}
}

@book{clarke1990optimization,
  title={Optimization and Nonsmooth Analysis},
  author={Clarke, Frank H},
  year={1990},
  publisher={Society for Industrial and Applied Mathematics},
  address={New York}
}

@book{Demmel1997appliednumericallnearalgebra,
author = {Demmel, James W.},
title = {Applied Numerical Linear Algebra},
publisher = {Society for Industrial and Applied Mathematics},
year = {1997},
address = {New York},
edition   = {},
}

@book{dhara2012optimality,
  title     = {Optimality Conditions in Convex Optimization: A Finite-Dimensional View},
  author    = {Dhara, Anulekha and Dutta, Joydeep},
  publisher = {CRC Press, Taylor \& Francis Group},
  address   = {Boca Raton, Florida},
  year      = {2012},
  mrclass   = {90C25, 49N15, 46N10},
  mrnumber  = {MR2931423}
}

@book{Rockafellar2014implicitfunctionsandsolutionmapping,
author = {Dontchev, Asen L. and Rockafellar, R. T. },
title = {Implicit Functions and Solution Mappings: A View from Variational Analysis},
publisher = {Springer, New York},
year = {2013},
address = {},
edition   = {},
}

@book{Fiorenza2016Holder,
  author    = {Fiorenza, Renato},
  title     = {H{\"o}lder and Locally H{\"o}lder Continuous Functions, and Open Sets of Class $C^k$, $C^{k,\lambda}$},
  series    = {Frontiers in Mathematics},
  publisher = {Springer},
  address   = {Cham},
  year      = 2016,
  pages     = {xi+152},
}

@book{krantz2013implicit,
  title={The Implicit Function Theorem: History, Theory, and Applications},
  author={Krantz, Steven G. and Parks, Harold R.},
  series={Birkh{\"a}user Basel},
  year={2013},
  publisher={Springer},
  address={New York},
}

@book{mordukhovich1996variationalI,
author = {Boris S. Mordukhovich},
title = {Variational Analysis and Generalized Differentiation I: Basic Theory},
publisher = {Springer Berlin, Heidelberg},
year = {2006},
address = {},
edition   = {}, 
}

@book{JRMunkres1991,
	AUTHOR = {James R. Munkres},
	TITLE = {Analysis on Manifolds},
	 publisher = {CRC Press, Taylor \& Francis Group},
  address   = {Boca Raton, Florida},
	YEAR = {1991}
}

@book{rao1995linear,
    author = {Rao, C. Radhakrishna and Toutenburg, Helge},
    title = {Linear Models: Least Squares and Alternatives},
    series = {Springer Series in Statistics},
    publisher = {Springer},
    address = {New York},
    year = {1995},
}

@book{rockafellar-1970a,
	AUTHOR = {Rockafellar, R. Tyrrell},
	TITLE = {Convex Analysis},
	PUBLISHER = {Princeton University Press, Princeton},
	YEAR = {1997},
	PAGES = {xviii+451},
	MRCLASS = {49-02 (26-02 46-02 58-02 90-02)},
	MRNUMBER = {1451876},
}

@book{rockafellar2009variational,
    title     = {Variational Analysis},
    author    = {Rockafellar, R Tyrrell and Wets, Roger J-B},
    series ={Grundlehren der mathematischen Wissenschaften},
    volume    = {317},
    year      = {2009},
    publisher = {Springer Science \& Business Media},
    address={Berlin , Heidelberg},
}

@book{rudin1976principles,
  author    = {Rudin, Walter},
  title     = {Principles of Mathematical Analysis},
  publisher = {McGraw-Hill},
  year      = {1976},
  address   = {New York},
}

@book{Shapiro2000perturbationanalysis,
  author    = {Bonnans, Joseph Fr{\'e}d{\'e}ric and Shapiro, Alexander},
  title     = {Perturbation Analysis of Optimization Problems},
  publisher = {Springer-Verlag},
  address   = {New York},
  year      = {2000},
  series    = {Springer Series in Operations Research and Financial Engineering}, 
}

@book {NY2018,
	AUTHOR = {Nesterov, Yurii},
	TITLE = {Lectures on Convex Optimization},
	PUBLISHER = {Springer, Cham},
	YEAR = {2018},
	PAGES = {xxiii+589},
	MRCLASS = {90-01 (90C25)},
	MRNUMBER = {3839649},
	MRREVIEWER = {Giorgio\ Giorgi},
}

@article{Shapiro1997,
  author  = {Shapiro, Alexander},
  title   = {First and second order analysis of nonlinear semidefinite programs},
  journal = {Mathematical Programming},
  year    = {1997},
  volume  = {77},
  number  = {1},
  pages   = {301--320},
  mrclass = {90C22},
  mrnumber = {1440711},
}

@article{Berk2023Lasso,
author = {Berk, Aaron and Brugiapaglia, Simone and Hoheisel, Tim},
title = {LASSO reloaded: a variational analysis perspective with applications to compressed sensing},
journal = {SIAM Journal on Mathematics of Data Science},
volume = {5},
number = {4},
pages = {1102-1129},
year = {2023}
}

@article{canovas2025lipschitz,
  title={Lipschitz stability in convex optimization},
  author={C{\'a}novas, MJ and Parra, Juan},
  journal={Set-valued and Variational Analysis},
  volume={33},
  number={3},
  pages={25},
  year={2025},
  publisher={Springer}
}

@article{clarke1976inverse,
  author  = {Clarke, Frank H.},
  title   = {On the inverse function theorem},
  journal = {Pacific Journal of Mathematics},
  volume  = {64},
  number  = {1},
  pages   = {97--102},
  year    = {1976}
}

@article{cui2024lipschitzstabilityleastsquaresproblems,
  author = {Cui, Ying and Hoheisel, Tim and Nghia, Tran T. A. and Sun, Defeng},
  title = {Lipschitz stability of least-squares problems regularized by functions with {C}$^2$-Cone reducible conjugates},
  journal = {Mathematics of Operations Research},
  year = {2026}, 
  pages={in press}
}

@article{dontchev1993lipschitzian,
  title={Lipschitzian stability in nonlinear control and optimization},
  author={Dontchev, Asen L and Hager, William W},
  journal={SIAM Journal on Control and Optimization},
  volume={31},
  number={3},
  pages={569--603},
  year={1993},
  publisher={Society for Industrial and Applied Mathematics}
}

@article{gfrerer16lipschitz,
  title={Lipschitz and {H}{\"o}lder stability of optimization problems and generalized equations},
  author={Gfrerer, Helmut and Klatte, Diethard},
  journal={Mathematical Programming},
  volume={158},
  number={1},
  pages={35--75},
  year={2016},
  publisher={Springer}
}

@article {he-26-tensor,
	AUTHOR = {He, Tiantian and Hu, Shenglong and Huang, Zheng-Hai},
	TITLE = {A Projected {Gauss-Newton} Variable Projection
Method for Low Rank Approximations of third order Tensors},
	JOURNAL = {Submitted Manuscript}, 
	YEAR = {2026}
}

@article{golub1973differentiation,
  title={The differentiation of pseudo-inverses and nonlinear least squares problems whose variables separate},
  author={Golub, Gene H and Pereyra, Victor},
  journal={SIAM Journal on Numerical Analysis},
  volume={10},
  number={2},
  pages={413--432},
  year={1973},
  publisher={Society for Industrial and Applied Mathematics}
}

@incollection{golub1976differentiation,
  title={Differentiation of pseudoinverses, separable nonlinear least square problems and other tales},
  author={Golub, GH and Pereysa, V},
  booktitle={Generalized Inverses and Applications},
  pages={303--324},
  year={1976},
  publisher={Elsevier}
}

@article{gulliksson2002perturbation,
  title={Perturbation bounds for constrained and weighted least squares problems},
  author={Gulliksson, M{\aa}rten and Jin, Xiao-Qing and Wei, Yi-Min},
  journal={Linear Algebra and its Applications},
  volume={349},
  number={1-3},
  pages={221--232},
  year={2002},
  publisher={Elsevier}
}

@article{izmailov2013note,
  title={A note on upper {L}ipschitz stability, error bounds, and critical multipliers for Lipschitz-continuous {KKT} systems},
  author={Izmailov, Alexey F and Kurennoy, Alexey S and Solodov, Mikhail V},
  journal={Mathematical Programming},
  volume={142},
  number={1},
  pages={591--604},
  year={2013},
  publisher={Springer}
}

@article{JiangL22Holder,
  author    = {Rujun Jiang and Xudong Li},
  title     = {H{\"{o}}lderian error bounds and {K}urdyka-{L}ojasiewicz inequality for the trust region subproblem},
  journal   = {Mathematics of Operations Research},
  volume    = {47},
  number    = {4},
  pages     = {3025--3050},
  year      = {2022}
}

@article{jittorntrum1978implicit,
  author = {Jittorntrum, K.},
  title = {An implicit function theorem},
  journal = {Journal of Optimization Theory and Applications},
  volume = {25},
  pages = {575--577},
  year = {1978}
}

@article{klatte2005strong,
  title={Strong {L}ipschitz stability of stationary solutions for nonlinear programs and variational inequalities},
  author={Klatte, Diethard and Kummer, Bernd},
  journal={SIAM Journal on Optimization},
  volume={16},
  number={1},
  pages={96--119},
  year={2005},
  publisher={Society for Industrial and Applied Mathematics}
}

@article{klatte2026lipschitz,
  title={Lipschitz stability for a class of parametric optimization problems with polyhedral feasible set mapping},
  author={Klatte, Diethard},
  journal={Computational Management Science},
  volume={23},
  number={1},
  pages={4},
  year={2026},
  publisher={Springer}
}

@article{kakutani1941generalization,
  title={A generalization of {B}rouwer's fixed point theorem},
  author={Kakutani, Shizuo},
  journal={Duke Mathematical Journal},
  volume={8},
  number={3},
  pages={457--459},
  year={1941},
  month={September},
  publisher={Duke University Press},
}

@article{McDonald1995,
  author = {McDonald, John},
  title = {Fiber polytopes and fractional power series},
  journal = {Journal of Pure and Applied Algebra},
  volume = {104},
  number = {2},
  pages = {213--233},
  year = {1995},
  mrclass = {52B11 (14M25 52B05)},
  mrnumber = {1354890},
  mrreviewer = {Jürgen Herzog},
  zbl = {0842.52009}
}

@techreport{kummer1989implicit,
  author = {Kummer, B.},
  title = {An Implicit-Function Theorem for ${C}^{0,1}$-Equations and Parametric ${C}^{1,1}$-Optimization},
  type = {Preprint},
  institution = {Sektion Mathematik, Humboldt-Universitat zu Berlin},
  address = {Berlin},
  year = {1989}
}

@article{mordukhovich2014full,
  title={Full {L}ipschitzian and {H}{\"o}lderian stability in optimization with applications to mathematical programming and optimal control},
  author={Mordukhovich, Boris S and Nghia, Tran TA},
  journal={SIAM Journal on Optimization},
  volume={24},
  number={3},
  pages={1344--1381},
  year={2014},
  publisher={Society for Industrial and Applied Mathematics}
}

@article{mordukhovich1994lipschitzian,
  title={Lipschitzian stability of constraint systems and generalized equations},
  author={Mordukhovich, Boris},
  journal={Nonlinear Analysis: Theory, Methods and Applications},
  volume={22},
  number={2},
  pages={173--206},
  year={1994},
  publisher={Elsevier}
}

@article{Nghia2025geometrycharacterizationslipschitzstability,
author = {Nghia, Tran T.A.},
title = {Geometric characterizations of {L}ipschitz stability for convex optimization problems},
journal = {SIAM Journal on Optimization},
volume = {35},
number = {2},
pages = {927-958},
year = {2025},
  publisher={Society for Industrial and Applied Mathematics}
}

@article{Wang01011994,
author = {Tao. Wang and Jong-Shi. Pang},
title = {Global error bounds for convex quadratic inequality systems},
journal = {Optimization},
volume = {31},
number = {1},
pages = {1--12},
year = {1994},
publisher = {Taylor and Francis}, 
}

@article{Robinson1979_GenEqI,
  author  = {Robinson, S. M.},
  title   = {Generalized equations and their solutions, Part I: Basic theory},
  journal = {Mathematical Programming Study},
  year    = {1979},
  volume  = {10},
  pages   = {128--141},
}

@article{Robinson1994_Implicit,
  author  = {Robinson, S. M.},
  title   = {An implicit-function theorem for a class of nonsmooth functions},
  journal = {Mathematics of Operations Research},
  year    = {1994},
  volume  = {19},
  number  = {2},
  pages   = {292--309},
}

@article{HuWang2026variantionalgeometryofoptimizations,
      title={Variational landscape of the {KKT} system of quadratic programming with ball constraints}, 
      author={Yu Wang and Shenglong Hu},
      year={2026},
      journal={In preparation},
      
}



		%
		%
		%
		%
		\appendix
			\section{Basics on Ball-LS Problems}\label{sec:basiclsp}
			
			In the appendix, we establish explicit formulae for the solution mapping $\mathcal{S}$ \eqref{eq: solution mapping} and the optimal mapping $\mathcal{V}$ \eqref{eq: optimal value mapping} of the ball constrained linear least squares problem \ref{eq:ball-nls}, recalled as 
			
			\begin{equation*} 
				\begin{array}{rl}
					\min_{\mathbf{x}}&  \frac{1}{2}\norm{\mathbf{Ax}-\mathbf{b}}^2\\
					\text{s.t.}& \norm{\mathbf{x}}^2\leq R^2.
				\end{array}
			\end{equation*}
			When $R=0$, the problem is trivial. 
			For the nontrivial case when $R\neq 0$, problem~\ref{eq:ball-nls} is a convex problem satisfying the Slater condition. Consequently, strong duality holds and the global optimality is characterized by the KKT system both necessarily and sufficiently (\cite{rockafellar-1970a}). 
			
			Denote the the optimal value mapping (\cite[Section~4.1]{Shapiro2000perturbationanalysis}) of the problem~\eqref{eq:ball-nls} as:
			\begin{equation}\label{eq: optimal value mapping}
				\begin{array}{rl}
					\mathcal{V} : \mathbb{R}^{m\times n}\times \mathbb{R}^{m}\times \mathbb{R}_{+}&\longrightarrow \mathbb{R}\\
					(\mathbf{A},\mathbf{b},R)&\longmapsto \min_{\mathbf{x}\in \mathcal{B}(R)} \frac{1}{2}\norm{\mathbf{Ax}-\mathbf{b}}^2.
				\end{array}
			\end{equation}
			
			Define 
			\begin{equation}\label{eq: the Lagrange dual function}
				h(\lambda;\mathbf{A},\mathbf{b},R):=-\frac{1}{2}\mathbf{b}^{\tp} \mathbf{A} (\mathbf{A}^{\tp} \mathbf{A} + 2\lambda \mathbf{I}_n)^{\dag} \mathbf{A}^{\tp} \mathbf{b} + \frac{1}{2}\mathbf{b}^{\tp}\mathbf{b} - \lambda R^2. 
			\end{equation}
			
			\begin{proposition}[Optimal Value and Solution Mappings]\label{prop:optimal-value-solution-mapping}
				Let $\mathbf{A}\in\mathbb{R}^{m\times n},\mathbf{b}\in\mathbb{R}^m$ and $R\geq 0$. Suppose that $\mathbf{A}=\mathbf{U}\begin{bmatrix}\mathbf{\Sigma}_0&0\\0&0
				\end{bmatrix}\mathbf{V}^{\tp}$ is a singular value decomposition of $\mathbf{A}$ with $\mathbf{\Sigma}_0:=\diag(\sigma_1,\dots,\sigma_r)\in\mathbb{R}^{r\times r}$ and $\sigma_1\geq \sigma_2\geq \cdots\geq \sigma_r>0$. Denote  $\mathbf{c}:=\mathbf{U}^{\tp} \mathbf{b}=(c_1,\dots,c_m)^{\tp}$ and 
				\begin{equation}\label{eq:criteria-delta}
					\Delta(\mathbf{A},\mathbf{b},R):=\sum_{k=1}^r \frac{c_k^2}{\sigma_k^2}-R^2.
				\end{equation}
				Then 
				the optimal value mapping of problem \ref{eq:ball-nls} is given by 
				\begin{equation}\label{eq: svd form of the optimal value mapping}
					\mathcal{V}(\mathbf{A},\mathbf{b},R)=\sup_{\lambda\geq 0}h(\lambda;\mathbf{A},\mathbf{b},R)	=\begin{cases}
						h(0;\mathbf{A},\mathbf{b},R)=\frac{1}{2}\sum_{i=r+1}^m c_i^2,&\Delta(\mathbf{A},\mathbf{b},R)\leq 0\\
						h(\lambda_0;\mathbf{A},\mathbf{b},R),&\Delta(\mathbf{A},\mathbf{b},R)> 0,
					\end{cases}
				\end{equation}
				where
				$\lambda_0=\lambda_0(\mathbf{A},\mathbf{b},R)\in\mathbb R_{++}\cup\{+\infty\}$ satisfies the first order condition 
				\begin{equation}\label{eq: eq:opt-lambda 1}
					\frac{\partial h}{\partial\lambda}(\lambda_0;\mathbf{A},\mathbf{b},R)=\sum_{k=1}^r \frac{\sigma_k^2c_k^2}{(\sigma_k^2+2\lambda_0)^2}-R^2=0.
				\end{equation}
				Moreover, the solution mapping of problem~\ref{eq:ball-nls} is given by 
				\begin{equation}\label{eq: exact form for the solution mapping}
					\begin{aligned}
						\mathcal{S}(\mathbf{A},\mathbf{b},R)=\begin{cases}
							((\mathbf{A}^{\tp} \mathbf{A})^{\dag}\mathbf{A}^{\tp} \mathbf{b}+\ker(\mathbf{A}))\bigcap \mathcal{B}(R),&\Delta(\mathbf{A},\mathbf{b},R) \leq 0,\\
							\{(\mathbf{A}^{\tp} \mathbf{A}+2\lambda_0(\mathbf{A},\mathbf{b},R) \mathbf{I}_n)^{-1}\mathbf{A}^{\tp} \mathbf{b}\},&\Delta(\mathbf{A},\mathbf{b},R) > 0.
						\end{cases}
					\end{aligned}
				\end{equation}
			\end{proposition}
			
			\begin{proof}{Proof}
				For the case when $R=0$, the convention $\frac{\alpha}{\infty}=0$ is employed for any $\alpha\in\mathbb R$. If $r=0$, the convention $\sum \emptyset=0$ is employed. 
				In the following, we assume that both $R>0$ and $r>0$.
				
				The Lagrangian function of problem~\ref{eq:ball-nls} is 
				\[	
				L(\mathbf{x},\lambda;\mathbf{A},\mathbf{b},R)= \frac{1}{2}\norm{\mathbf{Ax} - \mathbf{b}}^2 + \lambda (\norm{\mathbf{x}}^2 - R^2),
				\]
				where $\lambda\geq 0$ is the multiplier. The Lagrangian dual problem is 
				\begin{equation}\label{eq: Lagrange dual problem}
					\max_{\lambda\geq 0} \quad h(\lambda;\mathbf{A},\mathbf{b},R),
				\end{equation}
				where 
				\begin{equation}\label{eq:dual-subproblem}
					h(\lambda;\mathbf{A},\mathbf{b},R):=\inf_{\mathbf{x}\in\mathbb{R}^n}L(\mathbf{x},\lambda;\mathbf{A},\mathbf{b},R)=-\frac{1}{2}\mathbf{b}^{\tp} \mathbf{A} (\mathbf{A}^{\tp} \mathbf{A} + 2\lambda \mathbf{I}_n)^\dag \mathbf{A}^{\tp} \mathbf{b} + \frac{1}{2}\mathbf{b}^{{\tp}}\mathbf{b} - \lambda R^2.
				\end{equation}
				Here the equality follows from  \cite[Proposition 3.2 and Proposition 3.3]{Demmel1997appliednumericallnearalgebra} for the unconstrained linear least square problem.
				By the hypothesis on the SVD of $\mathbf A$, we get that
				\begin{equation}\label{eq: svd form of the dual function}
					\begin{aligned}
						h(\lambda;\mathbf{A},\mathbf{b},R)&=-\frac{1}{2}\mathbf{c}^{\tp} \begin{bmatrix}
							\mathbf{\Sigma}_0^2(\mathbf{\Sigma}_0^2+2\lambda \mathbf{I}_r)^{-1}	&0\\0&0
						\end{bmatrix}\mathbf{c}+\frac{1}{2}\mathbf{c}^{\tp} \mathbf{c}-\lambda R^2
						\\&=-\frac{1}{2}\sum_{k=1}^r \frac{\sigma_k^2c_k^2}{\sigma_k^2+2\lambda}+\frac{1}{2}\sum_{i=1}^m c_i^2-\lambda R^2.
					\end{aligned}
				\end{equation}
				Clearly, this implies that given $(\mathbf{A},\mathbf{b},R)$, $h(\lambda;\mathbf{A},\mathbf{b},R)$ is continuous with respect to $\lambda$ in the interval $[0,+\infty)$ and continuously differentiable in $(0,+\infty)$ with derivative 
				$\frac{\partial h}{\partial\lambda}(\lambda;\mathbf{A},\mathbf{b},R)=\sum_{k=1}^r \frac{\sigma_k^2c_k^2}{(\sigma_k^2+2\lambda)^2}-R^2$, which is decreasing.
				Note that $h(\lambda;\mathbf{A},\mathbf{b},R)$ is concave and $\lim_{\lambda\to +\infty}\frac{\partial h}{\partial\lambda}(\lambda;\mathbf{A},\mathbf{b},R)=-R^2<0$.
				Thus the optimal solution of \eqref{eq: Lagrange dual problem} $\lambda^{*}(\mathbf{A},\mathbf{b},R)$ is unique and given by 
				\begin{equation}\label{eq:opt-lambda}
					\lambda^{*}(\mathbf{A},\mathbf{b},R)=\argmax_{\lambda\geq 0}h(\lambda;\mathbf{A},\mathbf{b},R)=\begin{cases}
						0,& \Delta(\mathbf{A},\mathbf{b},R)\leq 0,\\
						\lambda_0(\mathbf{A},\mathbf{b},R),& \Delta(\mathbf{A},\mathbf{b},R) >0,
					\end{cases}
				\end{equation}
				where $\lambda_0(\mathbf{A},\mathbf{b},R)>0$ is the unique solution of~\eqref{eq: eq:opt-lambda 1}. 
				By strong duality $\min_{\mathbf x\in\mathbb R^n}\max_{\lambda\geq 0}L(\mathbf{x},\lambda;\mathbf{A},\mathbf{b},R)=\max_{\lambda\geq 0}\min_{\mathbf{x}\in \mathbb{R}^n}L(\mathbf{x},\lambda;\mathbf{A},\mathbf{b},R)$ (cf.\ \cite{rockafellar-1970a}), the characterization on the optimal value mapping \eqref{eq: svd form of the optimal value mapping} follows. 
				
				By \cite[Proposition 3.2 and Proposition 3.3]{Demmel1997appliednumericallnearalgebra}, the minimum of \eqref{eq:dual-subproblem} is attained at any  
				\begin{equation}\label{eq: first order condition of the dual problem}
					\mathbf{x}^{*}(\lambda;\mathbf{A},\mathbf{b},R)\in (\mathbf{A}^{\tp} \mathbf{A}+2\lambda \mathbf{I}_n)^\dag \mathbf{A}^{\tp} \mathbf{b}+\ker (\mathbf{A}^{\tp} \mathbf{A}+ 2\lambda \mathbf{I}_n).
				\end{equation}
				It is straight forward to see that the righthand side  of \eqref{eq: first order condition of the dual problem} has a nonempty intersection with $\mathcal{B}(R)$ for the optimal $\lambda\geq 0$ determined by \eqref{eq:opt-lambda}. Thus, by strong duality, 
				$\min_{\mathbf x\in\mathcal{B}(R)}\max_{\lambda\geq 0}L(\mathbf{x},\lambda;\mathbf{A},\mathbf{b},R)=\max_{\lambda\geq 0}\min_{\mathbf{x}\in \mathcal{B}(R)}L(\mathbf{x},\lambda;\mathbf{A},\mathbf{b},R)$ and the saddle points are given by 
				\[
				\{(\lambda^*,\mathbf x^*) \colon 
				\mathbf x^*\in (\mathbf{A}^{\tp} \mathbf{A}+2\lambda^*(\mathbf{A},\mathbf{b},R) \mathbf{I}_n)^\dag \mathbf{A}^{\tp} \mathbf{b}+\ker (\mathbf{A}^{\tp} \mathbf{A}+ 2\lambda^*(\mathbf{A},\mathbf{b},R) \mathbf{I}_n)\cap \mathcal{B}(R)\},
				\]
				with $\lambda^*=\lambda^*(\mathbf{A},\mathbf{b},R)$. By \cite[Theorem~36.3]{rockafellar-1970a} and \eqref{eq:opt-lambda}, the characterization on the solution mapping \eqref{eq: exact form for the solution mapping} follows. 
			\end{proof} 
			
			\begin{remark}
				\begin{enumerate}[label=(\arabic*)]
					\item 
					It follows from the proof that 
					\begin{equation*}
						\Delta(\mathbf{A},\mathbf{b},R)=\lim_{\lambda\to 0^{\dag}}\frac{\partial h}{\partial \lambda}(\lambda;\mathbf{A},\mathbf{b},R).
					\end{equation*}
					Thus, $h$ has continuous right derivative at $\lambda=0$. 
					\item An SVD independent formula for $\Delta(\mathbf{A},\mathbf{b},R)$ is 
					\begin{equation}\label{eq:criteria-delta without svd-0} 
						\Delta(\mathbf{A},\mathbf{b},R)=\mathbf{b}^{\tp} \mathbf{A}((\mathbf{A}^{\tp} \mathbf{A})^{\dag})^2\mathbf{A}^{\tp} \mathbf{b}-R^2=\norm{(\mathbf{A}^{\tp} \mathbf{A})^{\dag}\mathbf{A}^{\tp} \mathbf{b}}^2-R^2.
					\end{equation}
					There is a natural geometric meaning of \eqref{eq:criteria-delta without svd-0}. Consider the unconstrained linear LS problem 
					\begin{equation}\label{eq:unconstrained}
						\min_{\mathbf{x}\in\mathbb{R}^n}\norm{\mathbf{Ax}-\mathbf{b}}^2.
					\end{equation}
					By \cite[Proposition 3.2 and Proposition 3.3]{Demmel1997appliednumericallnearalgebra}, we know that the solution set of \eqref{eq:unconstrained} is the affine space $\mathcal{A}_{\mathbf{A},\mathbf{b}}:=\mathbf{A}^{\dag}\mathbf{b}+\ker(\mathbf{A})=(\mathbf{A}^{\tp}\mathbf{A})^\dag \mathbf{A}^{\tp}\mathbf{b}+\ker(\mathbf{A})$ and $\mathbf{x}^*:=(\mathbf{A}^{\tp}\mathbf{A})^\dag \mathbf{A}^{\tp}\mathbf{b}$ is the optimal solution of the smallest norm. It is clear that the norm $\norm{\mathbf{x}^*}$ is the distance from the origin to $\mathcal{A}_{\mathbf{A},\mathbf{b}}$. Thus for $R>0$, $\Delta(\mathbf{A},\mathbf{b},R)>0$ means that $\mathcal{A}_{\mathbf{A},\mathbf{b}}$ has an empty intersection with $\mathcal{B}(R)$, $\Delta(\mathbf{A},\mathbf{b},R)\leq 0$ implies that  $\mathcal{A}_{\mathbf{A},\mathbf{b}}$ has an nonempty intersection with $\mathcal{B}(R)$, while  $\Delta(\mathbf{A},\mathbf{b},R)=0$ indicates that $\mathcal{A}_{\mathbf{A},\mathbf{b}}$ is tangent to $\mathcal{B}(R)$.
					\item It is also clear that both 
					the optimal value mapping and the solution mapping can be concisely expressed respectively as
					\begin{equation}\label{eq: svd form of the optimal value mapping-compact}
						\mathcal{V}(\mathbf{A},\mathbf{b},R)=\sup_{\lambda\geq 0}h(\lambda;\mathbf{A},\mathbf{b},R)=h(\lambda^*;\mathbf{A},\mathbf{b},R), 
					\end{equation}
					and 
					\begin{equation}\label{eq: exact form for the solution mapping-compact}
						\begin{aligned}
							\mathcal{S}(\mathbf{A},\mathbf{b},R)=
							(\mathbf{A}^{{\tp}} \mathbf{A}+2\lambda^*(\mathbf{A},\mathbf{b},R) \mathbf{I}_n)^{\dag}\mathbf{A}^{\tp} \mathbf{b}+\ker(\mathbf{A}))\cap \mathcal{B}(R),
						\end{aligned}
					\end{equation} 
					where
					$\lambda^*=\lambda^*(\mathbf{A},\mathbf{b},R)\in\mathbb R_{+}\cup\{+\infty\}$ is the unique solution of the first order condition 
					\begin{equation}\label{eq: eq:opt-lambda 1-first}
						\frac{\partial h}{\partial\lambda}(\lambda^*;\mathbf{A},\mathbf{b},R)\in-\partial \delta_{\mathbb R_+}(\lambda^*).
					\end{equation} 
				\end{enumerate}
			\end{remark}

	\end{document}